\documentclass{amsart}
\usepackage{etex}

\usepackage[usenames,dvipsnames]{xcolor}
\usepackage[bookmarksopen,bookmarksdepth=2]{hyperref}
\hypersetup{colorlinks=true,citecolor=NavyBlue,linkcolor=BrickRed,urlcolor=Green} 

\usepackage{caption}
\usepackage{fancyhdr}

\newcommand{\Pone}{\mathbb{P}^1}

\RequirePackage{amsmath}
\RequirePackage{amssymb}
\RequirePackage{amsxtra}
\RequirePackage{amsfonts}
\RequirePackage{latexsym}
\RequirePackage{euscript}
\RequirePackage{amscd}
\RequirePackage{amsthm}
\RequirePackage{xypic}
\usepackage{stmaryrd}
\usepackage{mathtools}

\usepackage{mathrsfs}
\xyoption{all}

\usepackage{enumitem}

\def\free{\mathrm{free}}

\usepackage{tikz-cd}
\usetikzlibrary{shadings}
\newcommand{\xra}{\xrightarrow}
\newcommand{\thra}{\twoheadrightarrow}
\newcommand{\hra}{\hookrightarrow}
\newcommand{\sse}{\subseteq}
\newcommand{\bA}{\mathbf{A}}
\newcommand{\bG}{\mathbf{G}}
\newcommand{\bP}{\mathbf{P}}

\newcommand{\Sch}{\mathsf{Sch}}
\newcommand{\Set}{\mathsf{Set}}

\pgfdeclarefunctionalshading{Hermite-Gaussian modes}{\pgfpoint{-25bp}{-25bp}}{\pgfpoint{25bp}{25bp}}{}{
    10 atan sin 1000 mul cos 1 add
    exch
    10 atan sin 1000 mul cos 1 add
    mul 4 div
    dup dup
}

\renewcommand{\mathbb}{\mathbf}

\newcommand{\Acirc}{A^\circ}

\newcommand{\AKnrhat}{\widehat{\A}_K^{\nr}}
\newcommand{\AKnrhatA}{\widehat{\A}_{K,A}^{\nr}}

\newcommand{\lambdau}{\underline{\lambda}}

\newcommand{\piflat}{\pi^\flat}

\newcommand{\cyc}{\operatorname{cyc}}

\newcommand{\Kcyc}{K_{\cyc}}

\newcommand{\Aplus}{\A^+}
\newcommand{\AAA}{\A}

\newcommand{\negligible}{\mathrm{small}}
\newcommand{\OEA}{\cO_{\mathcal{E},A}}

\newcommand{\OEB}{\cO_{\mathcal{E},B}}

\newcommand{\cXbar}{\overline{\cX}}

\newcommand{\Frac}{\operatorname{Frac}}

\newcommand{\gMt}{\gM^{\inf}}

\usepackage{dsfont}

\let\emptyset\varnothing

\def\A{\mathbb A}

\def\C{\mathbb C}
\def\F{\mathbb F}
\def\Khat{\widehat{K}}

\def\Lhat{\widehat{L}}

\def\Q{\mathbb{Q}}

\def\Z{\mathbb{Z}}
\def\Fbar{\overline{\F}}

\def\Qptimes{\Q_p^{\times}}
\def\Qbar{\overline{\Q}}

\def\Zbar{\overline{\Z}}

\def\m{\mathfrak m}

\newcommand{\wh}{\widehat}

\def\chibar{\overline{\chi}}

\def\id{\mathrm{id}}

\def\red{\mathrm{red}}

\def\ab{\mathrm{ab}}
\def\nr{\mathrm{ur}}
\def\fl{\mathrm{fl}}

\def\ss{\mathrm{ss}}

\def\SL{\mathrm{SL}}

\def\ur{\mathrm{ur}}

\def\GL{\operatorname{GL}}

\def\Gal{\mathrm{Gal}}

\def\Sym{\mathrm{Sym}}
\def\Ext{\mathrm{Ext}}

\def\Hom{\mathop{\mathrm{Hom}}\nolimits}
\def\RHom{\mathop{\mathrm{RHom}}\nolimits}

\def\Spec{\mathop{\mathrm{Spec}}\nolimits}
\def\Spf{\mathop{\mathrm{Spf}}\nolimits}
\def\Frob{\mathop{\mathrm{Frob}}\nolimits}

\def\Ind{\mathop{\mathrm{Ind}}\nolimits}
\def\cInd{\mathop{c\mathrm{-Ind}}\nolimits}
\def\Fil{\mathop{\mathrm{Fil}}\nolimits}

\def\rhobar{\overline{\rho}}

\def\cotimes{\operatorname{\widehat{\otimes}}}

\def\crys{\mathrm{crys}}

\def\dR{\mathrm{dR}}

\def\WD{\mathrm{WD}}

\def\m{\mathfrak{m}}

\def\iso{\buildrel \sim \over \longrightarrow}

\swapnumbers
\newcommand{\ra}{\rightarrow}

\newcommand{\into}{\hookrightarrow}

\newcommand{\To}{\longrightarrow}
\newcommand{\isoto}{\stackrel{\sim}{\To}}
\newcommand{\textB}{\mathrm{B}}

 \newcommand{\BdR}{\textB_{\dR}}

\newlength{\ownl}

\newcommand{\ad}{{\operatorname{ad}\,}}

\newcommand{\es}{\emptyset}

\newcommand{\rec}{{\operatorname{rec}}}

\newcommand{\tr}{{\operatorname{tr}\,}}

\newcommand{\Gm}{{\mathbb{G}_m}}

\newcommand{\disc}{{\operatorname{disc}}}

\newcommand{\semis}{{\operatorname{ss}}}

\newcommand{\univ}{{\operatorname{univ}}}
\newcommand{\cA}{\mathcal{A}}

\newcommand{\cC}{\mathcal{C}}

\newcommand{\cE}{\mathcal{E}}
\newcommand{\cF}{\mathcal{F}}

\newcommand{\cI}{\mathcal{I}}

\newcommand{\cM}{\mathcal{M}}

\newcommand{\cO}{\mathcal{O}}

\renewcommand{\cR}{\mathcal{R}}

\newcommand{\cV}{\mathcal{V}}
\newcommand{\cW}{\mathcal{W}}
\newcommand{\cWW}{\mathcal{W}}
\newcommand{\cX}{\mathcal{X}}
\newcommand{\cY}{\mathcal{Y}}
\newcommand{\cZ}{\mathcal{Z}}

\newcommand{\gM}{{\mathfrak{M}}}

\newcommand{\gS}{{\mathfrak{S}}}

\newcommand{\barr}[1]{\overline{#1}}

\newcommand{\tC}{\widetilde{{C}}}

\newcommand{\tZ}{\widetilde{{Z}}}

\newcommand{\tc}{\widetilde{{c}}}

\newcommand{\tx}{\widetilde{{x}}}

\newcommand{\tz}{\widetilde{{z}}}

\newcommand{\alphabar   }{\overline{\alpha  }}

\newcommand{\varepsilonbar  }{\overline{\varepsilon}}

 \newcommand{\sigmabar   }{\overline{\sigma}}

 \newcommand{\tGamma     }{\widetilde{\Gamma}}
 \newcommand{\Gammat     }{\widetilde{\Gamma}}

\def\RCS$#1: #2 ${\expandafter\def\csname RCS#1\endcsname{#2}}
\RCS$Revision: 85 $
\RCS$Date: 2011-12-16 18:54:17 -0600 (Fri, 16 Dec 2011) $

\newcommand{\sq}{\square}

\DeclareMathOperator{\FL}{FL}

\newcommand{\set}[1]{\left\{#1\right\}}

 \newcommand{\Qp}{{\Q_p}}

\newcommand{\Zp}{{\Z_p}}
 
\newcommand{\Qpbar}{{\overline{\Q}_p}}

\newcommand{\Zpbartimes}{{\overline{\Z}_p^\times}}
\newcommand{\cbar}{\overline{c}}
\newcommand{\Zpbar}{{\overline{\Z}_p}}

\newcommand{\Fpbar}{{\overline{\F}_p}}
\newcommand{\Fpbartimes}{{\overline{\F}_p^\times}}

\newcommand{\Fp}{{\F_p}}
\newcommand{\Flbar}{\overline{\F}_l}

\newif\iffinalrun
\iffinalrun
\else
\fi

\iffinalrun
  \newcommand{\need}[1]{}
  \newcommand{\mar}[1]{}
\else
  \newcommand{\need}[1]{{\tiny *** #1}}
  \newcommand{\mar}[1]{\marginpar{\raggedright\tiny fixme #1}}
\fi

\newcommand{\Ainf}{\mathbf{A}_{\inf}}
\newcommand{\AAinf}[1]{\mathbf{A}_{\inf,#1}}
\newtheorem{theorem}[subsubsection]{Theorem}
\newtheorem{expectedtheorem}[subsubsection]{Expected Theorem}
\newtheorem{thm}[subsubsection]{Theorem}
\newtheorem{lemma}[subsubsection]{Lemma}
\newtheorem{lem}[subsubsection]{Lemma}

\newtheorem{cor}[subsubsection]{Corollary}
\newtheorem{conj}[subsubsection]{Conjecture}
\newtheorem{roughconj}[subsubsection]{Rough Conjecture}
\newtheorem{prop}[subsubsection]{Proposition}

\theoremstyle{definition}
\newtheorem{df}[subsubsection]{Definition}
\newtheorem{defn}[subsubsection]{Definition}
\newtheorem{definition}[subsubsection]{Definition}
\newtheorem{situation}[subsubsection]{Situation}

\theoremstyle{remark}
\newtheorem{remark}[subsubsection]{Remark}
\newtheorem{rem}[subsubsection]{Remark}

\newtheorem{example}[subsubsection]{Example}

\def\numequation{\addtocounter{subsubsection}{1}\begin{equation}}
\def\nummultline{\addtocounter{subsubsection}{1}\begin{multline}}
\def\anumequation{\addtocounter{subsection}{1}\begin{equation}}
\title{Moduli stacks of $(\varphi,\Gamma)$-modules: a survey}

\author[M. Emerton]{Matthew Emerton}\email{emerton@math.uchicago.edu}
\address{Department of Mathematics, University of Chicago,
5734 S.\ University Ave., Chicago, IL 60637}

\author[T. Gee]{Toby Gee} \email{toby.gee@imperial.ac.uk} \address{Department of
  Mathematics, Imperial College London,
  London SW7 2AZ, UK}

  \thanks{The first author was supported in part by the
   NSF grants DMS-1601871, DMS-1902307, and  DMS-2201242. 
  The second author was 
  supported in part by a Leverhulme Prize, EPSRC grant EP/L025485/1, Marie Curie Career
  Integration Grant 303605, and by
  ERC Starting Grant 306326.}
\begin{document}
\maketitle
\setcounter{tocdepth}{1}
\tableofcontents

This survey article is based on 10 lectures that we gave at the
September 2019 Hausdorff School in Bonn on ``the Emerton--Gee stack
and related topics'', and the sections correspond to the original
lectures. We have retained the somewhat informal style of some of
these lectures, but have filled them out with further details. We hope
that the first 9 lectures can serve as an extended introduction
to~\cite{emertongeepicture}; the reader might first wish to read the
actual introduction to~\cite{emertongeepicture}, which in particular
motivates the results described here. We also refer to
~\cite{emertongeepicture} for any definitions not given in these
lectures. The 10th lecture is a speculative look ahead to connections
between our stacks and the $p$-adic local Langlands correspondence,
and in particular explains some of our work in progress with Andrea
Dotto~\cite{DottoEG}.

We would like to thank the organisers (Johannes Ansch\"utz,
Arthur-C\'esar Le Bras, and Andreas Mihatsch) and participants of the
Hausdorff school for the many helpful questions and corrections
resulting from our lectures. We would especially like to thank Ashwin
Iyengar for transcribing M.E.'s lectures at the Hausdorff school, and
for allowing us to use his LaTeX files as a basis for some sections of
this article. Some material in Lecture~\ref{sec: Extensions and
  Crystalline Lifts} is based on the notes for a series of lectures
that TG gave at the ENS Lyon in June 2018, and we are grateful to the
organisers (Laurent Berger, Gabriel Dospinescu, Philippe Gille,
Wieslawa Niziol, Vincent Pilloni, Sandra Rozensztajn, Amaury
Thuillier) of the thematic trimester on ``Algebraic Groups and
Geometrization of the Langlands Program'' for the invitation to speak
there, and the encouragement to write up notes for the lectures. We
would also like to thank Matteo Tamiozzo for some comments on this material.

TG would like to thank the organisers (Najmuddin Fakhruddin, Eknath
Ghate, Arvind Nair, C.\ S.\ Rajan, and Sandeep Varma) of the
International Colloquium on Arithmetic Geometry at TIFR in January
2020, for the invitation to speak at the conference, for their
hospitality, and for the invitation to submit this article to the
proceedings, without which it is unlikely that we would have completed it.
Both authors thank the anonymous referee for their helpful comments and
corrections.

% \mar{TG: also thank TIFR!! ME: And the live-TeXers of
% my lectures.}

% All references to~\cite{emertongeepicture} are to the second arXiv
% version available at \url{https://arxiv.org/abs/1908.07185v2}, and may not agree with the numbering in any future
% versions.%\mar{TG: need to check that all citations agree with this version}

%\mar{ME:Change section command so the sections read ``Lecture 1.  Introduction'', etc.}

\section{Introduction}
\label{sec:intro}

 As the title suggests, the paper~\cite{emertongeepicture} is about moduli stacks of $(\varphi,\Gamma)$-modules. In
 fact, we're really interested in the moduli of (local) ($p$-adic and mod~$p$)
 Galois representations, so we'll first give some background on this
 problem, and the complications that arise. The theory of $(\varphi,\Gamma)$-modules
and the construction of our moduli stacks will be the topic of subsequent lectures. 
% and then describe how we get to our construction.
%\mar{ME: Maybe change the last clause, as the scope of the various lectures
%becomes clearer TG: unsure whether this is supposed to be about this
%lecture or the whole thing, but presumably either way it should be
%easy to do this now!}

\subsection{Galois representations} Let~$p$ be prime, and let $K/\Qp$ be a finite extension. We will be
interested in $p$-adic and mod $p$ representations of the absolute
Galois group $G_K := \Gal(\barr{K}/K)$. In the mod $p$  case,
this  means that  we will fix some $d  \geq 1$ and study continuous homomorphisms
    \[ G_K := \Gal(\barr{K}/K) \to \GL_d(\Fpbar). \]
Note these factor through $\GL_d(\F_q)$ for some $q$. Similarly we could look at continuous homomorphisms
    \[ G_K \to\GL_d(\Qpbar); \]
 these will factor through $\GL_d(E)$ for some finite $E/\Qp$, and using compactness, they can be  conjugated so as to %actually
factor through $\GL_d(\cO_E)$.

We want to arrange these representations into algebraic families (this
is more-or-less what it  means to construct a moduli space of $d$-dimensional
representations of $G_K$). This was done by Carl Wang-Erickson
in~\cite{MR3831282}, 
  but the topological nature of $G_K$ means that these families are
  ``smaller'' than one might hope. (We discuss the relationship
  between our stacks and Wang-Erickson's in Lecture~\ref{subsec: closed points GK reps}.)  We will elaborate  on this, 
before turning to our ``larger'' families (which are defined by passing from
$G_K$-representations  to $(\varphi,\Gamma)$-modules).

\begin{example}%\mar{TG: figure out numbering}
The first example of an ``algebraic family'' we might consider   is a family of unramified characters
	\[ \ur_x: G_K \to G_K/I_K \cong \Frob^{\wh{\Z}} \to \Fpbartimes \]
which sends Frobenius to $x \in \Fpbartimes$ (note this is well-defined because any element of $\Fpbartimes$ lives in $\F_q^\times$ for some $q$, so we define the unramified character by $\wh{\Z} \thra \Z/(q-1)\Z \ra %\xra \sim
\F_q^\times$). But this doesn't quite come together as an algebraic family: if we wanted to make this algebraic, we would let $\F_p[x,x^{-1}]^\times$ be our ring of coefficients and take specializations to $\Fpbartimes$. We can define a homomorphism $\Frob^{\Z} \to \F_p[x,x^{-1}]^\times$ taking $\Frob \mapsto x$, but this doesn't extend continuously to $\Frob^{\wh{\Z}}$. So this gives the first basic obstruction to constructing the right moduli space.
\end{example}

As we will explain in the rest of this section, to get around the
problem in the example, you could consider the Weil group, which is a
``decompletion'' of $\Frob^{\wh{\Z}}$; but ultimately, this only works well when
either  $d = 1$ or we work with coefficient rings in which~$p$ is
invertible, so when we work with $p$-adic coefficients in the case $d > 1$ we have to use
integral $p$-adic Hodge theory
(and specifically the theory of $(\varphi,\Gamma)$-modules), as in~\cite{emertongeepicture}.

\subsection{Weil--Deligne Case}%\mar{TG: reference Helm et al and Xinwen}
We begin by recalling the ``Weil--Deligne formalism'',
which is a standard and effective approach to dealing with topological issues in
the theory of $\ell$-adic Galois representations  of $G_K$,
when $\ell \neq p$.   As we will see, it is of more dubious merit in
the $p$-adic context.   The ideas we explain are dealt with in much more detail
(and for representations into more general groups than $\GL_d$)
in \cite{dat2020moduli} and \cite{zhu2020coherent}.

Note the local Galois group $G_K$ has a tame quotient and pro-$p$ wild inertia subgroup:
    \[ P_K \hra G_K \thra G_K^{\text{tame}} = (\Z \ltimes \Z[1/q])^\wedge \]
where $1\in\Z$ acts by multiplication by $q$ on~$\Z[1/q]$, and we take
the profinite completion.
%\mar{ME: Maybe mention this depends on choices, and cite Xinwen and Peter for more canonical formulations when $\ell \neq p$.}
We can then define the ``Weil--Deligne group'' $\WD_K$ as usual, which fits into the following diagram:
\begin{center}\begin{tikzcd}
	0 \rar & P_K \rar \dar & \WD_K \rar \dar & \Z \ltimes \Z[1/q] \rar \dar & 0 \\
	0 \rar & P_K \rar & G_K \rar & (\Z \ltimes \Z[1/q])^\wedge \rar & 0
\end{tikzcd}\end{center}
%\mar{TG: not sure what this individual s.e.s. is doing?}    \[ P_K \hra W_K \thra \Z \ltimes \Z[1/q] \]
Then for any open and finite index subgroup $Q \leq P_K$, $\WD_K/Q$ is a finitely presented discrete group, and
    \[ \WD_K = \varprojlim_Q \WD_K/Q. \]
So we have essentially ``decompleted and discretized'' the tame part, but we are remembering the topology on the wild part of the Galois group. Furthermore, $G_K/Q = (\WD_K/Q)^\wedge$, and since representations of $\WD_K/Q$ with values in finite rings do extend over the profinite completion, this seems like a good first approach to defining a moduli space.\footnote{One technical point is that this definition of $\WD_K$,
and thus the consequent definition of a moduli space, depend on choices: namely,
a choice of a lift of Frobenius and also a choice of a generator of
tame inertia.  We don't dwell on this issue here, but note that
Scholze and Zhu have explained how to make the same construction
more canonically (involving no auxiliary choices) in the  case
of $\ell$-adic representations
when~ $\ell~\neq~p$~\cite[Props.~3.1.10, 3.1.11]{zhu2020coherent}.}

\begin{definition}
Let $V_Q \to \Spec \Z$ denote the scheme parameterizing representations $\rho: \WD_K/Q \to \GL_d$.
\end{definition}

Note that $\WD_K/Q$ is a finitely presented group, so it's easy to find a
finite presentation for $V_Q$ as an affine scheme. If $Q' \sse Q$,
then there is a natural closed immersion $V_Q \hra V_{Q'}$ (this is easy to check using the moduli description), so we can study the Ind-scheme
    \[ V := \varinjlim_Q V_Q. \]
By construction, any continuous $\barr\rho: G_K \to \GL_d(\F)$ (for $\F$ any finite field) factors through a finite quotient of $G_K$, and thus a finite quotient of $\WD_K$, and thus factors through one of the $\WD_K/Q$, so gives an $\F$-valued point of $V$. Conversely a representation $\WD_K/Q \to \GL_d(\F)$ extends continuously to $G_K/Q$ and thus to $G_K$. Therefore
	\[ V(\F) = \set{\text{continuous } \barr\rho: G_K \to \GL_d(\F)}. \]
If $R_{\barr\rho}^\sq$ denotes the universal lifting ring of
$\barr\rho$ one can check that there is a natural map
    \[ \Spf R_{\barr\rho}^\sq \to V. \]
In fact, this map is versal to $V$ at $\barr\rho$,
and so $V$ is some sort of geometric object that joins together
all the formal deformation theory of the various $\barr\rho$.
The next question,  then, is exactly what sort of geometric object is $V$?
% So in some sense,
%the formal completion of $V$ at each $\F$-valued point gives you the
%thickenings of that point.\mar{TG: what does this mean?} The only problem is that
While the $V_Q$ are finitely presented affine schemes,
their  colimit $V$ is more infinitary,
and we need to study what sort of geometry we can get when we take this
sort  of colimit. % defining $V$.

So what does $V$ look like? It could involve taking a countable disjoint union of varieties: this is still geometric, and it's locally of finite presentation with countably many components (recall each of the $V_Q$ are finitely presented).
A slightly more complicated possibility is taking a not-necessarily disjoint union
of varieties, each of which is an irreducible (though not necessarily
connected) component of their  union.
%Or it could involve taking a variety, then intersecting it with another one, and another one, ad infinitum.
This isn't of finite presentation, but it is still a scheme, and it's still quite geometric.

%But it could fail to be a scheme! For instance, take
%$\bA^1_{\Fpbar}$ and glue a copy of $\bA^1_{\Fpbar}$ to each
%closed point.\mar{TG: could reference the examples in proper.tex
%  around here? ME: Yes, perhaps even include a couple of pictures?}

 \[\begin{tikzpicture}[node distance = 2cm, auto, scale=1.5, transform shape]
%\draw (0,0) -- (7,0);
\draw (1,-1) -- (1,1);
\draw (2,-1) -- (2,1);
\draw (3,-1) -- (3,1);
\draw (4,-1) -- (4,1);
\draw (5,-1) -- (5,1);
\draw [dotted, very thick] (6,0.2) -- (8,0.2);
\end{tikzpicture}\]
\begingroup
\captionof{figure}{An infinite union of connected components}
\endgroup
\medskip

   \begin{tikzpicture}[node distance = 2cm, auto, scale=1.5, transform shape]
%\draw (0,0) -- (7,0);
\draw (0,-1) -- (1,1);
\draw (1.5,-1) -- (0.5,1);
\draw (1,-1) -- (2,1);
\draw (2.5,-1) -- (1.5,1);
\draw (2,-1) -- (3,1);
\draw (3.5,-1) -- (2.5,1);
\draw (3,-1) -- (4,1);
\draw (4.5,-1) -- (3.5,1);
\draw (4,-1) -- (5,1);
\draw (5.5,-1) -- (4.5,1);
\draw (5,-1) -- (6,1);
\draw [dotted, very thick] (6,0.2) -- (8,0.2);
\end{tikzpicture}
\begingroup
\captionof{figure}{An infinite union of irreducible components}
\endgroup

But it could fail to be a scheme! For instance, take
$\bA^1_{\Fpbar}$ and glue a copy of $\bA^1_{\Fpbar}$ to each
closed point.% \mar{TG: could reference the examples in proper.tex
  % around here? ME: Yes, perhaps even include a couple of pictures?}

 \[\begin{tikzpicture}[node distance = 2cm, auto, scale=1.5, transform shape]
\draw (0,0) -- (7,0);
\draw (1,-1) -- (1,1);
\draw (2,-1) -- (2,1);
\draw (3,-1) -- (3,1);
\draw (4,-1) -- (4,1);
\draw (5,-1) -- (5,1);
\draw [dotted, very thick] (6,0.2) -- (8,0.8);
\end{tikzpicture}\]
\begingroup
\captionof{figure}{An Ind-scheme which is not a scheme}\label{fig: herringbone}
\endgroup
We refer the reader to \cite[\S 4.3]{EGstacktheoreticimages} for many more examples
of Ind-schemes that aren't schemes.

So if $V$ is of one of the first two types, and is still a scheme, then we'd be in good shape because we can do geometry, but otherwise we haven't done much to simplify our original problem. In fact, it really depends on the characteristic of our coefficient field: in characteristic $p$ we will see that things behave badly, while in characteristic $0$, we will see that we get a well-behaved formal scheme.

\subsection{\texorpdfstring{Analysis of $V$ when $d = 1$ and $K =
    \Qp$}{Analysis of V when d = 1 and K = Qp}}\label{subsec:
  discussion of rank 1 case}

For simplicity, assume $p > 2$. In this case, we use local class field theory to replace $G_{\Qp}$ with
    \[ G_{\Qp}^{\ab} = \wh{\Qptimes} = \Z_p^\times \times p^{\wh{\Z}}. \]
Taking the Weil--Deligne subgroup gives
    \[ \WD_{\Qp}^{\ab} = \Qptimes = \Z_p^\times \times p^{\Z}. \]
In this case, the pro-$p$ wild ramification part is a copy of $\Z_p$,
embedded in $\WD_{\Qp}^{\ab}$ via $1+p\Z_p \hra \Z_p^\times$. So we
choose a sequence $Q=Q_n$ so that
    \[ \WD_{\Qp}^{\ab}/Q_n = (\Z/p^{n+1})^\times \times p^\Z \cong \Z/{p^n}\Z \times \Z/{(p-1)}\Z \times p^\Z. \]
Thus
%\begin{multline*}
%    \[
\[V_{Q_n} %=  \Spec (\Z[x]/(x^{p^n}-1) \otimes \Z[x]/(x^{p-1}-1) \otimes \Z[x,x^{-1}])
% \\ 
= \Spec \Z[x]/(x^{p^n}-1) \times \Spec \Z[x]/(x^{p-1}-1) \times \bG_{m,\Z}. % \]
\]%\end{multline*}

\begin{enumerate}
	\item To see what happens away from $p$, we can invert $p$ on this scheme and see what we get:
\begin{multline*}
%		\[
V_{\Z[1/p]} = \varinjlim_n V_{Q_n,\Z[1/p]} \\  =  (\varinjlim_n \Spec \Z[1/p][x]/(x^{p^n}-1)) \times \Spec \Z[1/p][x]/(x^{p-1}-1) \times \bG_{m,\Z[1/p]}. % \]
\end{multline*}
	If we now base change to $k = \barr{\Q}$ or $k = {\Flbar}$ for $\ell\neq p$, then $(x^{p^n}-1)$ is separable, so we are just adding more and more closed points as we go further along the directed system, and so we basically get infinitely many copies of $\Spec\Z[x]/(x^{p-1}-1) \times \bG_m$, indexed by $p$th roots of unity in $k$.

	\item But instead if we base change to $\Z_p$, then the closed immersions $V_{Q_n,\Z_p} \hra V_{Q_{n+1},\Z_p}$ are actually nilpotent thickenings, and in fact we end up with
	    \[ V_{\Z_p} = \Spf \Z_p[[T]] \times \Spec \Z[x]/(x^{p-1}-1) \times \bG_m. \]
	% So in this case we get 
       Now $\Spf \Z_p[[T]]$ is a nice Noetherian formal scheme; be
       warned that this is specific to dimension 1. Indeed already in dimension 2 the situation becomes much more complicated, as we will see  shortly. %illustrate now.
\end{enumerate}
	
Before  turning to the case of dimensions $2$ and higher, we give some
more examples  of Ind-schemes.
	First, here's a formal scheme that isn't Noetherian: Take
        $\bA^1_{\Fpbar}$ and add a %an infinitesimal
extra formal direction
        at each closed point on the line (one can do this finitely
        many times and then take a filtered colimit over the finite
        steps). % \mar{TG: reorder this a bit to be slightly less
          % informal?}
        This is  an affine formal scheme $\Spf A$ for~$A$
        some topological ring, but~$A$ is not Noetherian. More precisely, we let $\Fpbar = \set{a_0,a_1,\dots,}$ and set $A:=\varprojlim_j A_j$ where
		\[ A_j = \Fpbar[x] \times_{\Fpbar} \Fpbar[[x-a_0]] \times_{\Fpbar} \cdots \times_{\Fpbar} \Fpbar[[x-a_j]]. \]
	
Now here's something even worse. Consider the subsheaf
of~$\bA^2_{\Fpbar}$ given by taking  a ``horizontal''
~$\bA^1_{\Fpbar}$  and adding a ``vertical'' $\bA^1_{\Fpbar}$ through
each closed point, as in Figure~\ref{fig:
          herringbone}, and then formally
%infinitesimally \mar{TG: referee says: `` Should the infinitesimal thickenings be formal thickenings?''}
thickening each of the
vertical lines (in $\bA^2_{\Fpbar}$) in the horizontal
        direction, as in Figure~\ref{fig: thickened herringbone}.  \[\begin{tikzpicture}[node distance = 2cm, auto, scale=1.5, transform shape]
            \draw (0,0) -- (7,0);
 (1,-1) --
           (1,1);
 \shade[left color=white,right color=white,middle color=gray] (0.9,-1)
 rectangle (1.1,1);
  \shade[left color=white,right color=white,middle color=gray]
  (1.9,-1) rectangle (2.1,1);
   \shade[left color=white,right color=white,middle color=gray]
   (2.9,-1) rectangle (3.1,1);
    \shade[left color=white,right color=white,middle color=gray]
    (3.9,-1) rectangle (4.1,1);
     \shade[left color=white,right color=white,middle color=gray] (4.9,-1) rectangle (5.1,1);           \draw [dotted, very thick] (6,0.2) -- (8,0.8);
\end{tikzpicture}\]
\begingroup

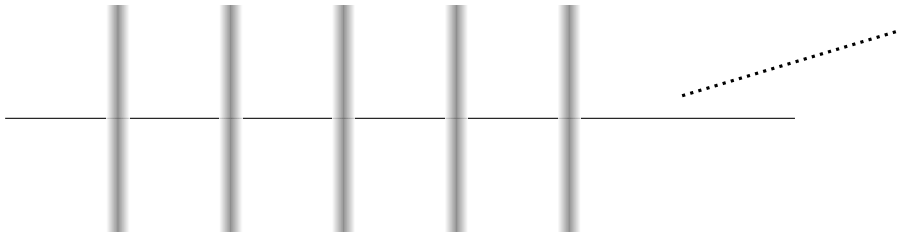
\captionof{figure}{An Ind-scheme with thickened vertical lines}\label{fig:
  thickened herringbone}\endgroup

\medskip

Then
        this is not even a formal scheme, and it's actually hard to
        tell this apart from $\bA^2$, because it has the same closed
        points and versal rings.%  (remember we're working over
        % $\Fpbar$).\mar{TG: explain why it's hard to tell apart?}
	
	We could even take the above construction %(call it $V_\star$)
        and then delete the point at the origin. Call this $V_\star$;
        we will return to it shortly.
        % and keep this in your
        % head for a moment\dots % \mar{TG: I suspect that this might be
        %   better done in the opposite order, i.e.\ mention the FL case
        % first and then introduce this example?}
%\end{enumerate}

\subsection{Analysis of $V$ in general}\label{subsec: FL vs WD}

For general $d$ and general $K$, $V_Q/\Z[1/p]$ is reduced,
 local complete intersection (so Cohen--Macaulay), and flat/$\Z[1/p]$ of
relative dimension $d^2$
\cite{dat2020moduli, zhu2020coherent},
%(David Helm)\mar{TG: could actually reference their paper?}
and
    \[ V_Q/\Z[1/p] \hra V_{Q'}/\Z[1/p] \]
is just adding connected components. Note in particular that the local deformation rings of mod $\ell$ residual representations are just computed using this variety by looking at the complete local rings at the stalks of the corresponding points.

Now let's switch back to the characteristic $p$ setting. Let $K =
\Qp$, $d = 2$, and assume that $p > 2$.  Let's just study the
behaviour of~$V$ in characteristic~$p$; let's also fix the determinant
of our two-dimensional Galois representations to be~$\omega^i$,
where $\omega$  denotes the mod $p$ cyclotomic  character,
and we choose $1 \leq  i  \leq  p-2$.
We write $V^{\det = \omega^i}_{/\F_p}$ for the resulting Ind-scheme.

Let's  begin by considering  the $\Fpbar$-valued  points
of~$V^{\det=\omega^i}_{/\F_p}$;
these correspond to $\rhobar: \WD_{\Q_p} \to  \GL_2(\Fbar_p)$ of determinant~$\omega^i$.
There is one irreducible such~$\rhobar$ up to isomorphism, 
namely $\Ind_{\Q_{p^2}}^{\Qp} \omega_2^i$.   
Because of the choice of framing, this will not correspond to a literal single
closed point of~$V^{\det=\omega^i}_{/\F_p}$,  but rather to an~$\SL_2$ orbit.\footnote{``Fixing  the determinant'' means  choosing an  isomorphism of $\WD_{\Q_p}$-representations
$\wedge^2 \rho \iso \omega^i$, and so the ``change  of frame'' group is
reduced from $\GL_2$ to $\SL_2$.}
This orbit {\em will} be (the underlying reduced subscheme  of)
a connected  component of~$V^{\det=\omega^i}_{/\F_p}$.

Any reducible $\rhobar$ will  be  an  extension  of mod $p$ characters
of $\WD_{\Q_p}$, 
and any such character is an unramified twist of a  power of~$\omega$.
Any family of such extensions will induce a corresponding family
of semi-simplifications (more  precisely, of associated pseudo-characters),
and the powers of the  cyclotomic character  that  appear will be  locally
constant,  and thus  constant on a  given connected  component of
the family.   We will restrict attention to those  $\rhobar$
for which $\rhobar^{\semis} = \ur_{\lambda^{-1}} \omega^i \oplus  \ur_{\lambda},$
for some  $\lambda  \in \Fbar_p^{\times}$,
and in fact to those $\rhobar$ which are an  extension of $\ur_{\lambda}$
by  $\ur_{\lambda^{-1}} \omega^i$ (rather  than  an  extension in the
opposite  order). Any such non-split  extension is unique up
to  isomorphism (equivalently,
the corresponding  $\Ext^1$ is one-dimensional) unless $i  = 1$ and
$\lambda = \pm 1$; in this  latter case there is a one-dimensional
space of {\em peu ramifi\'ee} extensions, and it is the $\rhobar$
classified by this space that we will consider.

There is actually a  $p$-adic Hodge-theoretic framework that describes exactly
these~$\rhobar$  that  we  are singling out (the irreducible~$\rhobar,$
and  the particular reducible~$\rhobar$ that  we have described); namely,
they are {\em Fontaine--Laffaille} with  Hodge--Tate weights  $\set{0,i}$
(and with  determinant $\omega^i$).
And what we will now describe is the Ind-scheme
$V^{\FL,\det=\omega^i}_{/\F_p}$
which classifies families of $\WD_{\Q_p}$-representations
over finite type $\F_p$-algebras~$A$ that  are two-dimensional, have determinant fixed 
to be~$\omega^i$, and which are Fontaine--Laffaille  with Hodge weights $\set{0,i}$
when pushed forward to  any Artinian  quotient of~$A$. 

Actually, to  avoid grappling with the details of framings,
we will describe  the  quotient stack\footnote{See the  next lecture 
for more about stacks!}~$[V^{\FL,\det=\omega^i}_{/\F_p}/\SL_2].$
To state the answer, recall the  subsheaf $V_\star$ of  $\A^2_{/\F_p}$
described above.  There is an action of $\Gm$ on this  substack
obtained  by  restricting the  following  action of $\Gm$  on $\A^2$:
$$ a  \cdot (x,y) = (x, a^2 y).$$

\begin{prop}
\label{rotten  herring}
There is  an isomorphism of stacks
$$[V^{\FL,\det=\omega^i}_{/\F_p}/\SL_2] \iso [V_{\star}/\Gm].$$
\end{prop}

Before sketching the  proof,
we  note
that $V_\star$ is rather nasty.  It is Zariski dense in $\A^2\setminus  \set{0}$
--- indeed, it has the same  set of $\Fbar_p$-points, and  the same
versal  rings at these  points, as $\A^2\setminus  \set{0}$ --- but it is disconnected, 
and it looks more like a formal scheme than an honest  scheme like $\A^2\setminus 
\set{0}.$
%\mar{TG: a bit confused about just how disconnected?}
%but looks a lot like
%$\bA^2 \setminus \set{0}$, so
Given this, it makes sense to ask: 
can one study some related but  different moduli problem so as to actually get
$\bA^2\setminus \set{0}$  (or, rather, the  stack~$[(\bA^2\setminus \set{0})/\Gm]$)
as the answer?
Well, you can! If you study families of Fontaine--Laffaille \textit{modules} (the ``linear algebra perspective'' in $p$-adic Hodge theory) instead of \textit{representations} then this is exactly what you get.
The inclusion
$$[V^{\FL,\det=\omega^i}_{/\F_p}/\SL_2] \iso [V_{\star}/\Gm]
\hookrightarrow [(\bA^2\setminus \set{0})/\Gm]$$
identifies the source with the  locus in the Fontaine--Laffaille  moduli
stack over  which the Fontaine--Laffaille functor from 
Fontaine--Laffaille modules to Galois representations
(or,  more generally, to $\WD_{\Q_p}$-representations)
can be  applied so as to obtain genuine $\WD_{\Q_p}$-representations.

\begin{proof}[Sketch of proof  of Prop.~\ref{rotten herring}]
As already noted, both the  source and the target
of the claimed isomorphism are disconnected:  the source  is the 
disjoint union of the reducible  and irreducible loci,
while the target is the disjoint union of 
$[\bigl(V_{\star} \cap (\Gm \times \A^1)\bigr)/\Gm] $ and its
complement
%\mar{TG: referee says: ``Doesn’t the thickening of the vertical lines need to be taken into consideration in order to descibe the complement?''}
(which equals $[ (\Spf \F_p[[x]] \times \Gm)/\Gm] =  [\Spf \F_p[[x]]/\mu_2]$).
We check the  isomorphism on each connected component  separately.
On the irreducible locus, it  amounts to  the fact that the mod $p$
Fontaine--Laffaille deformation ring of $\Ind_{\Q_{p^2}}^{\Q_p} \omega_2^i$
is isomorphic to $\F_p[[x]]$.
On the reducible locus, there  is  more to check, but the key point
is  that  if $\rhobar$ is a non-split  extension of $\ur_{\lambda}$
by $\ur_{\lambda^{-1}} \omega^i$,  for some $\lambda \in\Fbar_p^{\times}$,
 which is trivial on some given open subgroup  $Q$ of~$P_{\Q_p}$,
then the  degree  of $\lambda$ over $\F_p$ is bounded in terms
of the index $[P_{\Q_p}:Q]$.
\end{proof}

%To obtain the isomorphism $V^{\FL, [0,i],\det=\omega^i} \iso
%V_\star$ that was asserted above, one can study the Fontaine--Laffaille
%functor  from Fontaine--Laffaille modules to Galois representations
%and determine over precisely which sublocus in~$V$ it can be applied
%in order to construct  $G_K$-representations (or, more generally,
%$\WD_K$-representations).  One finds that this locus is precisely~$V_\star$.\footnote{Here and in the next paragraph, there is a subtlety that we ignore, namely
%that framing a Fontaine--Laffaille module, or a $(\varphi,\Gamma)$-module,
%is not the same as framing the  Galois representation to which it  corresponds.
%We ignore this  point, since once  we pass to stacks (as we will),
%it is obviated in any case.}

This example of $V^{\FL,\det=\omega^i}$ is illustrative 
of the overall behaviour of $V$:
when working over $\F_p$ or $\Z_p$, it is nasty as soon as $d > 1$.
In the special case of 
$V^{\FL,\det=\omega^i}$, the moduli space
of Fontaine--Laffaille modules served as a much more pleasant improvement.
Fontaine--Laffaille theory doesn't apply to general mod~$p$ or $p$-adic 
Galois representations (even for~$\Q_p$, but even more so if $K/\Q_p$
is ramified).
But  Fontaine's theory of $(\varphi,\Gamma)$-modules does!
And in fact the moduli space~$V$ does sit inside a moduli space of
$(\varphi,\Gamma)$-modules, and as we will see,
this latter space is a lot nicer.
Thus we are motivated to define moduli stacks of
$(\varphi,\Gamma)$-modules and do a thorough study of their
geometry.% \mar{TG: might want to add a further discussion of CWE
  % stacks, or just a forward reference.}

\section{Stacks, etc.}In this lecture we give a very brief
introduction to the idea of a formal algebraic stack; details can be
found in~\cite{Emertonformalstacks}.

\subsection{Functors of points} There's a tension between rings and topological spaces in algebraic geometry, going back to Diophantus and Descartes. Grothendieck's formulation encompasses both, by giving you a topological space, but also giving you a sheaf of rings. You then solve equations by studying morphisms of schemes.

To study algebraic spaces and stacks, it is helpful to take the functorial point of view, which encompasses the theory of schemes, but gives you a framework to extend it. This perspective starts with the following observation.

\begin{lemma}[Yoneda, Grothendieck]
A scheme $X$ can be thought of as a \emph{(}pre\emph{)}sheaf $\Sch \xra{X} \Set$ by taking $Y \mapsto \Hom(Y,X)$. In fact it's a Zariski/\'etale/{\em fppf} sheaf.
\end{lemma}

But we can consider more general sheaves. For instance consider a system $X_1 \hra X_2 \hra \dots$ with closed immersions as transition maps, and define a sheaf
	\[ X := \varinjlim_i X_i \]
This is an Ind-scheme. Often, this will not actually be a scheme:
e.g.\ embed a point in a line in a plane, in $3$-space, etc. This gives
you some ``infinite dimensional affine space'', which is not a scheme:
note the identity map from this space to itself doesn't factor through
one of the finite steps, whereas a map from a quasi-compact scheme
does always factor in this  way.

This last remark illustrates one way that
finiteness/quasi-compactness assumptions are used in our theory;
they help us to detect finite-dimensional parts of possibly infinite-dimensional
objects.
%for this purpose.\mar{TG: maybe write something clearer than ``for
%  this purpose''?}

\subsection{Algebraic Spaces}

Another example of a more general sheaf is an algebraic space. If $X$
is an {\em fppf} sheaf, and $U$ is a scheme and $U \to X$ is a morphism,
then we say that the morphism $U\to X$ is ``representable by schemes" if
in any fibre product diagram
\begin{center}\begin{tikzcd}
	T \times_X U \rar \dar & U \dar \\
	T \rar & X
\end{tikzcd}\end{center}
for which $T$ is a scheme, then $T \times_X U$ is also a scheme. It
usually suffices to check this for a certain subclass of $T$ (e.g.\
affine schemes). You could also require that the morphism $U \to X$ is surjective or \'etale, which are both properties that one can define via base change to schemes.

\begin{definition}
An \textit{algebraic space} is an {\em fppf} sheaf $X$
which admits a map $U \to X$ whose source is a scheme
and which is representable by schemes,
surjective, and \'etale.
\end{definition}

There is an obvious notion of an Ind-algebraic space.

\subsection{Formal algebraic spaces}

We want to be able to talk about formal algebraic spaces, and formal algebraic stacks.

Suppose $A$ is a complete topological ring with a countable basis of open neighborhoods of $0$ consisting of ideals~$I_n$. Assume that every $A/I_{n+1} \to A/I_n$ is a nilpotent thickening. Then we define
	\[ \Spf A = \varinjlim \Spec A/I_n. \]
For instance if $A = \Z_p$ and $I_n = p^n$, then we get $\Spf \Z_p$. This is a proper subsheaf of $\Spec \Z_p$. Interestingly, properties like flatness, reducedness, Cohen--Macaulayness need the ring $A$, and can't be seen on the quotients $A/I_n$. For instance, $\Z_p$ is reduced, while each $\Z/p^n\Z$ for $n > 1$ is not.

\begin{definition}
A \textit{formal algebraic space} is an {\em fppf} sheaf $X$
which admits a map $\bigsqcup_i U_i \to X$ for which each $U_i$ is an affine formal scheme,
and such that the map is representable by algebraic spaces, \'etale and surjective.
\end{definition}

Then if $X = \varinjlim X_i$, where the $X_i$ are algebraic spaces and the maps are nilpotent thickenings, one can show that $X$ is a formal algebraic space. Vice versa, you can write a formal algebraic space using an Ind-construction.

\subsection{Stacks}

Stacks are ``sheaves'' of groupoids. Some (small) groupoids are
contractible, and are equivalent to their underlying sets, but some
are not. For example, consider the category with one object $x$ and two morphisms $\set{1_x, \sigma}$ where $\sigma^2 = 1_x$: this is not contractible. The moral is that in sets, equality is a \textit{property}, but in higher category theory and the theory of stacks, equality is exhibited by an isomorphism which is some \textit{extra data}.

Really we want to say that a stack is a ``$2$-sheaf'', which means a $2$-functor from the category of schemes to the $2$-category of groupoids, satisfying a $2$-categorical analogue of the sheaf condition. This is technically possible to do, but is a bit complicated once you start trying to work out all the coherence conditions you need, and usually involves making some kind of choices of pullbacks.

On the other hand, one way to avoid this is to use the formalism of \textit{categories fibered in groupoids}, which~\cite{stacks-project} does, and which we will do. For example, a morphism of stacks is a fully faithful embedding of categories fibered in groupoids. An occasionally annoying terminological
issue is that ``isomorphism of stacks'' means an equivalence of
categories fibered in groupoids. From now on we will typically denote
stacks by calligraphic letters~$\cX$, and schemes (or more generally
sheaves) by roman letters~$X$.

%But
There are some interesting things you get from the ``extra data'' involved in
having to {\em choose} an isomorphism. 
For example, the diagonal morphism
$\Delta: \cX \to \cX \times \cX$ is a monomorphism if and only if $\cX$ is equivalent to something that lands in sets (in other words, if and only if $\cX$ can
be regarded as being a usual sheaf of sets). In general,
%The moral is that
the diagonal really tells you something about isomorphisms between objects in the groupoids. For instance, if $T$ is a scheme and
$f:T \to  \cX\times \cX$ is a morphism, induced by the pair of morphisms
$f_1,f_2: T \rightrightarrows \cX,$
then the pull-back of $\Delta$ along $f$ will be the Isom sheaf
classifying isomorphisms between the elements $f_1$ and $f_2$
of the groupoid $\cX(T)$.
%then take the pullback along an affine scheme and you should get the isomorphism sheaf between the two objects defined by the morphism from the affine scheme.

\subsection{Algebraic and Formal Algebraic Stacks}
We make the following definitions.

\begin{definition}\label{defn: algebraic stack}
We say that an {\em fppf} stack~$\cX$ is an \emph{algebraic stack} if there
exists a morphism $U \to \cX$ whose source is a scheme, and which is
representable by algebraic spaces, surjective, and smooth. Note that
in this case asking for the morphism to be \'etale is strictly
stronger than asking for smooth, and defines the notion of a Deligne--Mumford stack,
which we will not make use of.
\end{definition}

\begin{example}
Recall the schemes~$V_Q$ from the previous lecture,
whose $A$-valued  points  correspond to representations $\WD_K/Q  \to \GL_d(A)$.
Usually in Galois representation  theory, we want  to study representations
up to isomorphism, which means that we should quotient out by
the  conjugation  action of $\GL_d$.  This motivates the
consideration of the quotient stack $[V_Q/\GL_d]$,
which is an algebraic  stack,
and is the {\em moduli stack} of $d$-dimensional  $\WD_K/Q$-representations. 

We can also form $[V/\GL_d] = \varinjlim_Q [V_Q/\GL_d]$,
which is an Ind-algebraic stack.   If  we  work  over  $\Z[1/p]$,
then  we saw in the last lecture  that $V$ becomes a usual scheme (though
not  of finite type), and so then  $[V/\GL_d]$ is again an algebraic
stack.
\end{example}

If  we  work  over~$\Z_p$,  then  we've  seen that $V,$
and so also $[V/\GL_d]$, is  not a very good object.
As we've already said, we will replace this  moduli stack of $\WD_K$-representations
by  a moduli stack of \'etale $(\varphi,\Gamma)$-modules. 
This won't  be  an algebraic stack, but it will be  a formal
algebraic stack, in the following sense.

\begin{definition}
We say that an {\em fppf} stack~$\cX$ is a \textit{formal algebraic stack} if
there exists a morphism $\bigsqcup_i U_i \to \cX$ which is
representable by algebraic spaces, surjective, and smooth, with the~ $U_i$ being affine formal schemes.
\end{definition}

If $\cX$ is qcqs, then we may write $\cX = \varinjlim \cX_i$, where $\cX_i$ are algebraic stacks and the transition maps are thickenings. The converse is true if $\varinjlim$ is countably indexed and the transition maps are thickenings.

If $f: \cX \to \cY$ is a morphism of stacks that is representable by
algebraic stacks (in the obvious sense), you can ascribe geometric properties to $f$. In
particular, if $\Delta: \cY \to \cY \times \cY$ is representable by
algebraic spaces, and~$\cX$ is an algebraic stack, then any morphism $\cX \to \cY$  is representable by algebraic stacks.

Now imagine $\cX$ is an algebraic stack, locally of finite type over
an excellent locally Noetherian scheme $S$, and consider a map $\cX \to
\cF$, %where $\cF$ is limit preserving, and such that
where $\cF$ is an {\em fppf} stack for  which $\Delta_\cF$ is
representable by algebraic spaces. %Since we're
Then as remarked in  the previous paragraph,
the morphism $\cX \to \cF$ is 
representable by
algebraic stacks.  %spaces, we can say that
In particular, it makes sense to ask that $\cX \to \cF$ be proper. In
practice, $\cX$ will be a stack of Breuil--Kisin modules of bounded height, and $\cF$ will be a stack of \'etale $\varphi$-modules.

We would like to form the ``scheme-theoretic  image'' of $\cF$  in~$\cX$.
Morally, this amounts to contracting the proper equivalence relation $\cF \times_{\cX}
\cF$ on~$\cF$.
In general you wouldn't expect contracting a proper equivalence
relation to give you something algebraic. However, we have the
following theorem, the main result
of~\cite{EGstacktheoreticimages}.

\begin{theorem}[Scheme Theoretic Images]\label{thm: scheme theoretic images}
Assume given an {\em fppf} sheaf  $\cF$ which is limit preserving and
whose diagonal is representable by algebraic spaces.
Suppose that $\cX$ is an algebraic stack, and that $\cX \to \cF$ is a proper morphism.
Assume further that $\cF$ admits versal rings at all finite type points,
and that these rings
satisfy the following effectivity property with respect to
the image of $\cX$: 
if $\Spf R \to \cF$ is a morphism realizing the pro-Artinian local ring $R$ as
a versal ring to~$\cF$, if we pull $\cX$ back over this morphism
to obtain a proper map $\cX_R \to \Spf R$, and if we let $\Spf S \hookrightarrow
\Spf R$ be the scheme-theoretic image of this map,
then the composite morphism $\Spf S \to \Spf R \to \cF$ factors through $\Spec S$.
%that if we pull back $\Spf R \to \cF$ to $\cX_R \to \Spf R$ and we take the scheme theoretic image $\Spf S \to \Spf R$, that this factors through $\Spec S$.

Then there exists an algebraic closed substack $\cZ \hra \cF$ such that $\cX \to \cF$ factors through a morphism $\cX \to \cZ$
which is proper and scheme theoretically dominant.
Furthermore,
 the closed substack $\cZ$ is uniquely determined by these properties,
 and we refer to it as the \emph{scheme-theoretic image} of the
 morphism $\cX\to\cF$.
\end{theorem}

We will use this theorem repeatedly in our context to construct the Ind-algebraic stacks that we want.

\section{Definitions related to \'etale $\varphi$-modules}%  and \'etale $(\varphi,
% \Gamma)$-modules}
In this lecture we give the definitions of the
various kinds of $\varphi$-modules with coefficients that we make use
of; in the following lecture, we will begin to use these definitions to
define our moduli stacks, and to prove their basic properties. 
\subsection{Rings}\label{subsec:rings}We begin with some material
from~\cite[\S 2.1]{emertongeepicture}. We fix a finite extension $K$ 
of~$\Q_p$ with ring of integers~$\cO_K$ and residue field~$k$,
and regard~$K$ as a subfield of some fixed algebraic closure $\Qbar_p$
of~$\Q_p$. We let $\C$ denote the completion of 
$\Qbar_p$. It is a perfectoid field,
whose tilt
$\C^{\flat}$ is a complete 
non-archimedean valued perfect field of characteristic~$p$.
If $F$ is a perfectoid closed subfield of $\C$,
then its tilt $F^{\flat}$ is a closed, and perfect,
subfield of $\C^{\flat}$. We will only need to consider the following two examples of~$F$ (other than~$\C$
itself), which arise from the theories of $(\varphi,\Gamma)$-modules
and Breuil--Kisin modules.

%\begin{example}[The Kummer case]
%	\label{ex:kummer}
%	If we choose a uniformizer $\pi$ of $K$,
%	as well as a compatible system of $p$-power roots
%	$\pi^{1/p^n}$ of $\pi$ (here, ``compatible'' has
%	the obvious meaning, namely that $(\pi^{1/p^{n+1}})^p = \pi^{1/p^n}$),
%	then we define $K_{\infty} = K(\pi^{1/p^{\infty}}) :=
%	\bigcup_n K(\pi^{1/p^n}).$
%	If we let $\Khat_{\infty}$ denote the closure
%	of $K_{\infty}$ in $\C$,
%	then $\Khat_{\infty}$ is again a perfectoid subfield of $\C$.
%\end{example}

\begin{example}[The cyclotomic case]
	\label{ex:cyclo}
	We write $K(\zeta_{p^{\infty}})$ to denote the 
	extension of~ $K$ obtained by adjoining all $p$-power
	roots of unity.  It is an infinite degree Galois extension
	of $K$, whose Galois group is naturally
	identified with an open subgroup of $\Z_p^{\times}$.
	We let $K_{\cyc}$ denote the unique subextension of 
	$K(\zeta_{p^{\infty}})$ whose Galois group
	over $K$ is isomorphic to $\Z_p$ (so $K_{\cyc}$ is
	the ``cyclotomic $\Z_p$-extension'' of $K$).
	If we let $\Khat_{\cyc}$ denote the closure 
	of $K_{\cyc}$ in $\C$,
	then $\Khat_{\cyc}$ is a perfectoid subfield of $\C$.
\end{example}

\begin{example}[The Kummer case]
	\label{ex:kummer}
	If we choose a uniformizer $\pi$ of $K$,
	as well as a compatible system of $p$-power roots
	$\pi^{1/p^n}$ of $\pi$ (here, ``compatible'' has
	the obvious meaning, namely that $(\pi^{1/p^{n+1}})^p = \pi^{1/p^n}$),
	then we define $K_{\infty} = K(\pi^{1/p^{\infty}}) :=
	\bigcup_n K(\pi^{1/p^n}).$
	If we let $\Khat_{\infty}$ denote the closure
	of $K_{\infty}$ in $\C$,
	then $\Khat_{\infty}$ is again a perfectoid subfield of $\C$.
\end{example}

\begin{rem}
  \label{rem: Krasner Ax Tate Sen}Let~$L=K_{\cyc}$
  or~$K_\infty$. Then~$\Lhat\otimes_L\Qpbar$ is an algebraic closure
  of~$\Lhat$, so that the absolute Galois groups of~$L$ and~$\Lhat$
  are canonically identified. The action of~$G_L$ on~$\Qpbar$ extends
  to an action on~$\C$, and by a theorem of
  Ax--Tate--Sen~\cite{MR0263786}, we have~$\C^{G_L}=\Lhat$. 
\end{rem}

If $\cO_F$ denotes the ring of integers in $F$ (with~$F=\C$ or one of
the two possibilities just discussed), then $\cO_F^\flat$ is the ring
of integers in~$F^\flat$. Since $F^\flat$ and
$\cO_F^\flat$ are perfect,
we may form their rings of (truncated) Witt vectors $W(\cO_{F}^\flat)$,
 $W(F^\flat)$, $W_a(F^\flat)$, $W_a(\cO_F^\flat)$ (where~$a\ge 1$ is an integer).
We write  $\Ainf:=W(\cO_{\C}^\flat)$.

We always consider these rings of Witt vectors as topological rings
with the so-called {\em weak topology}, which admits the following
description: if $x$ is any element of $\cO_F^\flat$ of positive
valuation, and if $[x]$ denotes the Teichm\"uller lift of $x$, then we
endow $W_a(\cO_F^\flat)$ with the $[x]$-adic topology, so that
$W(\cO_F^\flat)$ is then endowed with the $(p,[x]$)-adic topology.
The topology on $W_a(F^\flat)$ is then characterized by the fact that
$W_a(\cO_F^\flat)$ is an open subring (and the topology on
$W(F^\flat)$ is the inverse limit topology).

While we could formulate all of our results in terms of
$\varphi$-modules over these rings of Witt vectors, for the purposes
of proving our structural results about the stacks (and for connecting
to the usual theories of $(\varphi,\Gamma)$-modules and Breuil--Kisin
modules), we need to consider various smaller (in particular
Noetherian) subrings. These come from the Fontaine--Wintenberger
theory of the field of norms, but we will skip over this and go
straight to the definitions.

Firstly, in the Kummer case, we set~$\gS=W(k)[[u]]$, with an
       endomorphism~$\varphi$ determined by the properties that it is
       semilinear for the Frobenius on~$W(k)$, and satisfies~$\varphi(u)=u^p$, and
       let~$\cO_{\cE}$ be the $p$-adic completion of~$\gS[1/u]$. The choice of compatible
       system of $p$-power roots of~$\pi$ gives an element
       $\pi^{1/p^\infty}\in \cO_{\widehat{K_\infty}}^\flat$, and there
       is a continuous $\varphi$-equivariant
       embedding \[\gS\into W(\cO_{\widehat{K_\infty}}^\flat)\]
       sending $u\mapsto [\pi^{1/p^\infty}]$. This embedding extends
       to a continuous $\varphi$-equivariant
       embedding \[\cO_{\cE}\into W((\widehat{K_\infty})^\flat).\]

   Now consider the cyclotomic case, where we write $\tGamma_K :=
   \Gal(K(\zeta_{p^{\infty}})/K)$, and  the
       cyclotomic character induces an embedding
       $\chi:\tGamma_K \hookrightarrow \Z_p^{\times}$.  Consequently,
       there is an isomorphism
       $\tGamma_K \cong \Gamma_K \times \Delta$, where
       $\Gamma_K \cong \Z_p$ and $\Delta$ is finite.  We have
       $K_{\cyc} = (K(\zeta_{p^{\infty}}))^{\Delta}$. We will usually
       write~$\Gamma$ for~$\Gamma_K$ from now on.
       
% \mar{TG: have to think through if there are any more assumptions
%   needed to avoid irritating $k$ versus $k_\infty$ versus $k_\infty'$
%   issues -- maybe $p>2$?}
       {\bf We now make a simplifying assumption for the purposes of
         exposition: assume that $K/\Qp$ is unramified.} We will keep
       this assumption in place for the rest of this lecture and the
       following two lectures. (All of our results hold for general
       $K/\Qp$, but there is a subtlety: for general~$K/\Qp$, % not only
       % is it hard to write down the action of~$\varphi$ explicitly,
       % but
       there is no~$\varphi$-stable analogue of the
       ring~$\A_K^+$ defined in the next paragraph. This means that in many arguments
       in~\cite{emertongeepicture} we reduce to the
       unramified case by somewhat technical although essentially
       straightforward arguments. We want to avoid these
       complications in these lectures.)

       If we choose a compatible system of $p^n$th roots of $1$,
       then these give rise in the usual way to an element
       $\varepsilon \in (\widehat{K(\zeta_{p^{\infty}})})^{\flat}.$
       There is then a continuous embedding
       $$W(k)[[T]] \hookrightarrow 
       W( \cO_{\widehat{K(\zeta_{p^{\infty}})}}^\flat)$$
       (the source being endowed with its $(p,T)$-adic topology,
       and the target with its weak topology),
       defined via $T \mapsto [\varepsilon]- 1$.
       We denote the image of this embedding by $(\A'_{K})^+$.
       This embedding extends to an embedding
       $$\widehat{W(k)((T))} \hookrightarrow W\bigl( (\widehat{K(\zeta_{p^{\infty}})})^\flat
       \bigr)$$
       (here the source is the $p$-adic completion of the Laurent
       series ring $W(k)((T))$),
       whose image we denote by $\A'_{K}$. Write $T'_{K}\in\A'_K$ for
       the image of~$T$. The actions of ~$\varphi$ and~$\gamma\in\Gammat_{K}$
on~$T'_{K}\in\A_{K}'$ are given by the explicit formulae
\numequation\label{eqn: phi on T basic case with
  epsilon}\varphi(T'_{K})=(1+T'_{K})^p-1, \end{equation}
\numequation\label{eqn: gamma on T basic case with
  epsilon}\gamma(1+T'_{K})=(1+T'_{K})^{\chi(\gamma)},
\end{equation}where~$\chi:\Gammat_{K}\to\Z_p^\times$ again denotes the
cyclotomic character. We set
$\A_K := (\A'_K)^{\Delta}$,
$T_{K}=\tr_{\A_{K}'/\A_{K}}(T_{K}')$, and
$\A_{K}^+=W(k)[[T_{K}]]$; then we
have~$\A_{K}=W(k)((T_{K}))^{\wedge}$, and~$\A_{K}^+$ is
$(\varphi,\Gamma_{K})$-stable. % After possibly
% replacing~$T_{K}$ by ~$(T_{K}-\lambda)$ for some~$\lambda\in
% pW(k)$ we can and do assume that
We have $\varphi(T_{K})\in T_{K}\A_{K}^+$, and $g(T_{K})\in
T_{K}\A_{K}^+$ for all~$g\in\Gamma_{K}$. From now on we will
often write~$T$ for~$T_K$.

\begin{rem}
  \label{rem: could work with full cyclotomic but Zp is convenient}We
  could equally well work with the rings $\A_{K}'$ and the theory
  of~$(\varphi,\Gammat)$-modules, but it is often convenient to work
  with the procyclic group~$\Gamma$ and not have to carry around the
  finite group~$\Delta$.
\end{rem}

\begin{rem}
  \label{rem: our coefficient rings have the same forms in both the
    Kummer and cyclotomic cases}Note that in both the Kummer and
  cyclotomic settings, we are considering rings which are abstractly a
  power series ring over the Witt vectors, equipped with a lift of
  Frobenius. This will mean that various foundational parts of the
  theory can be developed in parallel for the two cases.
\end{rem}
% \mar{TG: for now my plan is not to introduce the tilde versions, but
%   instead to bury them in the (outlines of the) arguments.}
\subsection{Coefficients}\label{subsec: coefficients}
We now consider coefficients. Recall that if~$A$ is a $p$-adically
    complete $\Zp$-algebra, then~$A$ is said to be \emph{topologically of finite
    type over~$\Zp$} if it can be written as a quotient of a
  restricted formal power series ring in finitely many variables~$\Zp\langle\langle
  X_1,\dots,X_n\rangle\rangle$; equivalently, if and only if~$A/p$ is a finite
    type $\Fp$-algebra (\cite[\S0, Prop.\
    8.4.2]{MR3752648}). In particular, if~$A$ is a
    $\Z/p^a$-algebra for some~$a\ge 1$, then~$A$ is
    topologically of finite type over~$\Zp$ if and only if it is of
    finite type over~$\Z/p^a$. Our coefficient rings will always be
    assumed to be topologically of finite
    type over~$\Zp$, and they will usually also be (finite type)
    $\Z/p^a$-algebras for some~$a\ge 1$. 

    \begin{rem}
      \label{rem: bigger rings do exist}Our stacks are all limit
      preserving, so that values on $A$-valued points (for any~$A$)
      are determined by their values on~$A$ which are (topologically)
      of finite type. It therefore does not cause us any problems to
      restrict to coefficients~$A$ of this kind. In addition many of
      our arguments with $\varphi$-modules require this assumption
      on~$A$, and we do not know if the results hold without it.
    \end{rem}% \mar{TG: need to decide whether to get rid of~$\cO$ or
    %   not. Probably we'll want it at some point e.g.\ for HT weights,
    %   so I might well keep it.}

    % It is often convenient to introduce an auxiliary base ring for our
    % coefficients, which we will take to be the ring of integers $\cO$
    % in a finite extension $E$ of $\Q_p$; in this case, we will let
    % $\varpi$ denote a uniformizer of $\cO$. In practice~$E$ will be
    % chosen large enough to contain the images of all embeddings~$K\into\Qpbar$.

Our various coefficient rings are all defined by
taking completed tensor products. 
More precisely, 
 if $a \geq 1$,
let $v$ denote an element of the maximal ideal
of $W_a(\cO_\C^\flat)$
whose image
in $\cO_\C^\flat$ is non-zero.
We then set
$$
W_a(\cO_\C^\flat)_A =
W_a(\cO_\C^\flat) \cotimes_{\Z_p}  A
:= \varprojlim_i \bigl( W_a(\cO_\C^\flat)\otimes_{\Z_p}A\bigr)/v^i$$
(so that the indicated completion is the $v$-adic completion).
Note that any two choices of $v$ induce the same topology
on $W_a(\cO_\C^\flat)\otimes_{\Z_p} A$, so
that 
$W_a(\cO_\C^\flat) \cotimes_{\Z_p}  A$
is well-defined independent of the choice of $v$.
We then define
$$W_a(\C^\flat)_A =  W_a(\C^\flat) \cotimes_{\Z_p} A
:= W_a(\cO_\C^\flat)_A[1/v];$$
this ring is again well-defined independently of the choice of $v$. We set
$$
W(\cO_\C^\flat)_A =
W(\cO_\C^\flat)\cotimes_{\Z_p} A := \varprojlim_a W_a(\cO_\C^\flat)_A,
$$ and similarly
$$
W(\C^\flat)_A =
W(\C^\flat)\cotimes_{\Z_p} A := \varprojlim_a W_a(\C^\flat)_A.
$$
In keeping with our notation above, we will usually write
$$\AAinf{A} :=  W(\cO_\C^\flat)_A.$$

We define the imperfect coefficient rings in the Kummer and cyclotomic
cases in the same way, but we can be more explicit: we have $\gS_A=(W(k)\otimes_{\Zp}A)[[u]]$, and~$\OEA$ is the $p$-adic
      completion of~$\gS_A[1/u]$, and similarly for $\A^+_{K,A}$ and $\A_{K,A}$.

      \subsubsection{A digression on flatness and completion}
Since~$\Ainf$ is not Noetherian, we have to be a bit
careful when taking completions, and for example it does not seem to
be obvious that the natural map $\gS_A\to\AAinf{A}$ is injective. The
following is~\cite[Prop.\ 2.2.12]{emertongeepicture}, and is proved using~\cite[\S 5]{MR2774689} and ~\cite[Rem.\ 4.31]{2016arXiv160203148B}.

% \mar{TG: probably also want this result for the topological finite
%   type case. I think everything is OK if we can prove that the targets
% are Noetherian outside~$(v,p)$. For this, can we perhaps reduce to considering
% the cases of~$v,p$ separately? i.e.\ to showing that each of
% $W(\C^\flat)_A[1/v]$ and $W(\C^\flat)_A[1/p]$ are Noetherian?}
\begin{prop}%\mar{TG: haven't tried to check what's true yet.}
  \label{prop: maps of coefficient rings are faithfully flat injections}
Suppose that~$A\to B$ is a flat homomorphism
  of finite type $\Z/p^a\Z$-algebras for
  some $a\ge 1$. Then all the maps in the following diagram are flat.
  Furthermore the vertical arrows are all injections,
while the  horizontal arrows are all faithfully flat {\em (}and
so in particular also injections{\em )}.
If $A\to B$ is furthermore faithfully flat, then the same is true of the
diagonal arrows.
% \mar{ME: Have to slightly rewrite proof to treat flat vs. faithfully flat
% set-up that I've just introduced. ME: Now done, but should check it
% all makes sense.}
  % \[\xymatrix{\gS_A \ar[r]\ar[dd]&\OEA\ar[d]\\&\widetilde{\cO}_{\cE,A}\ar[d]\\
  %   \AAinf{A}\ar[r] & W(\C^\flat)_A \\ &\tA_{K,A}\ar[u]\\
  %   \A_{K,A}^+\ar[uu]\ar[r]&\A_{K,A}\ar[u]} \]

\[\xymatrix{&\A_{K,B}^+\ar[rr]\ar[dd]&&\AAinf{B}\ar[dd]&&\gS_B
    \ar[dd]\ar[ll]\\
    \A_{K,A}^+\ar[rr]\ar[dd]\ar[ur]&&\AAinf{A}\ar[dd]\ar[ur]&&\gS_A
    \ar[dd]\ar[ll]\ar[ur]\\
&\A_{K,B}\ar[rr]&& W(\C^\flat)_B && \OEB\ar[ll]\\
\A_{K,A}\ar[rr]\ar[ur]&& W(\C^\flat)_A\ar[ur] &
& \OEA\ar[ll]\ar[ur]}\]
\end{prop}

It is straightforward to show that $\varphi$ extends to a continuous endomorphism
of~$\AAinf{A}$, $W(\cO_C^\flat)_A$, and so on; and the natural action
of~$G_K$ on these rings is continuous, as is the action of~$\Gamma$
on~$\A^+_{K,A}$ and ~$\A_{K,A}$.

\subsection{$\varphi$-modules, Breuil--Kisin(--Fargues) modules, $(\varphi,\Gamma)$-modules, and Galois
  representations}\label{subsec: phi modules and Galois representations} Let~$R$ be a $\Zp$-algebra, equipped
  with a  ring endomorphism~$\varphi$, which is congruent to the
  ($p$-power) Frobenius modulo~$p$. If~$M$ is an $R$-module, we write \[\varphi^*M:=R\otimes_{R,\varphi}M.\]  % , and 
% extending some of them from~$\Z/p^a\Z$-algebras to $p$-adically
% complete $\Zp$-algebras, in a mostly formal fashion.
% We also use this discussion as an opportunity to establish some notation
% related to these results.
\begin{defn}
  \label{defn: general phi module} An \emph{\'etale
    $\varphi$-module over~$R$} is a finite $R$-module~$M$, equipped
  with a $\varphi$-semilinear endomorphism $\varphi_M:M\to M$, which
  has the property that the induced $R$-linear morphism
  \[\Phi_M:\varphi^*M\stackrel{1\otimes\varphi_M}{\longrightarrow}M \]
  is an isomorphism.  A morphism of
   \'etale $\varphi$-modules is a morphism of the underlying
  $R$-modules which commutes with the
  morphisms~$\Phi_{M}$. We say that~$M$ is \emph{projective} (resp.\
  \emph{free}) if it is projective of constant rank (resp.\ free of
  constant rank) as an $R$-module. 
\end{defn}
We apply this in particular with~$R=\A_{K,A}$ or~$\OEA$. In the former
case, an \emph{\'etale $(\varphi,\Gamma)$-module} is an \'etale
$\varphi$-module over~$\A_{K,A}$, equipped with a commuting continuous
semilinear
action of~$\Gamma$. In both cases there is a relationship with Galois
representations as follows; the case without coefficients is due to
Fontaine (and was the motivation for introducing
$(\varphi,\Gamma)$-modules in the first place), and the version with
coefficients is due to Dee~\cite{MR1805474}.

Let~$\AKnrhat$ denote the $p$-adic
completion of the ring of integers
 of the
maximal unramified extension of~$\A_K[1/p]$ in~$W(\C^\flat)[1/p]$;
this is preserved by the natural actions of~$\varphi$ and~$G_K$
on~$W(\C^\flat)[1/p]$. Define  $\AKnrhatA$ as usual. Then if~$A$ is an
Artinian $\Zp$-algebra, we have an equivalence of categories between the category of finite projective  \'etale
$(\varphi,\Gamma)$-modules~$M$ with $A$-coefficients, and the category of
finite free $A$-modules~$V$ with a continuous action of~$G_K$, given by
the functors  \[V\mapsto (\AKnrhatA 
  \otimes_{A}V)^{G_{\Kcyc}},\]  \[M\mapsto (  \AKnrhatA \otimes_{\A_{K,A}}M)^{\varphi=1}.\]% \mar{TG:
% obviously this needs to be filled out and made correct.}
Taking limits (and forming $\AKnrhatA$ with respect to the $\m_A$-adic
topology), we can extend this to the case that~$A$ is complete local
Noetherian, or ~$A=\Fpbar$; this is important later on for
identifying the versal rings of our moduli stacks with Galois
deformation rings. Note though that for more general~$A$ (e.g.\ $A=\Fp[X]$), it is not the
case that there is an equivalence between $(\varphi,\Gamma)$-modules
and Galois representations % e.g.\ for~$A=\Zp[X]$ this follows from
(\emph{cf.} the discussion in Lecture~\ref{subsec: FL vs WD}). %\mar{try to convert this to $(\varphi,\Gamma)$ example?}

There is also a version of this in the Kummer setting: if $\cO_{\widehat{\cE^{\nr}}}$ is the $p$-adic
completion of the ring of integers
in the
maximal unramified extension of~$\Frac(\cO_{\cE})$
in~$W(\C^\flat)[1/p]$, then for~$A$ as above, we have an equivalence
of categories between \'etale $\varphi$-modules over~$\OEA$ and
$G_{K_\infty}$-representations, given by \[V\mapsto (\widehat{\cO}_{\widehat{\cE^{\nr}},A} 
  \otimes_{A}V)^{G_{K_\infty}},\]  \[M\mapsto (  \widehat{\cO}_{\widehat{\cE^{\nr}},A} \otimes_{\OEA}M)^{\varphi=1}.\]

In the next few lectures we will define moduli stacks of \'etale
$\varphi$-modules and \'etale $(\varphi,\Gamma)$-modules. One of the
key tools in proving the basic properties of these stacks is to study
the corresponding stacks of \emph{finite height}
$\varphi$-modules over~$\A_{K,A}^+$ or~$\gS_A$. We'll do this more
generally in the next lecture, but for now we'll just have the following
definition, which is important in the crystalline and semistable
theory. (Actually, because~$K/\Qp$ is assumed unramified, the finite
height modules over $\A^+_{K,A}$ have an intrinsic utility for
describing crystalline representations in terms of \emph{Wach
  modules}. However, this doesn't extend to general $K/\Qp$, or to the
semistable case, so we will not make any use of this property of Wach
modules in these lectures.)

\begin{defn}%\mar{TG: add Fargues}
Let~$E(u)$ be the minimal polynomial over~$W(k)$ of our fixed
uniformiser~$\pi$ of~$K$.
We define a (projective) Breuil--Kisin module 
(resp.\ a Breuil--Kisin--Fargues module) of height at
  most~$h$ with $A$-coefficients to be a finitely generated projective $\gS_{A}$-module
  (resp.\ $\AAinf{A}$-module)
  $\gM$, equipped with a~$\varphi$-semilinear morphism
  $\varphi:\gM\to\gM$, with the property that the corresponding
  morphism % of $\AAinf{A}$-modules
  $\Phi_{\gM}:\varphi^*\gM\to\gM$ is
  injective, with cokernel killed by~$E(u)^h$. % \mar{TG: understand how
    % this compares to usual BKF definition}  
If~$\gM$ is a Breuil--Kisin module   then
$\AAinf{A}\otimes_{\gS_{A}}\gM$ is a Breuil--Kisin--Fargues module.
\end{defn}
Note that our Breuil--Kisin(--Fargues) modules are ``effective'', i.e.\ $\varphi$-stable. This means
that we will end up only considering non-negative Hodge--Tate weights;
the general case follows by twisting by a sufficiently large power
of~$E(u)$, so this is harmless.

The connection between Breuil--Kisin(--Fargues) modules  and
crystalline/semistable Galois representations will be explained in Lecture~\ref{sec: crystalline and
  semistable}.
\subsection{Almost Galois descent}\label{subsec: almost Galois descent}Let $F$ be a closed perfectoid
subfield of $\C$. The following is~\cite[Thm.\
2.4.1]{emertongeepicture}, which is proved by a somewhat elaborate
argument making use of almost Galois descent.
\begin{theorem}
	\label{thm:descending projective modules}
			Let $A$ be a finite type $\Z/p^a$-algebra,
			for some $a \geq 1$.
			The inclusion
			$W(F^\flat)_A  \to W(\C^\flat)_A$
			is a faithfully flat morphism
			of Noetherian rings,
			and
	the functor $M \mapsto W(\C^\flat)_A\otimes_{W(F^\flat)_A} M$
	induces an equivalence between the category of finitely generated 
	projective $W(F^\flat)_A$-modules and the category of
	finitely generated projective $W(\C^\flat)_A$-modules endowed
	with a continuous semilinear $G_F$-action. A quasi-inverse
        functor is given by $N\mapsto N^{G_F}$.
      \end{theorem}
We will find this result very useful later on, because it lets us
combine the Kummer and cyclotomic settings. The key extra ingredient
(which also takes some work) is that when we have the additional
structure of a $\varphi$-module, we can descend from the perfect
coefficient ring $W(F^\flat)_A$ (with~$F=\widehat{K}_{\cyc}$ or
$F=\widehat{K}_\infty$) to the imperfect coefficient rings~$\A_{K,A}$
and~$\OEA$ respectively, in the following way.

\begin{defn}
  \label{defn: GK module over Ainf}Let~$A$ be a $p$-adically complete $\Zp$-algebra. % finite type
% $\cO/\varpi^a$-algebra for some~$a\ge 1$.
An \emph{\'etale $(\varphi,G_K)$-module
with $A$-coefficients} (resp.\ an \emph{\'etale $(\varphi,G_{\Kcyc})$-module
with $A$-coefficients}, resp.\ an \emph{\'etale $(\varphi,G_{K_\infty})$-module
with $A$-coefficients}) is by definition a finitely
generated~$W(\C^\flat)_A$-module $M$ equipped with an 
% $\varphi$-semilinear
isomorphism of $W(\C^\flat)_A$-modules % \mar{TG: probably shouldn't say that,
% we linearised it!}
\[\Phi_M:\varphi^*M\isoto M,\] and
a~$W(\C^\flat)_A$-semilinear action of~$G_K$ (resp.\ $G_{\Kcyc}$,
resp.\ $G_{K_\infty}$), which is continuous % \mar{TG:
% explain in what sense}
and commutes with~$\Phi_M$. We say that~$M$ is projective if it is projective of constant rank as
a~$W(\C^\flat)_A$-module.
\end{defn}

      The following is~\cite[Prop.\ 2.7.8]{emertongeepicture}.
\begin{prop}
  \label{prop: equivalences of categories to Ainf}Let~$A$ be a
  finite type $\Z/p^a$-algebra for some~$a\ge 1$.
  \begin{enumerate}
  \item The functor $M\mapsto W(\C^\flat)_A\otimes_{\A_{K,A}}M$ is an
    equivalence between the category of finite projective \'etale
$\varphi$-modules over~$\A_{K,A}$ and the category of
finite projective \'etale $(\varphi,G_{K_{\cyc}})$-modules
with $A$-coefficients.

It induces  an
    equivalence of categories between the  category of finite projective \'etale
$(\varphi,\Gamma_K)$-modules with $A$-coefficients and the category of
finite projective \'etale $(\varphi,G_K)$-modules
with $A$-coefficients. 

\item The functor $M\mapsto W(\C^\flat)_A\otimes_{\OEA}M$ is an
    equivalence of categories between the  category of finite projective \'etale
$\varphi$-modules over~$\OEA$ and the category of
finite projective \'etale $(\varphi,G_{K_\infty})$-modules
with $A$-coefficients. 
  \end{enumerate}

\end{prop}

\section{Moduli stacks of $\varphi$-modules}\label{sec:stacks of phi modules}We now define our first
moduli stacks, using the objects introduced in the previous
lecture. We maintain our assumption (made purely for the purposes of
exposition) that~$K/\Qp$ is unramified.
\subsection{Results of Pappas--Rapoport}We begin our discussion of stacks with
$\varphi$-modules; the more complicated case of
$(\varphi,\Gamma)$-modules will be built on this. These stacks were
first studied by Pappas--Rapoport~\cite{MR2562795} in the context of
Breuil--Kisin modules; they showed (among other things) that there  are algebraic stacks
of Breuil--Kisin modules, and that the morphisms to the stacks of
\'etale $\varphi$-modules are well-behaved. These results were built
upon in~\cite[\S 5]{EGstacktheoreticimages}, which showed that the stack of
\'etale $\varphi$-modules is itself well-behaved, and in particular is
Ind-algebraic.

We begin by recalling the results of Pappas--Rapoport (in a slightly
more general context, where~$\varphi$ is not necessarily given by
$\varphi(u)=u^p$; this makes some arguments messier but doesn't change
any of the key points). % We put ourselves in the following

% We can unite these as follows (it is convenient to be able to prove
% various general statements in the Kummer and
% cyclotomic cases simultaneously).
\begin{situation}
  \label{subsubsec:general framework}
  Fix a finite extension~$k/\Fp$ and
  write~$\Aplus:=W(k)[[T]]$. Write~$\A$ for the $p$-adic completion
  of~$\Aplus[1/T]$.

  % \begin{rem}In~\cite{EGstacktheoreticimages}, $\Aplus$ was
  %   denoted~$\gS$; however in this paper we will have to consider
  %   several different choices of~$\Aplus$, and we prefer to reserve
  %   the notation~$\gS$ for the ring occurring in the definition of
  %   Breuil--Kisin modules, as in Section~\ref{subsubsec: Kummer}.
  % \end{rem}
  % \mar{TG: update this remark once coefficient rings section is in
  % place, to make it clear which of those coefficient rings this
  % theory applies to.}

  If $A$ is a $p$-adically complete~$\Zp$-algebra, we write
  $\Aplus_A:=(W(k)\otimes_{\Zp}A)[[T]]$; we equip~$\Aplus_A$ with its
  $(p,T)$-adic topology, so that it is a topological $A$-algebra
  (where~$A$ has the $p$-adic topology). Let $\AAA_A$ be the $p$-adic
  completion of~ $\Aplus_A[1/T]$, which we regard as a topological
  $A$-algebra by declaring $\Aplus_A$ to be an open subalgebra.

  % \mar{TG: just take $q=p$ here and below.}
  Let~$\varphi$ be a ring endomorphism of~$\A$ which is congruent to
  the ($p$-power) Frobenius endomorphism modulo~$p$, and satisfies~$\varphi(\A^+)\subseteq \A^+$.

By~\cite[Lem.\ 5.2.2 and
  5.2.5]{EGstacktheoreticimages} and~\cite[Lem.\ 3.2.5,
  3.2.6]{emertongeepicture}, 
  $\varphi$ is faithfully flat, and induces the usual 
  Frobenius on~$W(k)$, and it extends uniquely to an $A$-linear
  continuous endomorphism of~$\A^+_A$ and~$\A_A$.

  Note in particular that the formation of~$\gS_A$, $\OEA$,
  $\A_{K,A}^+$ and~$\A_{K,A}$ above are particular instances of this
  construction.

\end{situation}
% Similarly
% Situation~\ref{subsubsec:general framework}; that is, we
%   write~$\Aplus:=W(k)[[T]]$ and~$\A$ for the $p$-adic completion
%   of~$\Aplus[1/T]$, and we let~$\varphi$ be a ring endomorphism of~$\A$ which is congruent to
%   the ($p$-power) Frobenius endomorphism modulo~$p$, and which
%   satisfies~$\varphi(\A^+)\subseteq \A^+$.

Fix a polynomial $F\in W(k)[T]$ which is congruent to a positive power
of~$T$ modulo~$p$. (For example, if we are working in the Breuil--Kisin
setting, we will take~$T=u$, $\varphi(u)=u^p$, and~$F=E(u)$.)

In order to know that the various categories in groupoids that we will
 define are actually stacks, we use the following result of
 Drinfeld~\cite[Thm.\ 3.11]{MR2181808} (see~\cite[Thm.\
 5.1.18]{EGstacktheoreticimages} for the precise statements given here).

\begin{thm}
  \label{thm: fpqc locality of projective and locally free}
  The following notions are local for the fpqc topology on $\Spec A$.
  \begin{enumerate}
  \item A finitely generated projective $\A_A$-module.
\item A projective $\A_A$-module of rank~$d$.
  \item A finitely generated projective $\A_A$-module which is fpqc
    locally free of rank~$d$.
  \item A finitely generated projective $\A_A^+$-module.
\item A projective $\A_A^+$-module of rank~$d$.
  \item A finitely generated projective $\A_A^+$-module which is fpqc
    locally free of rank~$d$.
  \end{enumerate}

\end{thm}

\begin{rem}
  \label{rem: what fpqc locality means for projective etc}More
  precisely, saying that the notion of a finitely generated projective
  $\A_A$-module is local for the \emph{fpqc} topology on $\Spec A$ means
  the following (and the  meanings of the other statements in Theorem~\ref{thm: fpqc
    locality of projective and locally free} are entirely
  analogous):

If $A'$ is any faithfully flat $A$-algebra, set
  $A'':=A'\otimes_A A'$. Then the category of finitely generated
  projective $\A_A$-modules is canonically equivalent to the
  category of finitely generated projective $\A_{A'}$-modules $M'$
  which are equipped
  with an isomorphism \[M' \otimes_{\A_{A'},a\mapsto 1\otimes a}\A_{A''}
  \isoto M' \otimes_{\A_{A'},a\mapsto a\otimes 1}\A_{A''} \] which satisfies the usual cocycle condition.
\end{rem}

If we fix integers $a,d \geq 1,$
then by Theorem~\ref{thm: fpqc locality of projective and locally free} we may define
an \emph{fpqc} stack in groupoids $\cR_d^a$ over
$\Spec \Z/p^a$ as follows: For any $\Z/p^a$-algebra $A$,
we define $\cR^a_d(A)$ to be the groupoid of
 \'etale $\varphi$-modules~$M$ over~$\A_A$ which are projective % {\em fpqc} locally free
 of rank~$d$. There is a closely related version of this considered
in~\cite{MR2562795}, namely~$\cR_{d,\free}^a$, where we demand
that~$M$ is furthermore \emph{fpqc}-locally free. (In fact, we don't
know whether or not this is a consequence of~$M$ being projective, so
we don't know whether $\cR_{d,\free}^a=\cR_d^a$.)
In either case, if $A \to B$ is a morphism of $\Zp$-algebras, and $M$
is an object of $\cR_d^a(A)$,
then the pull-back of $M$ to $\cR_d^a(B)$
is defined to be the tensor product~$\AAA_B\otimes_{\AAA_A} M$.

From now on we regard $\cR^a_d$ as an {\em fppf} stack over~$\Z/p^a$. 
By~\cite[\href{http://stacks.math.columbia.edu/tag/04WV}{Tag
  04WV}]{stacks-project}, 
we may also regard the stack $\cR_d^a$ as an {\em fppf} stack over $\Zp$,
and as $a$ varies, we may form the $2$-colimit $\cR := \varinjlim_a \cR^a_d$,
which is again an {\em fppf} stack over $\Zp$.

The following definition generalises that of a Breuil--Kisin
module. %\mar{think about moving it up to that definition?}
\begin{defn}% \mar{TG: here and below should be $p$-adically complete
    % $\Zp$-algebra~$A$ and replace $\Spec A$ with $\Spf A$?}
  \label{defn: phi module of finite E-height}% \mar{TG: if $A$ is not
  %   $p$-power torsion, I haven't thought about whether we will want to
  % modify this definition. In keeping with remarks elsewhere, I'm
  % imagining that we do want to work with $p$-adic coefficients from
  % the beginning wherever we can?}
 Let $h$ be a non-negative
integer. A  \emph{$\varphi$-module of
   $F$-height at most~$h$ over~$\A^+_A$}
% \mar{ME: Maybe call these {\em Breuil--Kisin--Wach-modules}? TG: have
%   introduced this terminology, not sure if we want to use it elsewhere
% in the paper so I'll leave this marginal comment here for now.}
is  a pair $(\gM,\varphi_M)$
consisting of a finitely generated $T$-torsion free
$\Aplus_A$-module~$\gM$, and a % \mar{TG: seems a bit ugly to have to say
% ``$T$-torsion free''; do we ever use non-projective Kisin modules
% now?! f not then we could just define it for projective rank~$d$. TG:
% we do use them in arguments with versal rings, so commenting this out}
$\varphi$-semilinear map $\varphi_\gM:\gM\to\gM$, with the further
properties that if we
write \[\Phi_{\gM}:=1\otimes \varphi_{\gM}:\varphi^*\gM\to\gM,\] then % \mar{TG: I suspect that either I'm missing something
%   obvious or this isn't the right definition - it doesn't seem so good
% for passing mod~$p^n$, for example I get a bit tangled up proving
% Lemma~\ref{lem: map from finite E height to etale phi module}
% below. In the mod $p^a$-case I think this is OK though; originally in
% proper.tex we had that it was an isomorphism after inverting~$u$ and
% then deduced injectivity from that.}
$\Phi_{\gM}$ is injective, and the cokernel of $\Phi_{\gM}$ is killed by $F^h$.

A \emph{$\varphi$-module of finite $F$-height over~$\A^+_A$} is a $\varphi$-module of
  $F$-height at most~$h$  for some $h\ge 0$. A
  morphism of $\varphi$-modules is a morphism of the underlying
  $\Aplus_A$-modules which commutes with the morphisms~$\Phi_\gM$.

  We say that a $\varphi$-module of finite $F$-height is projective of
  rank~$d$ if it is a finitely generated projective~$\Aplus_A$-module of
  constant rank~$d$.
%  and if Zariski locally free on~$\Spf A$, it is finite free
 % over~$\Aplus_A$ of rank~$d$. % \mar{TG: I think this is the sensible
    % notion regardless of $A$ being $p$-power torsion, but I think that
    % won't (at least obviously) be the case for $A((u))$-modules, so
    % I'm putting a marginal comment in now to be alert to these things.}
\end{defn}
If we fix integers $a,d\ge 1$ and an integer $h\ge 0$,
then we may define an 
\emph{fpqc} stack in groupoids $\cC_{d,h}^a$
over $\Spec \Z/p^a$ as follows: For any $\Z/p^a$-algebra $A$,
we define $\cC^a_{d,h}(A)$ to be the groupoid of $\varphi$-modules of
$F$-height at most $h$ over~$\A^+_A$ which are projective
% (or equivalently, by Proposition~\ref{prop:improving fpqc for A[[u]]},
% {\em fpqc})
 of rank~$d$.

% That $\cC^a_d$ and $\cR^a_d$ are indeed \emph{fpqc} stacks follows from
% the results of~\cite{MR2181808}, and more specifically
% from~\cite[Thm.\ 5.1.12]{EGstacktheoreticimages}. 
%By~\cite[\href{http://stacks.math.columbia.edu/tag/04WV}{Tag
%  04WV}]{stacks-project}, %\mar{TG: check what this says}
Just as for the stack $\cR_d^a$,
we may and do also regard the stack $\cC_{d,h}^{a}$ as an {\em fppf}
stack over $\Zp$,
%\mar{ME: I'm a bit confused by this, partly b/c $\Zp$ is a ring, rather than
%	a scheme.  Do we mean over $\Spec \Zp$ (in which case I don't
%	think this is true), or maybe $\Spf \Zp$? TG: I think we mean $\Spf\Zp$}
and we then, allowing $a$ to vary,   define $\cC_{d,h}:=\varinjlim_{a}\cC_{d,h}^a$,
obtaining an {\em fppf} stack over $\Spf\Zp$.
%which in fact lie over $\Spf \Zp$.
% \mar{TG: need to check that
%   for $A$ $p$-adic complete, $\cC_d(A)$ and $\cR_d(A)$ have the expected
% descriptions. Insert this as a lemma once freeness statements cleared
% up above.}
There are  canonical morphisms $\cC_{d,h}^a\to\cR_d^a$ and $\cC_{d,h}\to\cR_d$ given % by
% inverting~$u$,\mar{TG: should now say ``
  by tensoring with $\A_A$ over~$\A_A^+$. One can show that any
  projective~$\A^+_A$-module is even Zariski locally free, so these
  morphisms factor through the Pappas--Rapoport versions
  $\cR_{d,\free}^a$ (resp.\ $\cR_{d,\free}$). Reassuringly, we have
  the following lemma~\cite[Lem.\ 3.1.3]{emertongeepicture}.
\begin{lem}
  If $A$ is a $p$-adically complete $\Zp$-algebra, then there
is a canonical equivalence between the groupoid of morphisms
$\Spf A\to\cR_{d}$ and the groupoid of rank~$d$
 \'etale $\varphi$-modules over~$\A_A$;
and there is a canonical equivalence between the groupoid of morphisms
  $\Spf A\to\cC_{d,h}$ %  naturally biject\mar{TG: how should this be
%     phrased? Should we be writing $\cC_{d,h}(\Spf A)$? ME: Should state
% an equivalence of groupoids.} with
and the groupoid of $\varphi$-modules of rank~$d$ and $F$-height at most~$h$
 over~$\A^+_A$.
\end{lem}
The following is essentially~\cite[Thm.\ 2.1~(a), Cor.\ 2.6]{MR2562795}.

\begin{thm}
  \label{thm: C is a $p$-adic formal algebraic stack and related
    properties}\leavevmode
  \begin{enumerate}
  \item The stack $\cC_{d,h}^a$ is an algebraic stack of finite
    presentation over $\Spec\Z/p^a$, with affine diagonal.
  \item The morphism $\cC_{d,h}^a\to\cR_{d,\free}^a$ is representable by
    algebraic spaces, proper, and of finite presentation.
  \item The diagonal morphism
    $\Delta:\cR_{d,\free}^a\to\cR_{d,\free}^a\times_{\Z/p^a}\cR_{d,\free}^a$ is
    representable by algebraic spaces, affine, and of finite
    presentation.
 % \mar{TG: any reason we can't just add that $\cR$
%   is also a limit preserving Ind-algebraic stack, whose diagonal is
%     representable by algebraic spaces, affine, and of finite
%     presentation? I think it should be formal from this (?) and seems
%     to be what we actually use later. ME:  Just to check --- the point of
% your comment is $\cR$ vs.\ $\cR^a$, right?  If so, yes, we should just
% add this! TG: yes, exactly. Presumably we could/should rewrite this
% whole thing to be more about $\cC$ and $\cR$, in fact, explaining how
% the statements reduce to the $\cC^a$, $\cR^a$? A bit tired to try to
% think about this right now though. TG: maybe keep this and put in a
% corollary for $\cC$, $\cR$.}
  \end{enumerate}
\end{thm}
We very briefly indicate the main ideas of the proof. One of the key
points is the following~\cite[Prop.\ 2.2]{MR2562795}
(see~\cite[Lem.\ 5.2.9]{EGstacktheoreticimages} for this
version). Assume that~$A$ is a $\Z/p^a\Z$-algebra. For  $n\ge
0$, write % \mar{TG: this notation could probably be improved}
\[U_n=1+T^nM_d(\A_A^+),\] \[V_n=\{A\in\GL_d(\A_A)\mid A,A^{-1}\in T^{-n}M_d(\A_A^+)\}.\]
\begin{lem}
    \label{lem: phi conjugacy approximation} % Let $H,h$ be integers with $H>ph/(p-1)>0$, and suppose that
    % ~$X,X'\in M_d(R)$  satisfy $X-X'\in u^H M_d(R)$, and that there
    % exists $Z\in M_d(R)$ with~$ZX=u^h\Id_d$. Then there is a matrix
    % $Y\in M_d(R)\cap\GL_d(R[1/u])$ such that $\varphi(Y)XY^{-1}=X'$.
For each~$m\ge 0$ there is an~$n(m)\ge 0$ \emph{(}implicitly depending also on~$a$\emph{)} such that if~$n\ge n(m)$, then:
\begin{enumerate}
\item For each $g\in U_n$, $B\in V_m$, there is a unique $h\in U_n$
  such that $g^{-1}B\varphi(g)=h^{-1}B$.
\item For each $h\in U_n$, $B\in V_m$ there is a unique $g\in U_n$
  such that $g^{-1}B\varphi(g)=h^{-1}B$.
\end{enumerate}
\end{lem}
In each case the uniqueness statement is easy, and the existence is
proved by a (slightly tricky) successive approximation argument. Given
the lemma, the first part of Theorem~\ref{thm: C is a $p$-adic formal algebraic stack and related
    properties} follows from standard results about the affine
  Grassmannian: the point is that working locally, the data of a
  finite height $\A_A^+$-module is the data of the matrix in $M_d(\A_A^+)$ of the
  corresponding linearized~$\varphi$, and the finite height condition
implies that this is in~$V_n$ for some~$n$ depending on~$a$
and~$h$. Changes of basis correspond to~$\varphi$-conjugacy, and
Lemma~\ref{lem: phi conjugacy approximation} lets us replace
$\varphi$-conjugacy by left multiplication provided we work in a
sufficiently deep congruence subgroup, and this is enough to get the
result.

The rest of Theorem~\ref{thm: C is a $p$-adic formal algebraic stack
  and related properties} is also proved by reducing to explicit
statements about matrices over~$\A_A$. The key point in each case is
to obtain a bound on the $T$-adic poles in some matrix entries, and 
the bound on the height~$h$ gives a bound on how different the
denominators of~$g$ and~$\varphi(g)$ can be (for some matrix~$g$),
which can be played off against the fact that the poles
of~$\varphi(g)$ have approximately~$p$ times the order of the poles of~$g$.

We then deduce that Theorem~\ref{thm: C is a $p$-adic formal algebraic
  stack and related properties} holds with~$\cR_{d,\free}^a$ replaced
by~$\cR_d^a$; roughly speaking the idea is to regard a projective
$\A_A$-module as a pair consisting of an \emph{fpqc}-locally free
projective~$\A_A$-module of higher rank, together with an
idempotent. We also make use of the following technical
statement~\cite[Lem.\ 5.1.23]{EGstacktheoreticimages}: if $M$ is a
finitely generated projective $A((u))$-module, then there exists
an~$n_0\ge 1$ such that for all $n\ge n_0$, $M$ is Nisnevich (and in
particular \emph{fpqc}) locally free as an $A((u^n))$-module. In
particular we find it useful to regard a projective $A((u))$-module as
a locally free $A((u^n))$-module together with an action of~$u$.

\subsection{Scheme-theoretic images}The main result of~\cite{EGstacktheoreticimages}
is the following~\cite[Thm.\ 5.4.20]{EGstacktheoreticimages}.

\begin{thm}
  \label{thm: R is a ind-algebraic}
  $\cR^a_d$ is a limit preserving Ind-algebraic stack, whose diagonal is
    representable by algebraic spaces, affine, and of finite
    presentation.  
\end{thm}
The key point is of course the Ind-algebraicity. In fact we show that
$\cR^a_d=\varinjlim_h\cR_{d,h}^a$ , where $\cR_{d,h}^a$ is the
scheme-theoretic image (in the sense of Theorem~\ref{thm: scheme theoretic images})
of the morphism $\cC_{d,h}^a\to\cR_d^a$. It is at least intuitively
reasonable that we have $\cR^a_d=\varinjlim_h\cR_{d,h}^a$ (this
says that \'etale locally, every \'etale $\varphi$-module comes from a
finite height $\varphi$-module, which in the free case is immediate by
just choosing a basis and scaling it by a sufficiently large power
of~$F$), and the hard part is to show that~$\cR_{d,h}^a$ is actually
an algebraic stack.

The key
point in applying Theorem~\ref{thm: scheme theoretic images} is to show % that~$\cR_d^a$ admits versal rings at all finite type
% points, and
that~$\cR_{d,h}^a$ admits \emph{effective} Noetherian
versal rings. Apart from the effectivity, this is at least morally
quite straightforward: the versal rings for~$\cR_d^a$ correspond to (infinite-dimensional)
unrestricted Galois deformation rings, and those for ~$\cR_{d,h}^a$
are closely related to their finite height analogues, which are Noetherian. In particular,
if we are in the Kummer setting and we take~$h=1$, then~$\cR_d^a$
really admits the unrestricted deformation rings for~$G_{K_\infty}$ as
versal rings, and ~$\cR_{d,1}^a$ admits a version of the height 1 deformation
rings, which can be identified with finite flat deformation rings
for~$G_K$ (which are in particular Noetherian).

% Making this precise involves some work, in part because we don't have
% a simple description of the $A$-valued points of~$\cR_{d,h}^a$
% unless~$A$ is a field.
The real work is in proving the effectivity,
i.e.\ that the versal morphism $\Spf R\to\cR_{d,h}^a$ can be upgraded
to a morphism $\Spec R\to\cR_{d,h}^a$. The key point here is to check
that the composite morphism $\Spf R\to\cR_{d}^a$ comes from a morphism
$\Spec R\to\cR_{d}^a$; it is then relatively straightforward to check
that this morphism factors through~$\cR_{d,h}^a$, using arguments
which are very similar to those that we will give in
Lecture~\ref{subsec: X is Ind algebraic}. Slightly more explicitly,
having a morphism $\Spf R\to\cR_{d}^a$ means that we have an \'etale
$\varphi$-module over the $\m_R$-adic completion of~$\A_R$ (i.e.\
where we are allowed to have infinite Laurent tails, as long as they
tend to zero $\m_R$-adically), and we need to show that this arises
from a projective \'etale $\varphi$-module over~$\A_R$ itself.

% \mar{TG: need to rewrite the material that follows to reflect the
%   arguments in the final version of proper.tex}
The reason this is true is that there is a universal (not
necessarily projective)
$\varphi$-module~$\gM_R$ over~$\A_R^+$ of $F$-height at most~$h$,
arising as the pushforward of the universal $\varphi$-module over~$\cC_{d,h}^a$. We then set
$M_{R^h}=\gM_{R^h}[1/T]$, which gives us a (not necessarily
projective) \'etale $\varphi$-module over~$\A_R$, whose $\m_R$-adic
completion agrees with the $\varphi$-module obtained from  $\Spf
R\to\cR_{d}^a$. (This compatibility with the $\m_R$-adic completion is
not completely obvious, but ultimately follows from the theorem on
formal functions.) It remains to check that that~$M_R$ is projective;
this is immediate from
the following general theorem (\cite[Thm.\ 5.5.20]{EGstacktheoreticimages}). % $R=R^h$
% and $M=M_{R^h}$. 
% To show this effectivity, we  show that the versal ring~$R$ to~$\cR_{d,h}^a$ is a
% quotient of a ring~$R^h$, which in turn is the quotient of the versal
% ring to~$\cR_{d,h}^a$ which is universal for
% the property that for each Artinian quotient~$A$ of~$R^h$, with
% corresponding \'etale $\varphi$-module~$M_A$,
% there is a (not necessarily projective) $\A_A^+$-module $\gM_A$ of
% $F$-height~$\le h$ with $M_A=\gM_A[1/T]$. This is hopefully at least
% plausible; the actual argument is a little involved, as the versal
% ring~$\Spf R$ is defined as the scheme-theoretic image of the pullback
% of $\cC_{d,h}^a\to\cR_d$ to~$\Spf R^{\univ}$, where~$R^{\univ}$ is a
% versal ring to~$\cR_d$, which is a deformation ring for an \'etale $\varphi$-module.

% If we let~$A$ run over the Artinian quotients of~$R^h$, then by a
% Mittag--Leffler argument we can choose the~$\gM_A$ compatibly, and
% obtain a (not necessarily projective) $\A_{R^h}^+$-module $\gM_{R^h}$
% of $F$-height~$\le h$. 

\begin{theorem}
	\label{thm:effectivity of projectivity}
	Let $R$ be 
a complete Noetherian local $\Z/p^a$-algebra with maximal
ideal~$\m$, 
% \mar{ME: Have to figure out the correct coefficients eventually,
% 	but I *think* this is the right setting.}
      let $M$ be an
      \'etale $\varphi$-module over~$\A_R$,
and suppose that the $\mathfrak m$-adic completion $\widehat{M}$
is projective, or equivalently, free
{\em (}over  the $\mathfrak m$-adic completion
of~$\A_R${\em )}.
Then $M$ itself is projective {\em (}over $\A_R${\em )}.
\end{theorem}
% The projective module~$M_{R^h}$ gives (by definition) a morphism
% $\Spec R^h\to\cR_{d}^a$, and thus a composite morphism
% $\Spec R\to\Spec R^h\to\cR_{d}^a$, as required.

  Finally, passing to the limit over~$a\ge 1$ and using some results
  of~\cite{Emertonformalstacks}, we obtain the following.
\begin{cor}
  \label{cor: basic properties of C and R p adic stacks}\leavevmode
  \begin{enumerate}
  \item $\cC_{d,h}$ is a $p$-adic formal algebraic stack of finite
    presentation over $\Spf\Zp$, with affine diagonal.
 \item  $\cR_d$ is a limit preserving Ind-algebraic stack, whose diagonal is
    representable by algebraic spaces, affine, and of finite
    presentation.
  \item The morphism $\cC_{d,h}\to\cR_d$ is representable by
    algebraic spaces, proper, and of finite presentation.
  \item The diagonal morphism
    $\Delta:\cR_d\to\cR_d\times_{\Spf\Zp}\cR_d$ is
    representable by algebraic spaces, affine, and of finite
    presentation.
   \end{enumerate}
\end{cor}

\section{Stacks of~$(\varphi,\Gamma)$-modules} 
\subsection{Moduli stacks of
  \texorpdfstring{$(\varphi,\Gamma)$}{(phi,Gamma)}-modules}\label{subsec:
defn of Xd}
Recall that for simplicity we are always assuming that~$K/\Qp$ is
unramified when we discuss $(\varphi,\Gamma)$-modules.
\begin{df} \label{defn: Xd}
	% \mar{ME: Should expand on this definition a bit, explain
	% 	why we get an {\em fpqc} stack (via Drinfeld),
        %         etc. TG: this is probably not what we want to write,
        %         but wanted to add something to get us moving again}
	We let $\cX_{K,d}$ denote the moduli stack of projective
	\'etale $(\varphi,\Gamma)$-modules of rank~$d$. More
        precisely, if~$A$ is a $p$-adically complete $\Zp$-algebra,
        then we define $\cX_{K,d}(A)$ (i.e., the groupoid of morphisms
        $\Spf A\to\cX_{K,d}$) to be the groupoid of 
        projective \'etale~$(\varphi,\Gamma)$-modules  of rank~$d$ with~$A$-coefficients,
 with morphisms given by isomorphisms. If $A\to B$ is a morphism of complete
        $\Zp$-algebras, and $M$ is an object of $\cX_{K,d}(A)$, then the
        pull-back of~$M$ to~$\cX_{K,d}(B)$ is defined to be the tensor
        product $\A_{K,B}\otimes_{\A_{K,A}}M$. % \mar{TG: probably not as naive
       It again follows from Theorem~\ref{thm: fpqc locality of projective and locally free} that $\cX_{K,d}$ is indeed a stack.   % as this?}
     \end{df}

     Note that by
the equivalence between $(\varphi,\Gamma)$-modules
     and $G_K$-representations over finite $\Zp$-algebras, % \mar{ME: Maybe ``Artinian rings''
% {\tt -->} ``finite $\Z_p$-algebras''?}
the
     $\Fpbar$-points of~$\cX_{K,d}$ are in bijection with representations
     $\rhobar:G_K\to\GL_d(\Fpbar)$. More generally,  fix a point $\Spec \F\to\cX_{K,d}$ 
for some finite field~$\F$, giving rise to a continuous representation
$\rhobar: G_K \to \GL_d(\F)$, and
% and also fix a $p$-adic Hodge type~$\underline{\lambda}$ and an
% inertial type~$\tau$.
let~$R^{\square}_{\rhobar}$ denote the universal framed
deformation $W(\F)$-algebra for lifts of~$\rhobar$. Then it is easy to
check that the natural morphism $\Spf
R^{\square}_{\rhobar}\to\cX_{K,d}$
(which again is defined via
the equivalence between $(\varphi,\Gamma)$-modules
     and $G_K$-representations over finite $\Z_p$-algebras)
is versal.

      One of the main results that we will prove is that $\cX_{K,d}$
      is a Noetherian formal algebraic stack.  However, the proof of
      this is quite involved, and in this lecture we establish the
      preliminary result that $\cX_{K,d}$ is an Ind-algebraic stack,
      which we deduce from the results that we have proved for stacks
      of \'etale $\varphi$-modules.

\begin{df}
We let~$\cR_{K,d}$ (frequently abbreviated to~$\cR_d$) denote the moduli stack of rank $d$ projective \'etale
$\varphi$-modules, taking $\A$ to be $\A_K$. 
\end{df}

We begin by %now turn to
studying $\Gamma$-actions on our $\varphi$-modules.  We
choose a topological generator $\gamma$ of $\Gamma$, and let
$\Gamma_{\disc} := \langle \gamma \rangle;$ so
$\Gamma_{\disc} \cong \Z$.  Note that since $\Gamma_{\disc}$ is dense
in~$\Gamma$, in order to endow $M$ with the structure of an \'etale
$(\varphi,\Gamma)$-module, it suffices to equip $M$ with a continuous
action of $\Gamma_{\disc}$ (where we equip $\Gamma_{\disc}$ with the
topology induced on it by $\Gamma$).

% \mar{ME: Maybe this discussion has to be amplified slightly, but 
% 	I'd rather not introduce any special notation for this stack, because
% 	(a) it just ends up being $\cR^{\Gamma_{\disc}},$ and (b)
% 	it looks a bit crappy to have the definition of our main
% 	object be followed by a formal definition of an object
% 	of only very temporary significance.}
% To this end we note that we may
% also
% In order to study the properties of $\cX_{K,d}$ we will take advantage
% of this, and begin by considering the moduli stack of 
% projective \'etale $\varphi$-modules of rank~$d$ equipped with
% a semilinear action of $\Gamma_{\disc}$. % \mar{TG: not bothering to put
% % a $K$ into this notation for now}
% % \mar{TG: $\Gamma_{\disc}$?}
% We don't introduce particular
% notation for this stack, since the following proposition
% identifies it with a fixed point stack $\cR_{K,d}^{\Gamma_{\disc}}$, which
% we now define.

There is a canonical action of~$\Gamma_{\disc}$ on~$\cR_d$ (that is, a canonical morphism $\gamma:\cR_d\to\cR_d$):
 if~$M$ is an object of~$\cR_d(A)$, then $\gamma(M)$ is given
by $\gamma^*M:=\A_{K,A}\otimes_{\gamma,\A_{K,A}}M$. % (Note that this is naturally a
% $\varphi$-module, because the action of~$\gamma$ on~$\A_{K,A}$ commutes
% with~$\varphi$.) 
Then we set % \mar{TG: in this context, $\Gamma_\gamma$
%   for the graph of~$\gamma$ looks pretty bad, but sticking with it for
% now.}
\[\cR_d^{\Gamma_{\disc}}:=
\cR_d\underset{\Delta,\cR_d\times\cR_d,\Gamma_\gamma}{\times}\cR_d, \]where
 $\Delta$ is the diagonal of~$\cR_d$ and $\Gamma_\gamma$ is the
 graph of~$\gamma$, so that $\Gamma_\gamma(x)=(x,\gamma(x))$.

 We claim that $\cR_d^{\Gamma_{\disc}}$ is nothing other than the
 moduli stack of projective \'etale $\varphi$-modules of rank~$d$
 equipped with a semilinear action of $\Gamma_{\disc}$. This is an
 exercise in unwinding the usual construction of the 2-fibre product:
 $\cR_d^{\Gamma_\disc}$ consists of tuples
 $(x,y,\alpha,\beta)$, with $x,y$ being objects of~$\cR_d$, and
 $\alpha:x\isoto y$ and~$\beta:\gamma(x)\isoto y$ being
 isomorphisms. This is equivalent to the category fibred in groupoids
 given by pairs~$(x,\iota)$ consisting of an object~$x$ of~$\cR_d$ and
 an isomorphism~$\iota:\gamma(x)\isoto x$. Thus an object
 of~$\cR_d^{\Gamma_\disc}(A)$ is a projective \'etale $\varphi$-module
 of rank~$d$ with $A$-coefficients~$M$, together with an isomorphism
 of $\varphi$-modules $\iota:\gamma^*M\isoto M$; and this isomorphism
 is precisely the data of a semilinear action
 of~$\Gamma_\disc=\langle\gamma\rangle$ on~$M$, as required.

 Since~$\cR_d$ is an Ind-algebraic stack, so
 is~$\cR_d^{\Gamma_\disc}$; indeed it follows from Corollary~\ref{cor:
   basic properties of C and R p adic stacks} that
 $\cR^{\Gamma_{\disc}}_d$ is a limit preserving Ind-algebraic stack,
 whose diagonal is representable by algebraic spaces, affine, and of
 finite presentation.

 Restricting the $\Gamma$-action on an \'etale
 $(\varphi,\Gamma)$-module to $\Gamma_{\disc}$, we obtain a morphism
 $\cX_{K,d} \to \cR_d^{\Gamma_{\disc}},$ which is fully faithful.
 Thus $\cX_{K,d}$ may be regarded as a substack of
 $\cR_d^{\Gamma_{\disc}}$; in particular, we deduce that its diagonal
 is representable by algebraic spaces, affine, and of finite
 presentation. Although $\cX_{K,d}$ is a substack of the Ind-algebraic
 stack $\cR_d^{\Gamma_{\disc}}$, it is not a closed substack, but
 should rather be thought of as a certain formal completion (as in the
 discussion of the rank one case in Lecture~\ref{subsec:
  discussion of rank 1 case}); in
 particular, since it is not a closed substack, we cannot immediately
 deduce the Ind-algebraicity of $\cX_{K,d}$ from the Ind-algebraicity
 of $\cR_d^{\Gamma_{\disc}}$. Instead, we will argue as in the proof
 that~$\cR_d$ is Ind-algebraic, and exhibit~$\cX_{K,d}$ as the
 scheme-theoretic image of an Ind-algebraic stack.

 From now on we will typically drop~$K$ from the notation, simply
 writing $\cX_d$, $\cR_d$ and so on. % \mar{just do this earlier because
   % there's no basic stuff here}
 To go further we need to understand
 the continuity condition in the definition of~$\cX_d$ more
 carefully. It is easy to check that we have
 $\gamma(T)-T\in(p,T)T\A_A^+$, and using this one can check the
 following special case of~\cite[Lem.\ D.26]{emertongeepicture}:

 \begin{lem}
	\label{lem:testing continuity on M mod T}
        Suppose that $A$ is a $\Z/p^a$-algebra for
        some~$a\ge 1$. Let~$M$
        be a finite projective $\A_A$-module, equipped with a
        semilinear action of~$\Gamma_{\disc}$. Then
        the following are equivalent:
        \begin{enumerate}
        \item\label{item: gamma cts} The action of~$\Gamma_{\disc}$ extends to a continuous
          action of~$\Gamma$.
        % \item\label{item: gamma cts mod p} The action of~$\langle \gamma\rangle$ on~$M\otimes_{\cO/\varpi^a}\F$ extends to a continuous
        %   action of~$\Zp$.
        % \item\label{item: gamma minus 1 squeezing M to p T n} For any lattice~$\gM\subseteq M$, and any~$n\ge 1$, there exists~$s\ge 0$
        %   such that $(\gamma^{p^s}-1)^i(\gM)\subseteq (p,T)^n\gM$ for
        %   all~$i\ge 1$.
        %      %   \item For any lattice~$\gM\subseteq M$, there exists~$s\ge 0$
        %   % such that $(\gamma^{p^s}-1)(\gM)\subseteq (p,T)\gM$.
          \item\label{item: gamma minus 1 squeezing M to T} For any lattice~$\gM\subseteq M$, there exists~$s\ge 0$
          such that $(\gamma^{p^s}-1)(\gM)\subseteq T\gM$. \emph{(}Here a
          \emph{lattice} is a finitely generated $\Aplus_{K,A}$-submodule
    $\gM\subseteq M$  whose $\AAA_{K,A}$-span
    is~$M$.\emph{)}
         \item For some lattice~$\gM\subseteq M$ and some~$s\ge 0$, we have
          $(\gamma^{p^s}-1)(\gM)\subseteq T\gM$.
        %    \item\label{item: gamma minus 1 squeezing some M to p T} For some lattice~$\gM\subseteq M$ and some~$s\ge 0$, we have
        %   $(\gamma^{p^s}-1)(\gM)\subseteq (p,T)\gM$.
        \item\label{item: gamma minus 1 top nilpt mod p}  The action of~$\gamma-1$ on~$M\otimes_{\Z/p^a}\Fp$
          is topologically nilpotent.
          % \item\label{item: gamma minus 1 top nilpt}  The action of~$\gamma-1$ on~$M$
          % is topologically nilpotent.
        \end{enumerate}
  \end{lem}
It is easy to use this criterion, together with the fact that $\cR_d^{\Gamma_\disc}$ is limit
  preserving, to show that $\cX_d$ is limit preserving;
  see~\cite[Lem.\ 3.2.19]{emertongeepicture}.
 
\subsection{Weak Wach modules}\label{subsec: Wach modules}
In this section we introduce the notion of a weak Wach module of
height at most~$h$ and level at most~$s$. These will play a purely technical
auxiliary role for us, and will be used only in order to show
that~$\cX_d$ is an Ind-algebraic stack; we won't use their relation to
crystalline representations.

        By Lemma~\ref{lem:testing continuity on M mod T}, if $A$ is a $\Z/p^a$-algebra for some~$a\ge 1$, and
        $\gM$ is a rank~$d$ projective $\varphi$-module of
        $T$-height $\leq h$ over $A$, such that~$\gM[1/T]$ is equipped
        with a semilinear action of $\Gamma_{\disc}$, then this
        action extends to a 
        continuous action of~$\Gamma$ if and only if for some~$s\ge 0$ we have $(\gamma^{p^s}-1)(\gM)\subseteq
        T\gM$. This motivates the following definition.%  Indeed, $\Gamma\cong\Zp$, and the canonical topology
        % on~$\gM[1/T]$ is induced by the $(p,T)$-adic topology on the
        % open subspace~$\gM$.
        % is the product of a
                                %  finite group\mar{TG: no 
%   finite group now?}
% and a pro-$p$-group,
% a semilinear $\Gamma$-action on  $\gM[1/T]$ is 
% continuous if and only if the induced $\Gamma$-action on $\gM/T\gM$ 
% is continuous when the latter is endowed with its $p$-adic
% topology. Indeed, both conditions are easily seen to be equivalent to
% asking that if~$m$ is sufficiently large, then~$\gamma^{p^m}$ acts
% trivially on~$\gM/(p,T)\gM$.
% \mar{TG: I think this is right, but
%   leaving this comment just in case of idiocy.}
% \mar{TG: think this through again and check I understand;
  % maybe add a remark about the proof.}

\begin{df}
	\label{def:weak Wach modules}	A {\em rank $d$ projective weak Wach module of $T$-height
    $\leq h$ and level $\le s$}
	 is a rank~$d$ projective $\varphi$-module $\gM$ 
	over $\A^+_{K,A}$, % \mar{TG: have to slightly clarify somewhere
          % what this means given that $\A^+_{K,A}$ is not
          % $\varphi$-stable; the definition should just be that after
          % inverting~$F$ it's an \'etale $\varphi$-module}
        which is of $T$-height $\leq h$, %  when viewed as a
        % $\varphi$-module over~$\A^+_{\Qp,A}$,\mar{TG: need to be
        %   careful here as there is now more than one $T$, maybe
        %   something like $T_{\Qp}$-height or $\pi$-height?}
   such that $\gM[1/T]$ is equipped with a semilinear action
   of~$\Gamma_{\disc}$ which satisfies $(\gamma^{p^s}-1)(\gM)~\subseteq~T\gM$.
% Lemma~\ref{lem:testing continuity on M mod T} shows
% 	that this action extends to a continuous action of~$\Gamma$.
\end{df}

\begin{df}
	We let $\cWW_{d,h}$ denote the moduli stack of rank $d$
       	projective weak Wach modules of $T$-height $\leq h$. %  \mar{TG: at
          % this point, should this be $T$-height at most~$h$?}
         We let~$\cWW_{d,h,s}$ denote the substack of rank $d$
       	projective weak Wach modules of $T$-height $\leq h$ and
        level~$\le s$.  % \mar{TG: a
          % stack, should we say more about it?}
\end{df}
We will next show that the stacks~$\cWW_{d,h,s}$ are $p$-adic formal
algebraic stacks of finite presentation over~$\Spf \Zp$. Since the
canonical morphism $\varinjlim_s \cWW_{d,h,s} \to \cWW_{d,h}$ is an
isomorphism (by definition), this will show in particular
that~$\cWW_{d,h}$  is an Ind-algebraic stack; we will also see that
the transition maps in this injective limit are closed immersions.

Recall that we have the $p$-adic formal
algebraic stack
$\cC_{d,h}$ classifying rank $d$ projective $\varphi$-modules over~$\A_{K,A}^+$
of $T$-height at most~$h$.
We consider the fibre product
$\cR_d^{\Gamma_{\disc}}\times_{\cR_d}\cC_{d,h}$, where the
map~$\cR_d^{\Gamma_{\disc}}\to\cR_d$ is the canonical morphism given by
forgetting
the~$\Gamma_{\disc}$ action;  this is the moduli stack of rank $d$
projective $\varphi$-modules~$\gM$ over~$\A_{K,A}^+$ of
$T$-height at most~$h$, equipped with a semilinear action
of~$\Gamma_{\disc}$ on~$\gM[1/T]$. It follows from Corollary~\ref{cor:
  basic properties of C and R p adic stacks} that 
$\cR_d^{\Gamma_{\disc}}\times_{\cR_d}\cC_{d,h}$ is a
  $p$-adic formal algebraic stack of finite presentation over~$\Spf\Zp$.

  Restricting the $\Gamma$-action on a weak Wach module
to $\Gamma_{\disc}$, %(and taking into account the preceding proposition),
we may regard $\cWW_{d,h}$ as a substack of
$\cR_d^{\Gamma_{\disc}}\times_{\cR_d}\cC_{d,h}$. The following
is~\cite[Prop.\ 3.3.5]{emertongeepicture}.  %  The following
% proposition records the basic properties of the stacks~$\cWW_{d,h}$
% and~$\cWW_{d,h,s}$.
\begin{prop}
  \label{prop:inductive description of W} For $s \geq 1,$ the morphism
	$$\xymatrix{\cWW_{d,h,s}
          \ar% ^-{\text{\em{(\ref{eqn:Wach monomorphisms for s})}}}
          [r]%  &
          % \cWW_{d,h} \ar[r]
          & \cR_d^{\Gamma_{\disc}}\times_{\cR_d}\cC_{d,h}\\ }$$ is a
        closed immersion of finite presentation.
        In particular, each of the stacks $\cWW_{d,h,s}$ is a $p$-adic
          formal algebraic stack of finite presentation over
          $\Spf \Zp$; and for each $s'\ge s$, the canonical monomorphism $
          \cWW_{d,h,s} \hookrightarrow \cWW_{d,h,s'}$ is a closed
          immersion of finite presentation.
        \end{prop}
        The proof of this is fairly straightforward: by definition, we
        need to show that the condition that
        ~$(\gamma^{p^s}-1)(\gM)\subseteq T\gM$ is a closed condition,
        and is determined by finitely many equations. We do this by
        reducing to the free case and considering the equations on the
        level of matrices.

\subsection{$\cX_d$ is an Ind-algebraic stack}
\label{subsec: X is Ind algebraic}

By definition, we have a $2$-Cartesian diagram
\numequation
\label{eqn:C to R square}
\xymatrix{\cWW_{d,h} \ar[r]\ar[d] &
	\cR_d^{\Gamma_{\disc}}\times_{\cR_d}\cC_{d,h} \ar[d] \\
\cX_d \ar[r] & \cR_d^{\Gamma_{\disc}} }
\end{equation}
If $h' \geq h$ then
the closed immersion $\cC_{d,h} \hookrightarrow \cC_{d,h'}$ is
compatible with the morphisms from each of its source and target
to $\cR_d$, 
%$\Gamma_{\disc}$-equivariant, %  and so % (by )\mar{ME: Cite non-existent discussion
% 	% of equivariant closed immersions here TG: probably don't need
%         % anything now?}
% induces a closed immersion
% $\cC_{d,h}^{\Gamma_{\disc}} \hookrightarrow \cC_{d,h'}^{\Gamma_{\disc}}.$
% This is compatible with the morphisms of source and target 
% to $\cR_d$, and
and so % pulling back over $\cR_d^{\Gamma_{\disc}}$,
we obtain a closed immersion 
\numequation
\label{eqn:closed immersion of Wachs}
\cWW_{d,h} \hookrightarrow \cWW_{d,h'}.
\end{equation}

By construction,
the morphisms $\cWW_{d,h}\to\cX_d$  are compatible,
as $h$ varies,
with the closed immersions~(\ref{eqn:closed immersion of Wachs}). 
Thus we also obtain a morphism
\numequation
\label{eqn:Ind W to X}
\varinjlim_h \cWW_{d,h} \to \cX_d.
\end{equation}
Roughly speaking, we will prove that~$\cX_d$ is an Ind-algebraic stack
by showing that it is the ``scheme-theoretic image'' of the morphism
$\varinjlim_h \cWW_{d,h} \to \cR_d^{\Gamma_{\disc}}$ induced
by~\eqref{eqn:Ind W to X}. More precisely, choose $s \geq 0,$ and consider the composite
        \numequation
        \label{eqn:W to R composite}
        \cWW_{d,h,s} \to \cWW_{d,h} \to \cX_d \to \cR_d^{\Gamma_{\disc}}.
        \end{equation}
        This admits the alternative factorization
        $$\cWW_{d,h,s} \to \cWW_{d,h} \to 	\cR_d^{\Gamma_{\disc}}\times_{\cR_d}\cC_{d,h} 
        \to \cR_d^{\Gamma_{\disc}}.$$
        Proposition~\ref{prop:inductive description of W} shows
        that the composite of the first two arrows is a closed 
        embedding of finite presentation,
        while  Corollary~\ref{cor: basic properties of C and R p adic stacks} % Lemma~\ref{lem:C to R is proper}
        % Theorem~\ref{thm: C is a $p$-adic formal algebraic stack and related
    % properties}\mar{TG: might need to augment this reference, as it's
    % not about fixed point stacks as stated.}
        shows that the third arrow is representable by algebraic spaces, proper, 
        and of finite presentation. 
        Thus~(\ref{eqn:W to R composite}) is representable by
        algebraic spaces, proper, and of
        finite presentation.   

	Fix an integer $a \geq 1,$ and write $\cWW^a_{d,h,s} := \cWW_{d,h,s}
	\times_{\Spf \Zp} \Spec \Z/p^a$.   
	Proposition~\ref{prop:inductive description of W}
	shows that $\cWW_{d,h,s}$ is a $p$-adic formal
	algebraic stack of finite presentation over $\Spf \Zp$,
	and so $\cWW^a_{d,h,s}$ is an algebraic stack, 
	and a closed substack of $\cWW_{d,h,s}$.

Note that since at this point we don't know that $\cX_d$
is Ind-algebraic, we can't directly define a scheme-theoretic image of 
$\cW^a_{d,h,s}$ in $\cX_d$.  It might be possible to do this using
the formalism of \cite{EGstacktheoreticimages}; we take a slightly
different approach.        
\begin{df}  We let $\cX^a_{d,h,s}$ denote the
scheme-theoretic image of the composite 
	\numequation
	\label{eqn:another composite}
	\cWW^a_{d,h,s} \hookrightarrow \cWW_{d,h,s}
	\buildrel \text{(\ref{eqn:W to R composite})} \over
	\longrightarrow \cR_d^{\Gamma_{\disc}}.
        \end{equation}
\end{df}

This is a morphism of Ind-algebraic stacks, and the scheme-theoretic
image has the obvious meaning: since $\cR_d^{\Gamma_{\disc}}$ is an
Ind-algebraic stack, % inductive limit
	% of algebraic stacks under closed immersions, 
        constructed as the $2$-colimit of a directed system of algebraic
        stacks whose transition morphisms are closed immersions,
	the morphism~\eqref{eqn:another composite},
%	which is representable by algebraic spaces, proper, and of finite presentation,
whose domain is a quasi-compact algebraic stack,
	factors through a closed algebraic substack $\cZ$ of
	$\cR_d^{\Gamma_{\disc}}.$   We then define $\cX^a_{d,h,s}$
	to be the 
	% \mar{TG: is ``it'' $\cWW^a_{d,h,s}$? ME: Yes, clarified.}
        scheme-theoretic image of
	$\cWW^a_{d,h,s}$ in $\cZ$.   Then
	$\cX^a_{d,h,s}$ % \mar{TG: $\cZ$ should be $\cX^a_{d,h,s}$? ME: Yes, fixed.}
        is a closed algebraic substack of $\cR_d^{\Gamma_{\disc}}$,
        and is independent of the choice of~$\cZ$.

        Our next goal is to prove that $\cX^a_{d,h,s}$ is a
        (necessarily closed) substack of~$\cX^a_d$. Our argument for
        this is a little indirect. By definition, it is enough to
        check that  if~$A$ is a finite type $\Z/p^a$-algebra, then for any
  morphism $\Spec A\to\cX^a_{d,h,s}$, the composite morphism
  $\Spec A\to\cX^a_{d,h,s}\to\cR_d^{\Gamma_{\disc}}$ factors
  through~$\cX_d$. More concretely, 
  if $M$ denotes the \'etale $\varphi$-module over~$A$, 
  endowed with a $\Gamma_{\disc}$-action,
  associated to the given point $\Spec A \to \cR_d^{\Gamma_{\disc}},$
  then we must show that the $\Gamma_{\disc}$-action
  on $M$ is continuous. By Lemma~\ref{lem:testing continuity on M mod
    T}, we need to show that~$M$ contains a (not necessarily
  projective) lattice~$\gM$ such that for some~$s\ge 0$, we have
  $(\gamma^{p^s}-1)(\gM)\subseteq T\gM$.

  To do this, we note that the natural map
  $A \to B := \varprojlim_i A_i$ is injective, where~$A_i$ runs over
  the Artinian quotients of~$A$. It is enough to find such a lattice
  for~$M_B$, and using that lattices of
  bounded height are uniformly commensurable, it is in fact enough to find a
  lattice of some fixed height for each~$M_{A_i}$. We are therefore
  reduced to the following lemma~\cite[Lem.\ 3.4.8]{emertongeepicture}.
	\begin{lem}
		\label{lem:Artinian points of scheme theoretic images}
		Suppose that $M$ is a projective \'etale $\varphi$-module of rank $d$
		over a finite type Artinian $\Z/p^a$-algebra $A$,
		and that~$M$ is endowed with an action of $\Gamma_{\disc}$,
		such that the corresponding morphism $\Spec A \to
		\cR_d^{\Gamma_{\disc}}$ 
		factors through $\cX^a_{d,h,s}$.
		Then $M$ contains a $\varphi$-invariant
		lattice % {\em (}in the sense of Definition~{\em \ref{defn: lattice})}
                $\gM$ of $T$-height $\leq h$,
		such that $(\gamma^{p^s}-1) (\gM) \subseteq
		T\gM.$
              \end{lem}

              Our proof of this lemma is a little involved. We can
              immediately reduce to the Artin local case, and then to
              a problem about the universal framed deformation rings
              of a fixed $\varphi$-module equipped with an action
              of~$\Gamma_{\disc}$. Writing~$R$ for this universal
              deformation ring, we can consider the subfunctor of
              deformations which admit a lattice~$\gM$ of the required
              type, and it is straightforward to check that this is
              representable by a quotient~$S$ of~$R$. The statement of
              the lemma reduces to showing that~$\Spf S$ contains
              the scheme-theoretic image~$\Spf T$ of  the morphism \[X:=\cW_{d,h,s}\times_{\cR_d^{\Gamma_{\disc}}}\Spf
                R\to\Spf R.\]
              We prove this using the following criterion~\cite[Lem.\ A.32]{emertongeepicture}.

              \begin{lem}\label{lem: criterion for Artin to map to scheme theoretic
    image}	Let $R \to S$ be a continuous surjection of
  pro-Artinian local rings,
	and let $X\to \Spf R$ be a finite type morphism
	of formal algebraic spaces. 
        
	Make the following assumption: 
	if $A$ is any finite-type Artinian local $R$-algebra
	for which the canonical morphism $R\to A$ factors through 
	a discrete quotient of $R$,
	and for which the canonical morphism $X_A \to \Spec A$
	admits a section,
	then the canonical morphism $R \to A$ furthermore factors
        through $S$.

        Then the scheme-theoretic image of~$X\to\Spf R$ is a closed
        formal subscheme of~$\Spf S$.
\end{lem}
The result then follows by unwinding the definitions: the point is
that admitting a section to~$X_A$ in particular gives a morphism $\Spec A\to\cW_{d,h,s}$, and the
          corresponding weak Wach module is a lattice of the kind
          being considered, which gives the required factorisation through~$S$.

Finally, we can prove that~$\cX_d$ is Ind-algebraic.
\begin{prop}
	\label{prop:X is an Ind-stack basic case}
The canonical morphism
	$\varinjlim \cX^a_{d,h,s} \to \cX_d$ is an isomorphism.
	Thus $\cX_d$ is an Ind-algebraic stack, and may in fact
	be written as the inductive limit of
        algebraic stacks of finite presentation, with the transition maps
        being closed immersions.
	% \mar{ME: Extract appropriate finiteness statements from the
	% 	argument and add them.  Basically: $\cX_d$ will
	% 	be locally of finite presentation/limit preserving,
	% 	i.e.\ will satisfy~[1]
	% 	(just b/c it's an inductive limit of things
	% 	satisfying~[1]).
	% 	Maybe we can also get that it's quasi-separated, by
	% 	some general limit argument? 
	% 	Quasi-compactness will be less obvious; e.g.\ the
	% 	fact that $\cX_{d,\red}$ is quasi-compact is a
	% 	genuine theorem we prove below. (But add a remark
	%         here noting that we'll prove it below.)}
      \end{prop}%\mar{TG: add a sentence to reduce to the basic case.}
\begin{proof}We just need to show that if  $T = \Spec A$ for a Noetherian
  $\Z/p^a$-algebra~$A$, then  any morphism
  $T\to\cX_d$
  factors through some~$\cX^a_{d,h,s}$,
  or equivalently, that the closed immersion
  \numequation
  \label{eqn:closed immersion into test scheme}
  \cX^a_{d,h,s}\times_{\cX_d} T \to T
  \end{equation}
  is an isomorphism,
  for some choice of $h$ and $s$.

  If $M$ denotes the \'etale $(\varphi,\Gamma)$-module corresponding
  to the morphism $\Spec A \to \cX_d$,
  then \cite[Prop.~5.4.7]{EGstacktheoreticimages}
  shows that
  we may find a scheme-theoretically dominant morphism 
  $\Spec B \to \Spec A$ such that $M_B$ is free of rank $d$. (In
  outline, the proof of that result is to note that~$\gM$ is free if and only
  if~$\gM/T\gM$ is free, so that~$\gM$ is automatically free over a dense open
  subset of~$\Spec A$, and use Noetherian induction.)
  If we show that the composite $\Spec B \to \Spec A \to \cX_d$
  factors through $\cX^a_{d,h,s}$ for some $h$ and $s$,
  then we see that the morphism $\Spec B \to \Spec A$
  factors through the closed subscheme
  $\cX^a_{d,h,s} \times_{\cX_d} \Spec A$ of $\Spec A.$ 
  Since $\Spec B \to \Spec A$ is scheme-theoretically dominant,
  this implies that~(\ref{eqn:closed immersion into test scheme})
  is indeed an isomorphism, as required.

  Since $M_B$ is free,
  we may choose a $\varphi$-invariant free lattice $\gM \subseteq M_B$,
  of height $\leq h$ for some sufficiently large value of $h$.
  Since the $\Gamma_{\disc}$-action on $M$, and hence
  on $M_B$, is continuous by assumption,
  Lemma~\ref{lem:testing continuity on M mod T}
  then shows that $(\gamma^{p^s}-1)(\gM)\subseteq T\gM$ 
  for some sufficiently large value of $s$.
  Then $\gM$ gives rise to a $B$-valued point
  of $\cW^a_{d,h,s}$, whose image in $\cR^{\Gamma_{\disc}}_d$
  is equal to the \'etale $\varphi$-module $M_B$.
  Thus the morphism $\Spec B \to \cX_d$
  corresponding to $M_B$
  does indeed factor through $\cX^a_{d,h,s}$.
%
%
%
%
%
%  \mar{ME: Everything that follows will probably be removed.}
%  as in the proof of Lemma~\ref{lem: scheme theoretic images factor
%    through X} reduce to the case that~$T=\Spec A$ is an Artin local
%  $\Z/p^a$-algebra.
%
%  Let~$M$ be the corresponding \'etale $(\varphi,\Gamma)$-module with
%  $A$-coefficients.  Since~$A((T))$ is local, $M$ is free, so we may
%  choose a lattice~$\gM\subset M$ which is a free rank $d$
%  $\varphi$-module of finite $T$-height (by choosing any basis for~$M$
%  and scaling it by an appropriate power of~$T$). By
%  Remark~\ref{rem:testing continuity on M mod u}, $\gM$ is a weak Wach
%  module of height at most~$h$ and level at most~$s$ for some
%  sufficiently large~$h,s$. The morphism $\Spec A\to\cX_d$ therefore
%  factors through~$\cW^a_{d,h,s}$ and thus through~$\cX^a_{d,h,s}$, as
%  required.
\end{proof}  

\section{Crystalline and semistable moduli stacks}\label{sec: crystalline and
  semistable}From now on we allow $K/\Qp$ to be arbitrary.
\subsection{Breuil--Kisin--Fargues modules admitting all descents}The
connection between Breuil--Kisin modules and crystalline
representations was first established by Kisin in~\cite{KisinCrys},
where he showed that if $K_\infty$ is as in Example~\ref{ex:kummer},
and $\rho:G_K\to\GL_d(\Zpbar)$ is a lattice in a crystalline
representation with Hodge--Tate weights in~$[0,h]$, then the \'etale
$\varphi$-module corresponding to~$\rho|_{G_{K_\infty}}$ (via the
functors explained in Lecture~\ref{subsec: phi modules and Galois
  representations}) arises from a (unique) Breuil--Kisin module of
height at most~$h$. Kisin also showed that on crystalline
representations, the restriction from~$G_K$ to~$G_{K_\infty}$ is fully
faithful. Unfortunately, when $h>1$, Kisin's description of the
Breuil--Kisin modules which arise from lattices in crystalline
representations is transcendental in nature (it is phrased in terms of
a certain meromorphic connection, which only exists after an extension
of scalars, admitting at worst log poles), and it is unclear (to us)
how to define crystalline moduli stacks purely in terms of
Breuil--Kisin modules for~$K_\infty$.

Our approach to constructing crystalline and semistable moduli stacks
relies on the observation that while most Breuil--Kisin modules do not give
crystalline or semistable representations of~$G_K$, this is ``mostly''
for the simple reason that the corresponding
$G_{K_\infty}$-representations do not even extend to~$G_K$. However,
if~$\rho$ is a $G_K$-representation such that~$\rho|_{G_{K_\infty}}$
admits a Breuil--Kisin module, then~$\rho$ is ``almost''
semistable: in particular, $\rho$ is potentially semistable, and
becomes semistable over~$K(\pi^{1/p^m})$ for some~$m$ depending only
on~$K$.  Using this result it is straightforward
to deduce that~$\rho$ is semistable if and only
if~$\rho|_{G_{K_\infty}}$ admits a Breuil--Kisin module \emph{for all
  possible choices of~\emph{(}$\pi$ and\emph{)}
  $(\pi^{1/p^n})_{n\ge 0}$}. This is the main idea behind our
construction of the stacks.

In fact, this statement (that ~$\rho|_{G_{K_\infty}}$
admits a Breuil--Kisin module only if~$\rho$ is potentially
semistable) was only proved after~\cite{emertongeepicture} was
completed, by Hui Gao~\cite{2019arXiv190508555G}. Instead, we use a
slightly weaker statement, which is proved by TG and Tong Liu
in~\cite[App.\ D]{emertongeepicture} (see Theorem~\ref{thm: admits all
  descents if and only if semistable} below).

For each choice of uniformiser~$\pi$ of~$K$, and each
choice~$\pi^\flat\in\cO_\C^\flat$ of $p$-power roots of~$\pi$, we
write~$\gS_{\piflat}$ for~$\gS=W(k)[[u]]$, regarded as a subring
of~$\Ainf$ via $u\mapsto[\piflat]$. For each choice of~$\piflat$ and
each~$s\ge 0$ we write~$K_{\pi^\flat,s}$ for $K(\pi^{1/p^s})$,
and~$K_{\pi^\flat,\infty}$ for~$\cup_sK_{\pi^\flat,s}$.
Write~$E_\pi(u)$ for the Eisenstein polynomial for~$\pi$,
and~$E_{\piflat}$ for its image in~$\Ainf$.  There is a natural ring
homomorphism $\theta:\Ainf\to\cO_\C$, which
satisfies~$\theta([x])=x^\sharp$ for any~$x\in\cO_{\C}^\flat$; and
~$E_{\piflat}$ is a generator of~$\ker\theta$.

\begin{defn}\label{defn: BKF GK module}
  A Breuil--Kisin--Fargues $G_K$-module of height at most~$h$ is a Breuil--Kisin--Fargues
  module of height at most~$h$ which is equipped with a semilinear $G_K$-action which commutes
  with~$\varphi$. % ; so in particular, $W(\C^\flat)\otimes_{\Ainf}\gMt$ is an \'etale
% %\mar{ME: ``\dots naturally a{\em n \'etale} $(\varphi,G_K)$-module \dots''?}
% $(\varphi,G_K)$-module.
\end{defn}

\begin{rem}
Note that if~$\gMt$ is a Breuil--Kisin--Fargues $G_K$-module, then
$W(\C^\flat)\otimes_{\Ainf}\gMt$ is naturally a 
%\mar{ME: ``\dots naturally a{\em n \'etale} $(\varphi,G_K)$-module \dots''?}
$(\varphi,G_K)$-module
in the sense of Definition~\ref{defn: GK module over Ainf}.
\end{rem}

%\begin{rem}
%  \label{rem: all BK BKF are effective}Our definition of a
%  Breuil--Kisin--Fargues module is less general than that of
%  ~\cite{2016arXiv160203148B}, in that we require~$\varphi$ to
%  take~$\gM$ to itself; this corresponds to only considering Galois
%  representations with non-positive Hodge--Tate weights. This
%  definition is convenient for us, as it allows us to make direct
%  reference to the literature on Breuil--Kisin modules. The
%  restriction to non-positive Hodge--Tate weights is harmless in our
%  main results, as we can reduce to this case by twisting by a large
%  enough power of the cyclotomic character (the interpretation of
%  which on Breuil--Kisin--Fargues modules is explained in~\cite[Ex.\
%  4.24]{2016arXiv160203148B}). (We are also only considering free
%  Breuil--Kisin--Fargues modules, rather than the more general
%  possibilities considered in~\cite{2016arXiv160203148B}.)
%\end{rem}

% \begin{defn}
%   \label{defn: BKF module with GK action}A Breuil--Kisin--Fargues $G_K$-module is
%   a Breuil--Kisin--Fargues module equipped with a continuous
%   semilinear action of~$G_K$ which commutes with~$\varphi$.
% \end{defn}

\begin{defn}
\label{defn: descending BKF to BK}Let $\gMt$ be a
  Breuil--Kisin--Fargues $G_K$-module of height at most~$h$ with a semilinear $G_K$-action. Then
  we say that~$\gMt$ \emph{admits all descents}  if the following
  conditions hold.
  \begin{enumerate}
  \item\label{item: existence of descent Zp version} For every choice of~$\pi$ and~$\piflat$, there is a
    Breuil--Kisin module~$\gM_{\pi^\flat}$ of height at most~$h$ with
    $\gM_{\pi^\flat}\subset(\gMt)^{G_{K_{\pi^\flat,\infty}}}$
    for which the induced morphism
    $\Ainf\otimes_{\gS_{\pi^\flat}}\gM_{\pi^\flat} \to \gMt$
    is an isomorphism.
    %and
    %$\gMt=\Ainf\otimes_{\gS_{\pi^\flat}}\gM_{\pi^\flat}$.
  \item\label{appendix item: M mod u descends} The $W(k)$-submodule $\gM_{\piflat}/[\piflat]\gM_{\piflat}$ of
  $W(\overline{k})\otimes_{\Ainf}\gMt$ is independent of the
    choice of~$\pi$ and~$\piflat$.
  \item\label{appendix item: M mod E descends} The $\cO_K$-submodule $\varphi ^*\gM_{\piflat}/E_{\piflat}\varphi^*\gM_{\piflat}$
    of  $\cO_\C\otimes_{\theta, \Ainf}\varphi^*\gMt$  is independent of the
    choice of~$\pi$ and~$\piflat$.
  \end{enumerate}
% \mar{TL: we will see (2) is redundant TG: should we explain this
%   somewhere below? It doesn't seem completely trivial to extract from
%   the proof to me.}
\end{defn}

% \begin{rem}
%   \label{rem: second condition is redundant}In fact
%   condition~\eqref{appendix item: M mod u descends} in Definition~\ref{defn:
%     descending BKF to BK} is redundant; see
%   Remark~\ref{rem: explanation that we didn't need a condition in BKF
%     but that we want it for coefficients} below. However, we include
%   the condition as it will be useful when considering versions of the
%   theory with coefficients and descent data.
% \end{rem}

\begin{defn}
  \label{defn: crystalline descending BKF module}Let~$\gMt$ be a
  Breuil--Kisin--Fargues $G_K$-module which
  admits all descents. We say that~$\gMt$ is furthermore
  \emph{crystalline} if for each choice of~$\pi$ and~$\piflat$, and each~$g\in G_K$, we
  have \[(g-1)(\gM_{\piflat})\subset
  \varphi ^{-1}([\varepsilon]-1)[\piflat]\gMt.  \]% \mar{TG: hopefully correct but as
% ever need to be careful with~$\varphi$ twist - might be~$[\piflat]^p$?}\mar{TL : fixed the condition}
\end{defn}
There is an equivalence of categories between the category of
~$(\varphi,G_K)$-modules $M$ of rank~$d$ and the category of
free~$\Zp$-modules~$T$ of rank~$d$ which are equipped with a
continuous action of~$G_K$, with the Galois representation corresponding
to~$M$ being given by $T(M)=M^{\varphi=1}$. Write~$V(M):=T(M)\otimes_{\Zp}\Qp$.

We deduce the following theorem~\cite[Thm.\ F.11]{emertongeepicture} from the results of Tong Liu's
paper~\cite{1302.1888} and Laurent Fargues' correspondence between
Breuil--Kisin--Fargues modules and $\BdR^+$-lattices \cite[Thm.\
4.28]{2016arXiv160203148B}.
\begin{thm}[T.G.\ and Tong Liu]
  \label{thm: admits all descents if and only if semistable} Let~$M$
  be a $(\varphi,G_K)$-module. Then~$V(M)$ is semistable with
  Hodge--Tate weights in~$[0,h]$ if and only if there is a
  \emph{(}necessarily unique\emph{)}
  Breuil--Kisin--Fargues $G_K$-module $\gMt$ which is of height at
  most~$h$, which admits all descents, and which satisfies
  $M=W(\C^\flat)\otimes_{\Ainf}\gMt$.

  Furthermore, $V(M)$ is
  crystalline if and only if~$\gMt$ is crystalline.
\end{thm}

More recently, Bhargav Bhatt and Peter Scholze~\cite{bhatt2021prismatic} have proved a
classification of lattices in crystalline $G_K$-representations in
terms of prismatic $F$-crystals; this puts
Theorem~\ref{thm: admits all descents if and only if semistable} (in
the crystalline case) in its natural context.

\subsection{Definition of the crystalline stacks}For
simplicity of exposition, we will only address the crystalline case
from now on; the proofs in the semistable case are identical (in fact
slightly simpler, because the crystalline case has the extra condition
on the Breuil--Kisin--Fargues modules). The actual definition of these
stacks, and the proof of their basic properties, is quite involved,
but the basic idea is quite simple: we will take the scheme-theoretic
image in~$\cX_d$ of the moduli stack of Breuil--Kisin--Fargues
$G_K$-modules which admit all descents. There is one wrinkle: we do
not know that the stack of  Breuil--Kisin--Fargues
$G_K$-modules admitting all descents is reasonably behaved, and we
instead work with a refinement, for which the $G_K$-action is
``canonical'', in the following sense.

We have the following reinterpretation
of some of the results and arguments of~\cite[\S2]{MR2745530} (see~\cite[\S 4.3,
4.4]{emertongeepicture}).

\begin{prop}\label{prop: Caruso Liu Galois action on Kisin}
  For any fixed $a,h$, and any sufficiently large $N$, there is a
  positive integer~$s(a,h,N)$ with the property that for any finite type
  $\Z/p^a$-algebra~$A$, any
  projective Breuil--Kisin
  module~$\gM$  of height at most~$h$, and 
  any~$s\ge s(a,h,N)$, there is a unique continuous action of~ $G_{K_s}=G_{K(\pi^{1/p^s})}$ on
  $\gMt:=\AAinf{A}\otimes_{\gS_{A}}\gM$ which commutes with~$\varphi$
  and is semilinear with respect to the natural action of~$G_{K_s}$
  on~$\AAinf{A},$ with the additional property that for
  all~$g\in G_{K_s}$ we have $(g-1)(\gM)\subset
  u^N\gMt$. % \mar{TG: should probably
    % explain how we are thinking of~$\gM$ as living inside $\AAinf{A}\otimes_{\gS_A}\gM$}

Conversely, for any Breuil--Kisin--Fargues $G_K$-module~$\gMt$
\emph{(}with $\Zp$-coefficients\emph{)}
 corresponding as above to a semistable
$G_K$ representation with Hodge--Tate weights in~$[0,h]$, and for any
choice of~$\piflat$ with corresponding descent~$\gM_{\piflat}$, 
 if~$s\ge s(a,h,N)$ then %  there is a Breuil--Kisin module~$\gM$
 % \mar{ME: Projective, height $\leq h$?}
 % which is a lattice in ~$M_{B,\infty}^a$, for which
 the restriction to~$G_{K_{\piflat,s}}$ of the action of~$G_K$ 
 on $\gMt/p^a$ % ~\[W(\C^\flat)_{\cO/\varpi^a}\otimes_{\gS_{\cO/\varpi^a}}\gM_{\piflat}\subset
 %   W(\C^\flat)_{\cO/\varpi^a}\otimes_{\A_{K,\cO}}\gMt\]\emph{(}given by the
 % restriction of the natural action of~$G_K$ on the
 % $(\varphi,G_K)$-module
 % $W(\C^\flat)_{\cO/\varpi^a}\otimes_{\A_{}}M_B$\emph{)}
 agrees with the action of the previous paragraph.
\end{prop}
\begin{proof}
  The proof of the first part is quite straightforward; the point is
  that by a standard argument with the~$\varphi$-structure, one can
  uniquely upgrade approximate homomorphisms (i.e.\ maps which are
  homomorphisms modulo~$u^N$) between Breuil--Kisin(--Fargues) modules
  to actual homomorphisms (see~\cite[Lem.\
  4.3.2]{emertongeepicture}). We can write down an ``approximate
  action'' of~$G_{K_s}$ on~$\gMt$, by just letting it act trivially
  on~$\gM$ (after fixing some basis - since the action is semilinear,
  it doesn't literally make sense to ask that it be trivial).

  The second part is proved in~\cite{MR2745530}.
\end{proof}

Note that Definition~\ref{defn: crystalline descending BKF
  module} carries over in an obvious fashion to the case with
coefficients in a topologically of finite type $p$-adically complete
$\Zp$-algebra~$A$ (see~\cite[Defn.\ 4.2.4]{emertongeepicture}). Using
the faithful flatness results of Proposition~\ref{prop: maps of coefficient rings
  are faithfully flat injections}, one can check that the existence of
a descent to a Breuil--Kisin module depends only on~$\pi$, and not
on~$\piflat$, and that any such descent is unique (\cite[Lem.\ 4.2.7,
4.2.8]{emertongeepicture}).

\begin{defn}For any~$h\ge 0$ we let~$\cC_{d,\crys,h}^a$ % and~$\cC_{d,\cris,h}$
  denote the
 limit preserving category of groupoids over
  $\Spec\Z/p^a$ determined by decreeing, for any finite
  type $\Z/p^a$-algebra~$A$, that
  $\cC_{d,\crys,h}^a(A)$ is the
  groupoid of Breuil--Kisin--Fargues $G_K$-modules with
  $A$-coefficients, which are of height at most~$h$, which admit
  all
  descents, whose $G_K$-actions are canonical, and which are crystalline. We let~$\cC_{d,\crys,h}:=\varinjlim_a\cC_{d,\crys,h}^a$.% \mar{TG: why do we
  % do it this way -- can't we just define the $p$-adic one to start
  % with? Is it harder to give some justification of the procedure of
  % passing to general objects by taking limits if we do that?}
\end{defn}% \mar{TG: this definition might change if the $p$-adic HT
  % stuff does, but it shouldn't be too serious, and I propose to leave
  % it alone for the begin.}
It is easy to check that if $A$ is a $p$-adically complete
  $\Zp$-algebra which is topologically of finite type, then $\cC_{d,\crys,h}(A)$ %\mar{TG: $\Spf$ versus
  %Spec issue? Also there are the same issues in Lemma~\ref{lem: finite
  %  O algebra points of potentially semistable Kisin stack} below.}
  is the groupoid of
  Breuil--Kisin--Fargues $G_K$-modules with $A$-coefficients, which
  are of height at most~$h$, which admit all descents, whose
  $G_K$-actions are canonical, and which are crystalline.
There is a natural morphism \[\cC_{d,\crys,h}\to\cX_d,\] which
 is defined, for finite type 
$\Z/p^a$-algebras, via
$\gMt \mapsto W(\C^\flat)_A \otimes_{\AAinf{A}} \gMt,$
with the target object being regarded as an $A$-valued point of $\cX_d$
via the equivalence of Proposition~\ref{prop: equivalences of
  categories to Ainf}.

For now we will admit the following theorem~\cite[Thm.\
4.5.20]{emertongeepicture}, which we prove below.
\begin{thm}\label{thm: semi stable stack is formal algebraic}% \mar{TG:
    % more finite type finite presentation. ME: Added to statement and proof.}
 $\cC_{d,\crys,h}$ is a $p$-adic formal algebraic stack
of finite presentation and affine diagonal.
The morphism
$\cC_{d,\crys,h}\to\cX_d$
is representable by algebraic
  spaces, proper, and of finite presentation. 
% \mar{ME: Add affine diagaonal. ME: Added to statement; still have to
%   add to proof TG: hopefully more or less immediate from Prop~\ref{prop: BKF closed in BK}?
% ME: Yes, added a sentence to this effect. Remove this comment if you're happy with what
% I wrote.}
\end{thm}%\mar{TG:  theoremneeds proof}

Let $\cC_{d,\crys,h}^{\fl}$ denote the flat part of~$\cC_{d,\crys,h}$
(i.e.\ the maximal substack which is flat over~$\Spf\Zp$; see~\cite[Ex.\
9.11]{Emertonformalstacks}). Then we define $\cX_d^{\crys,h}$ to be
the scheme-theoretic image of the morphism
$\cC_{d,\crys,h}^{\fl}\to\cX_d$. The following result
is essentially~\cite[Thm.\ 4.8.12]{emertongeepicture} (for the
purposes of exposition, we are ignoring inertial types and Hodge--Tate
weights for now).
\begin{thm}\label{thm: existence of ss stack}The closed substack~$\cX_{d}^{\crys,h}$ 
   of~$\cX_d$ is a   $p$-adic formal
  algebraic stack, % \mar{TG: make sure this is proved somewhere}
  which
  is of finite type and flat over~$\Spf\Zp$, and  is
  uniquely % \mar{TG: tweak this slightly to get the uniqueness statement
  % correctly ME: Did something}
  determined as a $\Zp$-flat
  closed substack of~$\cX_d$ by the following property:
%  $\cX_{d}^{\crys,\lambdau,\tau}$ and
%   ~$\cX_{d}^{\semis,\lambdau,\tau}$ are
% \item\label{item: the finite flat points of the crystalline stack}
  if~$\Acirc$ is a finite flat~$\Zp$-algebra, then
  $\cX_d^{\crys,h}(\Acirc)$ is precisely the
  subgroupoid of~$\cX_d(\Acirc)$ consisting of
  $G_K$-representations which are crystalline with Hodge--Tate weights
  contained in~$[0,h]$.
\end{thm}
That $\cX_{d}^{\crys,h}$ is a $p$-adic formal algebraic stack and is of finite type and flat over~$\Spf\Zp$  follows
from its construction and~\cite[Prop.\ A.21]{emertongeepicture} (that
is, ``the scheme-theoretic image of a $p$-adic formal algebraic stack
is a $p$-adic formal algebraic stack''). The characterisation of
its~$\Acirc$-points can be reduced to the local case, and thus to the
following important statement about versal rings: fix a point $\Spec \F\to\cX_d(\F)$ 
for some finite field $\F$, giving rise to a continuous representation $\rhobar: G_K \to \GL_d(\F)$.
% and also fix a $p$-adic Hodge type~$\underline{\lambda}$ and an
% inertial type~$\tau$.
Let~$R^{\crys,h}_{\rhobar}$ denote the universal framed
deformation $W(\F)$-algebra for lifts of~$\rhobar$ which are
crystalline with Hodge--Tate weights in~$[0,h]$ (which exists
by~\cite{MR2373358}).  Then there is an induced morphism $\Spf
R^{\crys,h}_{\rhobar}\to\cX_d$, and we have~\cite[Prop.\ 4.8.10]{emertongeepicture}:
\begin{prop}
  \label{prop: versal rings for pst stacks}The morphism  $\Spf R^{\crys,h}_{\rhobar}\to\cX_d$
factors
through a versal morphism $\Spf
R^{\crys,h}_{\rhobar}\to\cX_d^{\crys,h}$.
\end{prop}
\begin{proof}This is proved using Lemma~\ref{lem: criterion for Artin to map to scheme theoretic
    image}, in a similar way to the way we used it in the sketch proof
  of Lemma~\ref{lem:Artinian points of scheme theoretic images}; the
  actual argument is slightly more complicated as we prove an
  algebraization statement and then work with~$p$ inverted.
\end{proof}

%\subsection{Canonical actions}

We now prove Theorem~\ref{thm: semi
  stable stack is formal algebraic}. We will deduce it from the
corresponding statements for the stacks~$\cC_{d,h}$ of Breuil--Kisin
modules, which we discussed in Lecture~\ref{sec:stacks of phi
  modules}. % ; the key point is (roughly) to show that the additional
% structure of a $G_K$-action admitting all descents is only a finite
% amount of data.

Then for any~$h\ge 0$ and any choice of~$\pi^\flat$, we write~$\cC_{\piflat,d,h}$ for the moduli stack of
rank~$d$ Breuil--Kisin modules for~$\gS_{\pi^\flat,A}$ of height at
most~$h$, and~$\cR_{\piflat,d}$ for the corresponding stack of \'etale
$\varphi$-modules. We have a natural morphism
$\cX_{K,d}\to\cR_{\piflat,d}$ given by Proposition~\ref{prop:
  equivalences of categories to Ainf} and restriction from~$G_K$ to~$G_{K_\infty}$; for each~$s\ge 1$, 
we can factor this as\[\cX_{K,d}\to\cX_{K_{\piflat,s},d}\to\cR_{\piflat,d}. \]
Then for any fixed $a,h$, and any~$N$ and~$s\ge s(a,h,N)$ as in
  Proposition~\ref{prop: Caruso Liu Galois action on Kisin}, there is a canonical morphism
  $\cC^a_{\piflat,d,h}\to\cX^a_{K_{\piflat,s},d}$ obtained from the canonical
  action of Proposition~\ref{prop: Caruso Liu Galois action on Kisin},
which fits into a commutative
  triangle 
  \[\xymatrix{& \cC^a_{\piflat,d,h}\ar[dl]\ar[d]\\ \cX^a_{K_{\piflat,s},d}\ar[r]&\cR^a_{\piflat,d}}\]

We then define a
stack~$\cC^a_{\piflat,s,d,h}$ by the requirement that it fits into a 
2-Cartesian diagram \numequation\label{eqn: 2 cartesian canonical C}\xymatrix{\cC^a_{\piflat,s,d,h} \ar[r]\ar[d] &  \cC^a_{\piflat,d,h}\ar[d]\\
  \cX^a_{K,d}\ar[r] & \cX^a_{K_{\piflat,s},d}}\end{equation}
The lower horizontal arrow in
%the diagram defining $\cC^a_{s,d,h}$
this diagram is again defined via Proposition~\ref{prop:
  equivalences of categories to Ainf}, and it is not so hard to show
(using the properties that we have proved about~$\cX_{K,d}$ and its diagonal)
that it
is representable by algebraic spaces and of 
finite presentation.
Since~$\cC^a_{\piflat,d,h}$ is an
algebraic stack of finite presentation over $\Z/p^a$,
it follows that~$\cC^a_{\piflat,s,d,h}$   is also an algebraic
stack of finite presentation over $\Z/p^a$.
It is also straightforward to check that the right hand vertical arrow in this diagram is representable by
algebraic spaces, proper, and of
finite presentation, so that the same is true of the left
hand vertical arrow.%  Since~$\cX^a_{K,d}$ is an Ind-algebraic stack by Proposition~\ref{prop:X is an Ind-stack}, we can
% form the scheme-theoretic image~$\cZ^a_{\piflat,s,d,h}$ of the morphism $\cC^a_{\piflat,s,d,h}\to \cX^a_{K,d}$, which is a
% closed algebraic substack of~$\cX^a_{K,d}$, and so, by the following lemma
% (together with Proposition~\ref{prop:X is an Ind-stack}),
% is in particular of 
% finite presentation over $\Z/p^a$.
% \mar{ME: I think this last claim (and the implied argument for it --- that a closed
% immersion into something of finite presentation is of finite presentation) is correct,
% but maybe needs a tiny bit of justification. This might also help with the finite 
% presentation claims in what follows.}
% \mar{ME: Actually, I couldn't see where the scheme-theoretic image $\cZ^a$
% is actually used.  All mentions of it that I could find by searching the \TeX \ 
% file seem to be commented out. So I haven't actually added ``the following lemma'' yet,
% since maybe it's not needed, at least not at this point? TG:
% correct! This is something vestigal from the earlier arguments which
% only constructed effective versal rings.}

Given all this, the key  statement remaining to be proved is the
following~\cite[Prop.\ 4.5.15]{emertongeepicture}.
\begin{prop}
  \label{prop: BKF closed in BK}For each~$\piflat$ and each~$s$ as
  above, there is a  natural
  closed immersion $\cC^a_{d,\crys,h}\to
  \cC^a_{\piflat,s,d,h}$. In particular, ~$\cC^a_{d,\crys,h}$ is an
 algebraic stack of finite presentation over~$\Z/p^a$, and its
 diagonal is affine. % has an
 % affine diagonal. 
 % \mar{TG: needs proving, didn't check it yet!
  % This should be the key statement for having control of the
  % semistable stack.}\
%\mar{ME: Add affine diagonal statement}
\end{prop}%\mar{TG: remember to update this if the definitions change.}
\begin{proof}We have a morphism $\cC^a_{d,\crys,h}\to\cX^a_{K,d}$ (given
  by extending scalars to~$W(\C^\flat)$), and it follows that there
  is a natural morphism $\cC^a_{d,\crys,h}\to\cC^a_{\piflat,d,h}$, defined via
  $\gMt \mapsto \gM_{\piflat}$. The composite morphisms
  $\cC^a_{d,\crys,h} \to \cC^a_{\piflat,d,h} \to
  \cX^a_{K_{\piflat,s},d}$ and
  $\cC^a_{d,\crys,h} \to \cX^a_{K,d}\to \cX^a_{K_{\piflat,s},d}$
  coincide by construction, so we have an induced morphism
\numequation
\label{eqn:induced map}
\cC^a_{d,\crys,h}\to \cC^a_{\piflat,s,d,h}
\end{equation}
which we need to show is a closed immersion.

  To see that~\eqref{eqn:induced map}
%the induced map $\cC^a_{d,\crys,h}\to \cC^a_{\piflat,s,d,h}$
is at least a monomorphism, it is enough to note
  that if~$A$ is a finite type $\Z/p^a$-algebra,
  and~$\gMt$ is a Breuil--Kisin--Fargues module  over $A$ which admits all descents,
then~$\gMt=\AAinf{A}\otimes_{\gS_{\piflat,A}}\gM_{\piflat}$ is
  determined by~$\gM_{\piflat}$, and the $G_K$-action on~$\gMt$ is
  determined by the $G_K$-action on
  $W(\C^\flat)_A\otimes_{\AAinf{A}}\gMt$.

Finally, the proof~\eqref{eqn:induced map}
%$\cC^a_{d,\crys,h}\to \cC^a_{\piflat,s,d,h}$
is a closed immersion is a bit more involved: we need to show that the
conditions that $\gMt=\AAinf{A}\otimes_{\gS_{\piflat,A}}\gM_{\piflat}$
is $G_K$-stable and admits all descents are closed conditions. We make
repeated use of the results of~\cite[App.\ B]{emertongeepicture},
particularly ~\cite[Lem.\ B.28]{emertongeepicture} to show that the
vanishing loci of various morphisms of Breuil--Kisin--Fargues modules
are closed.\end{proof}

\subsection{Potentially crystalline moduli
  stacks with fixed Hodge--Tate weights}For applications, it is
important to be able to fix the Hodge--Tate weights in our crystalline
moduli stacks, and also to consider representations
which are only potentially crystalline. To this end, let~$L/K$ be a
fixed finite Galois extension; then there is an obvious notion of a
Breuil--Kisin--Fargues $G_K$-module which admits all descents
over~$L$, and an obvious extension of the above results to give stacks
of $G_K$-representations which become crystalline over~$L$ with
Hodge--Tate weights contained in~$[0,h]$. We
write~$\cC^{L/K,\fl}_{d,\crys,h}$ for the corresponding stack of
Breuil--Kisin--Fargues modules.

Since the inertia types and Hodge--Tate weights are discrete
invariants, we at least morally expect these stacks to decompose as a
disjoint union of stacks of potentially crystalline representations
with fixed inertial and Hodge types; so the key point should be to see
how to read this data off from the Breuil--Kisin--Fargues modules.

Fix a finite extension~$E/\Qp$ with ring of integers~$\cO$, uniformizer~$\varpi$, and residue
field~$\F$, and assume that~$E$ is large enough to contain the images
of all embeddings $K\into\Qpbar$. We will abusively write~$\cX_d$ for the base
change of~$\cX_d$ to~$\Spf\cO$, without further comment. From now
on~$\Acirc$ will denote a $p$-adically complete flat $\cO$-algebra
which is topologically of finite type over~$\cO$, and we write $A=\Acirc[1/p]$.

Let~$\gM_{\Acirc}$ be a Breuil--Kisin--Fargues $G_K$-module with~$\Acirc$-coefficients which
  admits all descents over~$L$, and write
  $\overline{\gM}_{\Acirc}:=\gM_{\Acirc,\piflat}/[\piflat]\gM_{\Acirc,\piflat}$
  (for some choice of~$\piflat$, with~$\pi$ now denoting a uniformiser
  of~$L$). Then~$\overline{\gM}_{\Acirc}$ has a natural $W(l)\otimes_{\Zp}A$-semilinear
  action of~$\Gal(L/K)$, which is defined as follows: if~$g\in \Gal(L/K)$,
  then~$g(\gM_{\Acirc,\piflat})=\gM_{\Acirc,g(\piflat)}$, so the
  morphism $g:\gM_{\Acirc,\piflat}\to g(\gM_{\Acirc,\piflat})=\gM_{\Acirc,g(\piflat)}$
  induces a morphism \[g: \gM_{\Acirc,\piflat}/[\piflat]\gM_{\Acirc,\piflat} \to
    \gM_{\Acirc,g(\piflat)}/[g(\piflat)]\gM_{\Acirc,g(\piflat)},\]and the source and
  target are both
  canonically identified with~$\overline{\gM}_{\Acirc}$.

    This action of~$\Gal(L/K)$ on~$\overline{\gM}_{\Acirc}$ induces an
  $L_0\otimes_{\Qp}A$-linear action of $I_{L/K}$ on the projective
  $L_0\otimes_{\Qp}A$-module
  $\overline{\gM}_{\Acirc}\otimes_{\Acirc}A$. Fix a choice of
  embedding $\sigma:L_0\into E$, and
  let~$e_\sigma\in L_0\otimes_{\Qp} E$ be the corresponding
  idempotent. Then
  $e_\sigma(\overline{\gM}_{\Acirc}\otimes_{\Acirc}A$) is a projective
  $A$-module of rank~$d$, with an $A$-linear action
  of~$I_{L/K}$. Up to canonical isomorphism, this module does
  not depend on the choice of~$\sigma$. We  write
    \[\WD(\gMt_{\Acirc}):=e_\sigma(\overline{\gM}_{\Acirc}\otimes_{\Acirc}A),\]
    a projective~$A$-module of rank~$d$ with an $A$-linear action of
    $I_{L/K}$. It is easy to check that this is compatible with base
    change (of $p$-adically complete flat $\cO$-algebras which are
    topologically of finite type over~$\cO$).

We now turn to Hodge types. Fix some choice of~$\piflat$ a uniformiser of~$L$, write~$\gM_{\Acirc}$
for~$\gM_{\piflat,\Acirc}$, and~$u$ for~$[\piflat]$. For
each~$0\le i\le h$ we define
$\Fil^i\varphi^*\gM_{\Acirc}=\Phi_{\gM_{\Acirc}}^{-1}(E(u)^i\gM_{\Acirc})$,
and we set $\Fil^i\varphi^*\gM_{\Acirc}=\varphi^*\gM_{\Acirc}$
for~$i<0$. Then for each~$0\le i\le h$, \[(\Fil^i\varphi^*\gM_{\Acirc}/E(u)\Fil^{i-1}\varphi^*\gM_{\Acirc})\otimes_{\Acirc}A\] is a finite
  projective $L\otimes_{\Qp}A$-module, whose formation is
  compatible with base change (see~\cite[Prop.\
  4.7.2]{emertongeepicture}). There is a natural action of $\Gal(L/K)$,
	which is semilinear with respect to the action
	of $\Gal(L/K)$ on $L\otimes_{\Q_p} A$ induced by its action
	on the first factor.
	Since $L/K$ is a Galois extension, the tensor product
	$L\otimes_{\Q_p} A$ is an \'etale $\Gal(L/K)$-extension 
	of $K\otimes_{\Q_p} A$, and so \'etale descent allows
	us to descend 
	$(\varphi^*\gM_{\Acirc}/E(u)\varphi^*\gM_{\Acirc})
	\otimes_{A^{\circ}} A$  
	to a filtered module over $K\otimes_{\Q_p} A$;
	concretely, this descent is achieved by taking $\Gal(L/K)$-invariants.
This leads to the following definition. %\mar{TG: perhaps the following
  %two definitions, and the paragraphs following them, should be
  %slightly rearranged? ME: Did something.}

\begin{df}
	\label{def:DdR L/K case}
	In the preceding situation, we write
	$$D_{\dR}(\gMt_{\Acirc}) := 
	\bigl((\varphi^*\gM_{\Acirc}/E(u)\varphi^*\gM_{\Acirc})
	\otimes_{A^{\circ}} A\bigr)^{\Gal(L/K)},$$  
		and more generally, for each $i \geq 0,$
		we write
	$$\Fil^i
	D_{\dR}(\gMt_{A^{\circ}}) :=
	\bigl((\Fil^i\varphi^*\gM_{\Acirc}/E(u)\Fil^{i-1}\varphi^*\gM_{\Acirc})\otimes_{\Acirc}A\bigr)^{\Gal(L/K)}$$
(and for $i < 0$, we write 
	$\Fil^i D_{\dR}(\gMt_{A^{\circ}}) :=
	D_{\dR}(\gMt_{A^{\circ}})$).
The property of being a finite rank projective module is preserved
under \'etale descent, and so we find that $D_{\dR}(\gMt_{\Acirc})$
is a rank $d$ projective $K\otimes_{\Q_p} A$-module, filtered by projective
submodules. %with projective associated graded modules. %  Thus %, just as above,

	Since $A$ is an $E$-algebra,
	we have the product decomposition  $K\otimes_{\Q_p} A
	\iso \prod_{\sigma: K \hookrightarrow E} A,$
	and so, if we write $e_{\sigma}$ for the idempotent
	corresponding to the factor labeled by $\sigma$
	in this decomposition,
	we find that 
	$$D_{\dR}(\gMt_{A^{\circ}}) = 
	\prod_{\sigma: K \hookrightarrow E} e_{\sigma} D_{\dR}(\gMt_{A^{\circ}}),$$
	where each
	$e_{\sigma}D_{\dR}(\gMt_{A^{\circ}})$
	is a projective $A$-module of rank~$d$.
	For each $i$, we write
	$$\Fil^i 
	e_{\sigma}D_{\dR}(\gMt_{A^{\circ}})
	=
	e_{\sigma}\Fil^i D_{\dR}(\gMt_{A^{\circ}}).$$
	Each $\Fil^i
	e_{\sigma}D_{\dR}(\gMt_{A^{\circ}})$ is again a projective $A$-module.
\end{df}

\begin{df}
	\label{def:Hodge type}
A Hodge type~$\lambdau$ of rank $d$ is by definition a set of tuples of
integers $\{\lambda_{\sigma,i}\}_{\sigma:K\into\Qpbar,1\le i\le d}$
with $\lambda_{\sigma,i}\ge \lambda_{\sigma,i+1}$ for all~$\sigma$ and
all $1\le i\le d-1$. We say that it is \emph{effective} if all
$\lambda_{\sigma,i}\ge 0$, and \emph{bounded by~$h$} if
all~$\lambda_{\sigma,i}\le h$. % If~$\lambdau$ is a Hodge type, we
% write~$\lambda|_L$ for the Hodge type for~$L$ determined by
% $(\lambda|_L)_{\sigma,i}=\lambda_{\sigma|_K,i}$ for each
% $\sigma:L\into\Qpbar$.

If $\underline{D} := (D_{\sigma})_{\sigma:K \hookrightarrow E}$
is a collection
of rank $d$ vector bundles over $\Spec A$, labeled (as indicated) by the embeddings
$\sigma:K \hookrightarrow E$,
then we say that $\underline{D}$ has Hodge type $\lambdau$ if
  $\Fil^i D_{\sigma}$
  has constant rank equal 
  to $\#\{j\mid\lambda_{\sigma|K,j}\ge i\}.$
\end{df}
As the notation suggests, the Hodge type of~$V(\gMt)$ agrees with the
Hodge type of $D_{\dR}(\gMt)$. Putting this all together, we
have~\cite[Prop.\ 4.8.2]{emertongeepicture}:

\begin{prop}
  \label{prop: HT weights are a closed condition}
  Let~$L/K$ be a finite
  Galois extension. Then the stack
 $\cC_{d,\crys,h}^{L/K,\fl}$
is a scheme-theoretic union of closed substacks
  $\cC_{d,\crys,h}^{L/K,\fl,\lambdau,\tau}$, where~$\lambdau$ runs
  over all effective Hodge types that are bounded by~$h$, and~$\tau$ runs over
  all $d$-dimensional $E$-representations of~$I_{L/K}$. These latter closed substacks are
  uniquely characterised by the following property: if~$\Acirc$ is a
  finite flat~$\cO$-algebra, then an $\Acirc$-point of
  $\cC_{d,\crys,h}^{L/K,\fl}$\emph{)} is a point of
  $\cC_{d,\crys,h}^{L/K,\fl,\lambdau,\tau}$ if and only if
  the corresponding Breuil--Kisin--Fargues module~$\gMt_{\Acirc}$ has
  Hodge type~$\lambdau$ 
  and inertial type~$\tau$.
\end{prop}%\mar{TG: I am making a mess of this, but I suspect that
\begin{rem}\label{rem: not obvious Hodge independent of piflat}It is
  not obvious (at least to us) from the definition that the Hodge
  filtration on $D_{\dR}(\gMt)$ is independent of the choice
  of~$\piflat$; rather, we deduce this independence from its
  compatibility with the Hodge filtration on~$D_{\dR}(V(\gMt))$.
\end{rem}

We then define the corresponding stacks
$\cX_{d,}^{\crys,\lambdau,\tau}$; we can extend the definition to
possibly negative Hodge--Tate weights by twisting by an appropriate
power of the cyclotomic character. Finally we can deduce~\cite[Thm.\ 4.8.14]{emertongeepicture}:

\begin{thm}%\mar{TG: put in pst pcris versions, say that it is uniquely
%   determined by this property.}
  \label{thm: dimension of ss stack}  % it has regular generic fibre, 
The algebraic
stack~$\cX_{d}^{\crys,\lambdau,\tau}\times_{\Spf\cO}\Spec \F$ is %  and
  % ~$\cX_{d}^{\semis,\lambdau,\tau}\times_{\Spf\cO}\Spec \F$ are
  equidimensional of dimension \[\sum_{\sigma}\#\{1\le i<j\le
  d|\lambda_{\sigma,i}>\lambda_{\sigma,j}\}.\]%

In particular, if~$\lambdau$ is regular, then the algebraic stack~$\cX_{d}^{\crys,\lambdau,\tau}\times_{\Spf\cO}\Spec \F$ % and
  % ~$\cX_{d}^{\semis,\lambdau,\tau}\times_{\Spf\cO}\Spec \F$ are
is  equidimensional of
  dimension~$[K:\Qp]d(d-1)/2$.   % $\cX_{d,\semis,h}(\Spf\Zpbar)$
  % \emph{(}resp.\ \emph{)} is naturally equivalent to the
  % groupoid of $\Zpbar$-lattices in potentially semistable \emph{(}resp.\
  % potentially crystalline\emph{)} $d$-dimensional
  % $\Qpbar$-representations of~ $G_K$ of Hodge type~$\lambdau$ and
  % inertial type~$\tau$. 
\end{thm}% \mar{TG: not sure if we actually have dimension theory of
This follows from Proposition~\ref{prop: versal rings for pst stacks},
from which it follows that the versal rings
to~$\cX_{d}^{\crys,\lambdau,\tau}$ are given by the corresponding
deformation rings $
R^{\crys,\underline{\lambda},\tau}_{\rhobar}$, and the computation of
the dimensions of these deformation rings in~\cite{MR2373358} (which
comes down to a computation on the generic fibre, using weakly
admissible modules).

\section{The Herr Complex, and families of extensions}\label{sec: Herr
complex and Ext groups}We continue to allow~$K/\Qp$ to be an arbitrary
finite extension.

\subsection{Herr Complex}

Let $M$ be a projective \'etale $(\varphi,\Gamma)$-module over
$\bA_{K,A}$, where~$A$ can be any $\Z_p$-algebra, but to make nontrivial statements we will often think about the case where $A$ is of finite type over $\Z/p^a$ for some $a \geq 1$. Then we define the \textit{Herr complex}
	\[ \cC^\bullet(M) = [M \xra{(\varphi - 1, \gamma-1)} M \oplus M \xra{(\gamma - 1) \oplus (1 - \varphi)} M] \]
concentrated in degrees $0,1,2$. Note that this is a complex of
$A$-modules, rather than $\bA_{K,A}$-modules (the terms are
$\bA_{K,A}$-modules, but the maps are only $\bA_{K,A}$-semilinear with
respect to the $\varphi$- or $\gamma$-actions). The following
is~\cite[Thm.\ 5.1.22]{emertongeepicture}.

\begin{theorem}\label{thm: properties of Herr complex}
Suppose $A$ is a Noetherian $\Z/p^a$-algebra with $A/p$
countable. Then the  Herr complex~$\cC^\bullet(M)$ is a perfect
complex of $A$-modules. If  $A$ is of finite type over~$\Z/p^a$, and~$B$ is a finite type $A$-algebra, then
there is a natural quasi-isomorphism $\cC^\bullet(M)\otimes^{\mathbb{L}}_AB\isoto\cC^\bullet(M\otimes_{\A_{K,A}}\A_{K,B}).$ %$\cC^\bullet(M \otimes_{\bA_{K,A}} \bA_{K,B}) \simeq B \otimes_A^{\bL} \cC^\bullet(M)$.
\end{theorem}%\mar{TG: currently the proof says nothing about the base change}
\begin{proof}
Because $A$ is Noetherian, $\bA_{K,A}$ is flat over $A$. Since $M$ is
a projective~$\A_{K,A}$-module, it is therefore also a flat
$A$-module, so the Herr complex is a complex of flat $A$-modules
(which are however not finite over $A$, but only over $\bA_{K,A}$). To
check that~$\cC^\bullet(M)$ is perfect, it therefore suffices to show
that the cohomology $A$-modules are of finite type over $A$. By
Nakayama's lemma for complexes, it in fact suffices to prove the
theorem in the case that $A$ is an $\F_p$-algebra, so we assume this
from now on.

If $A$ is in fact a  finite-dimensional  $\Fp$-vector space, Herr~\cite{MR1693457} showed that
$\cC^\bullet(M)$ computes $H^\bullet(G_K, \rho)$, where $\rho$ is the
associated $G_K$-representation corresponding to $M$; so the statement
we want to prove is analogous to proving a version of the finiteness
of Galois cohomology in the presence of coefficients. Our proof largely
follows that of Herr.

We don't  have an equivalence between
$(\varphi,\Gamma)$-modules and Galois representations for general
coefficient rings, but we have similar behavior, as in the following lemma.

\begin{lemma}
$H^i(\cC^\bullet(M_1^\vee \otimes M_2)) =
\Ext_{(\varphi,\Gamma)/\bA_{K,A}}^i(M_1, M_2)$  for $i = 0,1$. In addition, $H^2(\cC^\bullet(\ad M))$ controls the obstructions to infinitesimal deformations of $M$.
\end{lemma}
It may well be the case that
$H^2(\cC^\bullet(M_1^\vee \otimes M_2)) =
\Ext_{(\varphi,\Gamma)/\bA_{K,A}}^2(M_1, M_2)$, but we have not
attempted to verify this.

There is an $A$-linear map $\psi: \bA_{K,A} \thra \bA_{K,A}$, which is given by
$\psi = \frac{1}{p} \tr\varphi$,  and an induced map $\psi: M \thra M$,
which is defined so that for~$a\in A$ and $m\in M$ we have $\psi(\varphi(a)m) = a \psi(m)$ and $\psi(a\varphi(m)) = \psi(a)m$.

Since~$\varphi$ commutes with~$\Gamma$, so does~$\psi$, and we have an
induced morphism $(1 - \gamma): M^{\psi = 0} \to M^{\psi = 0}$. In fact,
this morphism is an isomorphism. The proof of this takes some work,
but at least morally it has the following explanation in the simplest
case, when~$A=\Fp$, $K = \Qp$
and~$M=\A_{\Qp}$. Working instead with~$\Gammat$
   and $(\bA_{\Qp}')^+ = \Z_p[[\Z_p]]$, % which has a natural action of
   % $\Z_p^\times$. Therefore
we have $\F_p[[\Z_p]]^{\psi = 0} = \F_p[[\Z_p^\times]] $ (as an
$\F_p[\Gammat]$-module)
$= \F_p[\Delta] \otimes_{\Fp} \F_p[[x]]$, where $x = \gamma - 1$. Thus we see
that~$(\gamma-1)$ acts on~$\Fp[[x]]$ by multiplication by~$x$, which
is injective.   After passing to  the fraction field of $\F_p[[\Z_p]]$,
i.e.\ after  passing to $\bA_{\Qp}'$, i.e.\ after  inverting~$T'_{\Q_p}$,
one finds, upon passing to $\psi$-invariants, that  this  corresponds
to  inverting~$x$; i.e.\ 
that $(\bA_{\Qp}')^{\psi = 0} = \F_p[[\Delta]\times \F_p((x)).$
%$\psi$-invariants and is furthermore an isomorphism after inverting~$x$),
%while it acts trivially on~$\Fp[\Delta]$, as required.   \mar{TG: probably
%  something to be added here.}
Thus multiplication by~$\gamma  -1  = x$  now  acts invertibly.

Using the~$\psi$ operator, we can give an alternative description of
the Herr complex as follows. Let $\cC_\psi^\bullet(M)$ be the complex in degrees
$0,1,2$ given by \[\xymatrix{0\ar[r]&M\ar[rr]^{(\psi -1,\gamma-1)}&&
      M\oplus M\ar[rr]^{(\gamma-1)\oplus(1 - \psi)}&&M\ar[r]&
    0. }\]

  We have a morphism of complexes
  $\cC^\bullet(M)\to\cC^\bullet_\psi(M)$ given by
 % \numequation
%  \label{eqn:Herr quasi-iso}
\[  \xymatrix{0\ar[r]&M\ar[d]^{1}\ar[rr]^{(\varphi-1,\gamma-1)}&&
      M\oplus M\ar[d]^{(-\psi,1)}\ar[rr]^{(\gamma-1)\oplus(1-\varphi)}&&M\ar[d]^{-\psi}\ar[r]&
    0 \\ 0\ar[r]&M\ar[rr]^{(\psi-1,\gamma-1)}&&
      M\oplus M\ar[rr]^{(\gamma-1)\oplus(1-\psi)}&&M\ar[r]&
    0 } \]
%\end{equation}
That this is a morphism of complexes follows from the facts that
  $\psi\circ\varphi=\id$, and that~$\psi$ commutes with~$\gamma$; and 
% Now
% \begin{center}\begin{tikzcd}
% 	\cC^\bullet(M): & 0 \rar & M \dar[equals] \rar{(\varphi - 1, \gamma - 1)} & M \oplus M \rar{(\gamma - 1) \oplus (1 - \varphi)} \dar{-\psi \otimes \id} & M \rar \dar{-\psi} & 0 \\
% 	\cC^\bullet_\psi(M): & 0 \rar & M \rar{(\psi - 1, \gamma - 1)} & M \oplus M \rar{(\gamma - 1) \oplus (1 - \psi)} & M \rar & 0
% \end{tikzcd}\end{center}
since $(1 - \gamma): M^{\psi = 0} \to M^{\psi = 0}$ is an isomorphism,
we see that % the morphism of complexes
% $\cC^\bullet(M)\to \cC^\bullet_\psi(M)$
is in fact a quasi-isomorphism.

With some more work, we can show that this morphism induces
topological isomorphisms of the cohomology groups. We deduce the
required finiteness from this: the basic idea is that a Frobenius
truncation argument shows that the cohomology groups of
$\cC^\bullet(M)$ are subquotients of finitely generated
$A((T))/A[[T]]$-modules, and thus have the discrete topology, while
the cohomology groups of $\cC^\bullet_\psi(M)$ are subquotients of
finitely generated $A[[T]]$-modules (so the slogan is: ``discrete and
compact implies finite''). It is in these arguments that we use the
countability assumption on~$A$, in order to use % the same topology and using that
% that $A$ is countable and $\bA_{K,A}$ has a countable basis as an
% $A$-module, we are , so we can use
the theory of Polish groups.
%The
% top is discrete, the bottom is finitely generated, and with some
% effort we can see that \mar{TG: flesh this out slightly}

Finally, to see that the obvious map of complexes
$\cC^\bullet(M)\otimes_AB\isoto\cC^\bullet(M\otimes_{\A_{K,A}}\A_{K,B})$
induces a quasi-isomorphism, we use an argument of~\cite{MR3230818} to
formally reduce to the case that~$B$ is a finite $A$-module, in which
case the claim follows easily as $\A_{K,B}=\A_{K,A}\otimes_AB$.
\end{proof}

\subsection{Constructing families of extensions}
\label{subsec:families}
%$G_K$-reps into $\GL_d(\Fpbar)$}
Our next goal is  to cover the underlying reduced substack
 $\cX_{d,\red}$ by certain families of \'etale  $(\varphi,\Gamma)$-modules,
having  rather  explicit parameterizations,
with the goal of showing  that~$\cX_{d,\red}$ is of finite presentation over~$\F_p,$ and
of bounding its dimension from above,   

We begin  with the case of $K = \Q_p$ and $d  = 1$,
and to this  end
%Let us
return briefly to the discussions of Lecture~\ref{subsec: FL vs
  WD}. %, and in particular to the case $K=\Qp$. %If furthermore~$d=1$, then
By local class field theory (see also Lecture~\ref{subsec: discussion of rank 1 case}), the $\Fpbartimes$-valued characters of
$G_K$ are all of the form $\omega^i\ur_\alpha$, for $i = 0, \dots p-2$
and $\alpha \in \Fpbartimes$. Going over to the setting
of~$(\varphi,\Gamma)$-modules, we can take the trivial
$(\varphi,\Gamma)$-module over $\bA_{\Qp, \F_p[\alpha,\alpha^{-1}]}$
and twist $\varphi$ by $\alpha$ and $\Gamma$ by $\omega^i$, to get a
$\bG_m$ worth of representations (one for each~$\alpha\in \Fbar_p^{\times}$; 
here $i$ is being kept fixed), but
also a $\bG_m$ worth of scalar automorphisms. As we run over~$i$ we obtain a map
	\[ \bigsqcup_{i=0}^{p-2} [\bG_m/\bG_m] \to \cX_{1,\red} \]
        which with a little work can be shown to be an isomorphism.

%Start with $K = \Qp$ and $d = 1$. Then by local class field theory the $\Fpbartimes$-valued characters of $G_K$ are all of the form $\omega^i\ur_\alpha$, for $i = 0, \dots p-2$ and $\alpha \in \Fpbartimes$. So if I take the trivial $(\varphi,\Gamma)$-module over $\bA'_{\Qp, \F_p[\alpha,\inv{\alpha}]}$ and twist $\varphi$ by $\alpha$ and $\wt{\Gamma}$ by $\omega^i$, we get a $\bG_m$ worth of representations, but also a $\bG_m$ worth of scalar automorphisms. In summary we obtain a map
%	\[ \bigsqcup_{i=0}^{p-2} [\bG_m/\bG_m] \to \cX_{1,\red} \] which is actually an isomorphism.

Now let $d = 2$.  Any  irreducible representation has the form
%If you're irreducible, then you are of the form (for some $i$ as above)
	\[ \ur_\alpha \Ind_{G_{\Q_{p^2}}}^{G_{\Qp}} \omega_2^i \]
(for  some $0<i<p^2-1$ with $(p+1)\nmid i$).
This gives us a point $[\bG_m/\bG_m] \hra \cX_{2,\red}$ (again the $\alpha$ gives a $\bG_m$ and then we quotient by scalars). But we know (looking  ahead to
Theorem~\ref{thm:reduced dimension})
% dimension computation
that $\dim \cX_{2,\red} = 1$ and that it's equidimensional, so the
irreducible representations lie in a substack of positive
codimension. We can now consider families of reducible
representations, so we want to look at representations coming from
$\Ext^1(\ur_\alpha\omega^i, \ur_\beta \omega^j)$. For generic choices
of~$\alpha,\beta$ there's a unique
nontrivial extension, and the intuition is that as $\alpha,\beta$ vary,
the extension class varies with them, and then you mod out by scalar endomorphisms
to get  the expected one-dimensional stack (two dimensions
coming from the choice of $\alpha,\beta$).

In fact, we want to apply this same technique in more generality. The idea is to
begin with some representation, and then to  iteratively build its spaces of extensions by some given irreducible representations. More generally, assume that we're given a family $\barr\rho_T \to T$ of rank $d$ \'etale $(\varphi,\Gamma)$-modules, where $T$ is a reduced finite type variety over $\Fpbar$. (This is our starting 
representation --- of course in practice we work with a $(\varphi,\Gamma)$-module,
but we think of it as being a family of Galois representations.)
Then take $\barr\sigma$ an irreducible Galois representation (or,
more properly, the associated \'etale $(\varphi,\Gamma)$-module).
As a technical hypothesis,
assume furthermore  that the spaces
$\Ext^2_{G_K}(\barr\sigma, \barr\rho_t)$ are of constant dimension when we vary~ $t$ over the closed points of~$T$.

The base change property for $\cC^\bullet(M)$ in Theorem~\ref{thm: properties of Herr complex} involves a derived tensor product, so there's a spectral sequence relating $\cC^\bullet(M)$ and $\cC^\bullet(M \otimes_A B)$, but since $H^2$ is the highest degree, it satisfies naive base change.
Therefore, our assumption  on the $\Ext^2$  having constant  degree
implies that  $H^2(\cC^\bullet(\barr\sigma^\vee \otimes \barr\rho_T))$ is actually a vector bundle over~$T$.
%(because the $\Ext^2$ have constant degree), which
This  implies that $\tau_{\leq 1} \cC^\bullet(\barr\sigma^\vee \otimes
\barr\rho_T)$ is still perfect, and can thus be modeled as a two term complex $[C^0 \to
Z^1]$ concentrated in degrees $0,1$, whose terms are finite locally free over
$T$. Note that we have a surjection $Z^1 \thra H^1(\cC^\bullet(\barr\sigma^\vee \otimes \barr\rho_T)) = \Ext^1(\barr\sigma, \barr\rho_T)$.

Now let $V$ be the vector bundle over $T$ corresponding to $Z^1$. Let
$\barr\rho_V$ be the pullback of $\barr\rho_T$ to $V$. By analogy with
the 2-dimensional case, we want to study a ``universal extension''
$\cE_V$ of $\barr{\rho}_V$ by $\barr{\sigma}$, which should fit in the
short exact sequence
	\[ 0 \to \barr\rho_V \to \cE_V \to \barr\sigma \to 0. \]
To construct this thing, we consider
	\[ \Ext^1(\barr\sigma, \barr\rho_T \otimes (Z^1)^\vee) = \Ext^1(\barr\sigma, \barr\rho_T) \otimes (Z^1)^\vee, \]
which is a quotient of
	\[ Z^1 \otimes (Z^1)^\vee .\]
This tensor  product  has a ``trace element'', whose image
in  $\Ext^1(\barr\sigma, \barr\rho_T \otimes (Z^1)^\vee)$ is an extension class
corresponding to a short exact sequence
	\[ 0 \to \barr\rho_T \otimes (Z^1)^\vee \to \text{ 
extension } \to \barr\sigma \to 0. \]
If we extend scalars from $\Fbar_p$  to  $\cO_V$, and then pushforward under the map
$$\rho_T \otimes_{\cO_T}  (Z^1)^{\vee} \otimes_{\F_p} \cO_V \to \rho_T\otimes_{\cO_T} \cO_V$$
induced by the inclusion $(Z^1)^{\vee} \hookrightarrow \Sym^{\bullet}  (Z^1)^{\vee}
=  \cO_V$ together  with the multiplicative structure of~$\cO_V$,
we obtain the desired extension
	\[ 0 \to \barr\rho_T \otimes_{\cO_T} \cO_V \to \cE_V \to 
\barr\sigma\otimes_{\Fbar_p}\cO_V \to 0. \]

Here  $\cE_V$ is now an  \'etale $(\varphi,\Gamma)$-module
over~$V$, and so is classified by a map $V \to \cX_{d,\red}$
(where $  d  = \dim \barr\rho_T + \dim \barr\sigma$). 
The object $\cE_V$ will admit various automorphisms, 
and  if we compute them precisely, we can hope to get an  embedding
of  a quotient  stack of $V$  into $\cX_{d,\red}$. 
Iterating this construction, we obtain families of \'etale
$(\varphi,\Gamma)$-modules parameterized by quotient  
stacks of the  total spaces of the various  vector  bundles $V$
appearing via applications of the preceding construction,
which will ultimately cover the various
$\cX_{d,\red}$.  (Since any Galois representation over $\Fbar_p$
can  be  written  as an iterated  extension of irreducible 
representations.)

\subsection{Explicit families and labelling by Serre
  weights}\label{subsec: explicit families and Serre weights}
The iterative construction that we just explained  ultimately allows us  to get  a description of the
underlying reduced stacks~$\cX_{d,\red}$ for  general~$d$.
% The point is that this will let us described the underlying reduced stack in general, which
Before stating our general result, we return to the case of~$K=\Qp$
and~$d=2$. %\begin{remark}
Again, while the stack $\cX_2$ is really a stack of $(\varphi,\Gamma)$-modules, % so to literally construct the geometric objects above, one has to construct a $(\varphi,\Gamma)$-module in some ring of coefficients, and then view it as a point of the underlying reduced stack $\cX_{2,\red}$. For
for the purposes of gaining intuition, it is reasonable to imagine %think of
our families of extensions as being %literal
extensions of Galois
representations, and we will describe things from this perspective %do so
without further comment for the rest of
this section. For simplicity we assume that~$p>2$ throughout this discussion.
%\end{remark}

%\mar{ME: Somewhere  add that $p > 2$ throughout.}

% So intuitively, a point of $V$ is a point of $T$ and a vector over it; the point should give us a mod $p$ Galois representation, and the vector specifies an extension class. So $\cE_V$ reflects this intuition, and glues all of these things together.

% So again return to the $K = \Qp$ and $d = 2$ case.
Take $T = \bG_m = \Spec \F_p[\alpha,\alpha^{-1}]$ and $\barr\rho_T =
\ur_\alpha \omega^i$. For the moment let $\barr\sigma~=~1$ (this will
suffice to describe all the geometric phenomena; twisting by a
1-dimensional family $\ur_\beta\omega^j$ for some $0\le j<p-1$ gives
the remaining cases). Fix some~$0\le i<p-1$. Then
	\[ \Ext^2(1,\ur_\alpha\omega^i) = H^2(G_{\Qp},\ur_\alpha\omega^i) \]
which (for example by Tate local duality) vanishes unless $i = 1$ and $\alpha = 1$, in which case it's 1-dimensional.

We begin with the case that~$2\le i\le p-2$. Then
		\[ \Ext^2(1,\ur_\alpha\omega^i) = 0, \]
	and
		\[ \Ext^0(1, \ur_\alpha\omega^i) = (\ur_\alpha\omega^i)^{G_{\Qp}} = 0 \]
	so actually (by Tate's local Euler characteristic formula) %\mar{TG: isn't it rather by the Euler characteristic formula?}
$\Ext^1(1, \ur_\alpha\omega^i)$ has constant rank $1$, so $\Ext^1(1, \barr\rho_T)$ is a line bundle. Its total space $V$ is then a copy of $\bA^1\times \Gm$
that parameterizes a family of extensions
		\[ \begin{pmatrix}\ur_\alpha\omega^i   & * \\ 0 & 1\end{pmatrix}. \]
We can  add one dimension to  this family by forming unramified twists  --- thus
obtaining a family parameterized by $V \times \bG_m = \bA^1\times \bG_m \times \bG_m$,
parameterizing extensions of the form
		\[ \begin{pmatrix}\ur_{\alpha\beta}\omega^i   & * \\ 0 & \ur_{\beta}\end{pmatrix}. \]
There is an action of (a different copy of) $\Gm\times \Gm$ on this family, arising
from the action of $\Gm$ by automorphisms
on each of the characters
$\ur_{\alpha\beta} \omega^i$, 
and~$\ur_{\beta}$. %and $\ur_{\alpha\beta} \omega^i$, 
If we let $(a,b)$ denote the coordinates on this copy $\Gm\times\Gm$ that is acting,
and let $x$  denote the coordinate on $\A^1$ (thus $x$ provides a coordinate
on each of the $1$-dimensional spaces $\Ext^1(\ur_{\beta},\ur_{\alpha\beta}\omega^i)$)
--- so that the coordinates
on $V  = \A^1\times \bG_m\times \bG_m$ are~$(x,\alpha,\beta)$ --- then this action is given
by the formula
$(a,b) \cdot (x,\alpha,\beta) = (a b^{-1} x, \alpha,\beta).$
(The fact that the stabilizer of any particular  point  $(x,\alpha,\beta)$
is the diagonal copy of $\Gm$ if $x \neq 0$, and the entirety of $\Gm\times\Gm$
when $x = 0$,
is a manifestation of the fact that the only automorphisms of a non-split extension
of distinct characters 
are the scalar  matrices, while the
automorphisms of a split extension (i.e.\ a direct sum) of distinct characters
are the diagonal matrices.)
Thus we get  an embedding
$$[ (\bA^1\times\bG_m\times \bG_m)/(\bG_m\times \bG_m)]
\hookrightarrow \cX_{2,\red},$$
whose source visibly has dimension $3 - 2 = 1$.
%
%
%	and these have only scalar automorphisms. Therefore, we get a
%        family in $\cX_{2,\red}$ that looks like\mar{TG: make it
%          clearer what the action of~$\Gm$ is? ME: Really we get 
%$[(\bA^1  \times \bG_m \times \bG_m)/(\bG_m \times \bG_m)]$, so maybe better to
%just put that here (and refer back to the  unramified twist that was mentioned,
%but then suppressed, earlier)?}
%		\[ [\bA^1 \times \bG_m/\bG_m] \]
%	which visibly has dimension $3 - 2  = 1$.
%
%

We next turn to the case $i = 0$, where Tate local duality gives us that
		\[ \Ext^2(1, \ur_\alpha) = H^2(G_{\Qp}, \ur_\alpha) = H^0(G_{\Qp}, \ur_\alpha \otimes \epsilon^{-1}) = 0 \]
	for all $\alpha \in \Fpbartimes$. However, we now run into the
        issue that $\Ext^0$ is not a vector bundle, because it has
        rank $0$ everywhere except for when $\alpha = 1$, and at this
        point the rank jumps from $0$ to $1$;
correspondingly, the rank of $\Ext^1(1,\ur_{\alpha})$ jumps
from~$1$ to~$2$.
        In the framework of Lecture~\ref{subsec:families},
this means that we can't choose the complex $C^0 \to Z^1$ so that
$C^0 = 0$ and $Z^1 = \Ext^1$; we have to instead allow $C^0$ to have rank~$1$,
and $Z^1$ to have rank $2$.
If we go  through the construction of that subsection,
we end up with a family parameterized by  the total space 
$V$ of a rank~$2$ vector  bundle on~$\Gm$.   We can add in the
unramified twists, to obtain a family parameterized by~$V\times \bG_m$,
and we obtain 
a map
\numequation
\label{eqn:i = 0 case}
V\times\bG_m \to \cX_{2,\red}
\end{equation}
classifying the 
family of Galois representations that  $V\times \bG_m$  parameterizes. 
This map won't be an embedding,
of course, because various different points on $V\times \bG_m$
will give rise to isomorphic Galois representations.  
Just as in the case $2 \leq i \leq p-2$ considered above,
there is a scaling action of $\Gm \times \Gm$ on the various $\Ext^1$ spaces.
Also, any two points in of $V$  that lie over the same point
$\alpha \in \Gm(\Fbar_p)$,
and which lie in the same fibre of the surjection 
\numequation
\label{eqn: Z to  Ext}
Z^1_{\alpha} \to \Ext^1( 1,\ur_{\alpha})
\end{equation}
 (here we write $Z^1_{\alpha}$ to denote the fibre
of $Z^1$ over the point~$\alpha$) parameterize the same extension
of $1$ by~$\ur_\alpha$. 
And there are even more isomorphisms that can occur between the
representations parameterized by this family: for example,
if we restrict to the family of split extensions $\ur_{\alpha\beta} \oplus
\ur_\beta$ parameterized by $\bG_m\times \bG_m$ (the ``zero section''
of $V\times \bG_m$, which we also refer to as the ``split locus''
of the family), then the points
$(\alpha, \beta)$ and $(\alpha^{-1}, \alpha\beta)$
both parameterize the direct sum
$\ur_{\alpha\beta} \oplus \ur_{\beta}$.   Since these ``extra isomorphisms''
occur only on a proper subvariety of~$V\times\bG_m$, though, they don't play a 
role in determining the dimension  of the image of~\eqref{eqn:i = 0 case}.
Indeed, since 
the kernel of the surjection~\eqref{eqn: Z  to Ext}
generically (more precisely, whenever $\alpha \neq 1$) has dimension~$1$,
we find that the image of~\eqref{eqn:i  = 0 case}
is an irreducible constructible substack of $\cX_{2,\red}$ of
dimension $4 - 2 -  1 = 1$.
The structure of this substack along the image of the split locus,
as well as along the image
locus where~$\alpha = 1$,
is rather complicated; for example,
along the split locus we get 
a fold singularity (corresponding to the identification of pairs 
of points 
$(\alpha, \beta)$ and~$(\alpha^{-1}, \alpha\beta)$),
degenerating to some kind of
cusp singularities along the points where furthermore $\alpha =
1$. (See~\cite{sandermultiplicities} for some computations of the
versal rings at these points.)
%\mar{ME: Could cite  Fabian Sanders here?}
%
%
%What this means is that we 
%        we get a family in $\cX_{2,\red}$ which looks like a sort of
%        weird cusp-y thing (???)\mar{TG: this needs some elaboration}

Finally, in the case $i = 1$, % (and assuming that $p > 2$, so that $1 \leq p-2$),
Tate local duality gives
		\[ \dim \Ext^2(1, \ur_\alpha\omega) = \dim
                  H^2(G_{\Qp}, \ur_\alpha\omega) = \dim H^0(G_{\Qp},
                  \ur_\alpha) = \begin{cases} 1 & \alpha = 1 \\ 0 & \alpha \neq 1 \end{cases} \]
	so $\Ext^2$ is no longer a vector bundle. On the  other hand,  we
have
		\[ H^0(G_{\Qp}, \ur_\alpha\omega) = 0. \]
Following  the strategy of Lecture~\ref{subsec:families},
we thus break $\Gm$ up into two pieces: $T_0 := \Gm \setminus\set{\alpha = 1 }$,
and $T_1 := \{\alpha = 1\}$.  Over $T_0$,
we see that $\Ext^2(1,\ur_\alpha\omega)$
has constant rank~$0$, and we can construct a family analogously to the
case $2  \leq i \leq p-1$, parameterized by the total space
of a line bundle over~$\Gm\setminus \set{\alpha = 1}$.  Adding
in the unramified twists, this leads to an embedding
$$[(\A^1 \times (\Gm\setminus \set{\alpha = 1} ) \times  \Gm) / (\Gm\times \Gm)]
\hookrightarrow \cX_{2,\red}.$$

Passing now to $T_1$ (which is just a single  point! --- over which
 $\Ext^2(1,\ur_\alpha\omega)$  has constant rank~$1$),
we obtain a family parameterized by
the $2$-dimensional vector space $\Ext^1(1,\omega).$
Again adding in unramified twists, we obtain
an embedding
$$[(\A^2 \times \Gm)/ (\Gm\times \Gm)] \hookrightarrow \cX_{2,\red}.$$
%	Geometrically, this means that on $\bG_m \setminus \set{1}$, we get $[\bA^1 \times (\bG_m \setminus \set{1})/\bG_m]$ and on $\set{1}$ we get $[\bA^2/\bG_m]$. So in the end, the pictures fit together at $\alpha = 1$ and we get two planes intersecting each other transversally in a line.

Each of the finite number of families that we have just described
(along with the families obtained by taking a  twist  of one of these
families by $\omega^j$ for some $j = 1,\ldots,p-2$)
gives rise to a $1$-dimensional irreducible  constructible substack of~$\cX_{2,\red}$. 
By construction, every $\Fbar_p$-point of $\cX_{2,\red}$  that
corresponds to a {\em reducible} Galois representation lies exactly
one of these subsets.

We have already seen that the $\Fbar_p$-points corresponding to {\em irreducible}
Galois representations lie in a finite union of images of the $0$-dimensional
stack $[\Gm/\Gm]$.   We thus deduce that $\cX_{2,\red}$ is $1$-dimensional,  
and that the closure of each of the $1$-dimensional and mutually disjoint
constructible substacks that we've just constructed must be an irreducible
component of $\cX_{2,\red}$.
%
%\mar{TG: join up}
We now explain how we can label these
% that 
%^irreducible representations of~$\cX_{2,\red}$. We can label these
components by so-called \emph{Serre weights}. % ; we now briefly digress
% to explain the definition of a Serre weight in general, before
% returning to the case at hand.

Temporarily return to the case of
general~$K$, $d$, and let~$k$ be the residue field of~$K$. Then by
definition, a Serre weight is an (isomorphism class of)
irreducible $\Fpbar$-representations of~$\GL_n(k)$. These are
determined by their highest weights, which are tuples of integers
~$\{k_{\sigmabar,i}\}_{\sigmabar:k\into\Fpbar,1\le i\le d}$
with the properties that \begin{itemize}
\item $p-1\ge k_{\sigmabar,i}-k_{\sigmabar,i+1}\ge 0$ for each $1\le i\le
  d-1$, and
\item $p-1\ge k_{\sigmabar,d}\ge 0$, and not every~$k_{\sigmabar,d}$
  is equal to~$p-1$.
\end{itemize} We refer the reader to Lecture~\ref{sec: geometric BM}
for some explanations and motivation for the appearance of the
representation theory of~$\GL_n(k)$ in the geometry of our stacks.

Returning to the case that~$K=\Qp$ and~$d=2$, the Serre weights are
the representations $\det^{j}\Sym^{i-1}\overline{\F}_p^2$ with $0\le j<p-1$, $1\le
i\le p$; in the notation above, these have highest
weight~$(k_1,k_2)=(j+i-1,j)$. % where~$k$ is the
% residue field of~$K$):\mar{TG: probably not where this definition
%   should be, and probably want a forward citation for why Serre
%   weights have anything to do with anything.}  \mar{TG: explain
%   firstly as a labelling by pairs of integers, and then we can explain
%   this statement in general in the following theorem?}
Our labelling of components by Serre weights is then as follows.\footnote{For
accuracy, in case the reader tries to carefully compare our present
discussion with the corresponding discussion in~\cite{emertongeepicture},
we note that our labelling here differs from the one given in that reference;
the component that here we label by the Serre weight~$\sigma$
will there be labelled by the Serre weight $\sigma^{\vee}\otimes \det^{-1}$.
We have chosen our present convention just because it is notationally
easier to work with extensions of $1$ by powers of $\omega$ rather than 
with extensions of powers of $\omega^{-1}$ by~$1$.  
If you like, we have applied the ``Cartier  duality'' involution
--- given by $\rhobar \mapsto \rhobar^{\vee} \otimes \omega$ ---
to $\cX_{2,\red}$.}
An
irreducible component which generically parameterizes representations of
the form \[\begin{pmatrix}\ur_{\alpha\beta}\omega^{i+j}   & * \\ 0 &
    \ur_{\beta}\omega^{j}\end{pmatrix} \] is labelled by the highest
weight~$(i+j-1,j)$ (that is, by the representation
$\det^j\Sym^{i-1}\overline{\F}_p^2$). This is unambiguous except for the
possibilities $i=1,p$, which we distinguish as follows. In the
discussion above, we saw that if~$i=1$ there is one irreducible
component for which~$\alpha=1$ at a dense set of points, and we label
this component by~$\det^j\Sym^{p-1}\overline{\F}_p^2$; and there is one
component where~$\alpha\ne 1$ at a dense set of points, which we label by~$\det^j\Sym^{0}\overline{\F}_p^2$.

%(the duals are an
%artifact of a choice of ordering earlier; \mar{TG: does this actually
%  matter? We could just avoid it altogether by not taking duals here?!
%I guess if we did so we might want to say that we aren't using the
%same conventions as \cite{emertongeepicture}.}

% \begin{itemize}
% 	\item When $i = 2, \dots, p-2$, we assign the label $(\Sym^{i-1}\Fpbar^2)^\vee$.
% 	\item When $i = 0$, we assign the label $(\Sym^{p-2} \Fpbar^2)^\vee$
% 	\item When $i = 1$ we assign the label
%           $(\Sym^{p-1}\Fpbar^2)^\vee$.\mar{TG: aren't there two
%             components here, and one is~$\Sym^0$?}
% \end{itemize}\mar{TG: mention determinant twists}

We can inductively extend the the construction technique
that we have just described for~$d=2$, so as to prove
the following general theorem. The notion of being ``generically maximally nonsplit of niveau
  one'' means that for an open set of
  points, the corresponding $G_K$-representation is a successive
  extension of characters in a unique way; to such a representation
  one can associate a Serre weight~$\underline{k}$, via a
  straightforward extension of the labelling we just explained in the
  case~$K=\Qp$ and~$d=2$. % \mar{TG: need to phrase this more
  % cleanly, and should say that we prove the ``no small components''
  % statement in the following lecture.}
% \begin{theorem}\label{thm: main result on dimension of algebraic stack
%   and irred components and so on}\mar{TG: replace with a cut and paste
% from moduli}
% 	\[ \cX_{d,\red} = \left( \bigcup_\text{Serre weights} \text{closure of an irred. niveau 1 family of dim $[K:\Qp]d(d-1)/2$ } \right) \cup \text{ lower dim things} \]
% In fact, there are no lower dimensional things, so really
% 	\[ \cX_{d,\red} = \bigcup_\text{Serre weights} \text{closure of an irred. niveau 1 family of dim $[K:\Qp]d(d-1)/2$ } \]
%       \end{theorem}

      \begin{thm} 
	\label{thm:Xdred is algebraic}\leavevmode
        \begin{enumerate}
        \item 	The Ind-algebraic stack $\cX_{d,\red}$
is an algebraic stack, of finite presentation over $\F$.
\item  We can write~$(\cX_{d,\red})_{\Fpbar}$ 
% \mar{ME: Any reason not to write $\cX_{d,\red,\Fbar_p}$ instead, to fit
% with the notation for the irred.\ comps.?}
as a union of closed algebraic substacks of finite
presentation
over~$\Fpbar$ 
% \mar{TG: stupidity: can we descend the small guy, or should
  % this result just be over~$\Fpbar$, and then the descent can happen
  % at the very end? Or we could descend the weight~$\underline{k}$
  % stacks now but not the small one.}
\[(\cX_{d,\red})_{\Fpbar}=\cX_{d,\red,\Fpbar}^{\negligible}\cup\bigcup_{\underline{k}}\cX_{d,\red,\Fpbar}^{\underline{k}},\]
where:
\begin{itemize}
\item $\cX_{d,\red,\Fpbar}^{\negligible}$ is empty if~$d=1$, and otherwise is
  non-empty of dimension
  strictly less than~$[K:\Qp]d(d-1)/2$. %, and is preserved under twist,
\item each $\cX_{d,\red,\Fpbar}^{\underline{k}}$ is a closed 
irreducible substack of dimension 
% \mar{ME: Added {\em absolutely} to the irreducibility statement here, should address
% it in the proof too. TG: changed this to ``geometric'', which I think
% is what you mean? Wasn't sure how you meant to prove it (I think it
% just follows from the compatibility of our constructions with
% extension of scalars?), but I added a
% sentence below.}
  ~$[K:\Qp]d(d-1)/2$, and is generically maximally nonsplit of niveau
  one and weight~$\underline{k}$.%\mar{TG: explain what this means!}
\end{itemize}

% \item  If we fix an irreducible representation
% $\alphabar: G_K \to \GL_a(\Fbar_p)$ {\em (}for some $a \geq 1${\em )},
% then the locus of $\rhobar$ in
% $\cX_{d,\red}(\Fbar_p)$ for which
% $\dim \Hom_{G_K}(\rhobar,\alphabar) \geq r$ {\em (}for any $r \geq 1${\em )}
% is {\em (}either empty,
% or{\em )} of dimension at most
% \[[K:\Qp]d(d-1)/2- \lceil {r\bigl((a^2+1)r-a\bigr)}/{2} \rceil.\] % \mar{TG: perhaps we
%   % should be more precise about what this means, e.g.\ it could be
%   % defined and checked on pullbacks to schemes?}
% Furthermore, the locus
% of $\rhobar$ in
% $\cX_{d,\red,\Fpbar}^{\negligible}(\Fbar_p)$ for which
% $\dim \Hom_{G_K}(\rhobar,\alphabar) \geq r$ is of dimension strictly
% less than this.
% \mar{TG: think we need this for the induction.}

\item  % If we fix an irreducible representation
% $\alphabar: G_K \to \GL_a(\Fbar_p)$, then t
If we fix an irreducible representation
$\alphabar: G_K \to \GL_a(\Fbar_p)$ {\em (}for some $a \geq 1${\em )},
then the locus of $\rhobar$ in
$\cX_{d,\red}(\Fbar_p)$ for which
$\dim \Ext^2_{G_K}(\alphabar,\rhobar) \geq r$ is of dimension at most
\[[K:\Qp]d(d-1)/2-r.\]%  \mar{TG: also add $H^2$
% statement explicitly}
        \end{enumerate}
\end{thm}
\begin{rem}
  Note that in part~(3), the locus of points in question
  corresponds to a closed substack of~$(\cX_{d,\red})_\Fpbar$, by
  upper-semicontinuity of fibre dimension. 
\end{rem}

\begin{rem}\label{rem: different k give different components}
  Since the~$\cX_{d,\red,\Fpbar}^{\underline{k}}$ are irreducible, have
  dimension equal to that of~$(\cX_{d,\red})_\Fpbar$, and 
  have pairwise disjoint open substacks (corresponding to maximally
  nonsplit representations of niveau~$1$ and weight~$\underline{k}$),
  they are in fact distinct irreducible components of~$(\cX_{d,\red})_\Fpbar$.
\end{rem}

\begin{remark}
	We can, and will, be much more precise about the structure of $\cX_{d,\red}$.
	% Namely, we are able to show that it is equidimensional, of
	% dimension $[K:\Q_p]d(d-1)/2$, and also to
	% enumerate its irreducible components.
	% In Subsection~\ref{subsec: dimensions of families} 
	% we begin our preparations for this, by studying more carefully
	% the dimensions of the various families of extensions constructed
	% in the proof of Proposition~\ref{prop:Xdred is algebraic}.  
	Namely, in Lecture~\ref{sec: geometric BM} we combine
        Theorem~\ref{thm:Xdred is algebraic} with the results of Lecture~\ref{sec: Extensions and Crystalline
  Lifts}
        to show that $\cX_{d,\red}$ is equidimensional of dimension
        $[K:\Q_p]d(d-1)/2$; % so that  we may
        % take~$\cX_{d,\red}^{\negligible}=0$; 
        accordingly, the
        irreducible components of~$(\cX_{d,\red})_\Fpbar$ are
        precisely the
        ~$\cX_{d,\red,\Fpbar}^{\underline{k}}$, and in particular
        are in bijection with the Serre
        weights~$\underline{k}$. % We also show that these irreducible
        % components are all defined over~$\F$.
        % to compute its dimension, and to enumerate
        % its components.
        % \mar{ME: Have to remember
% 		to add this last result to
%                 Section~\ref{sec:properties}. TG: do we really
%                 enumerate its components? Maybe we do, but this seems
%                 a little painful even conjecturally, keeping track of
%                 the tres ram extensions.}
      \end{remark}
      % \begin{rem}\label{rem: open eigenvalue morphism implies ratio
      %     not cyclo}
      %   Note that since the eigenvalue morphism has dense image, it
      %   follows that~$\cX_{d,\red}^{\underline{k}}$ contains a dense
      %   open substack with the property that $\nu_i(t)\ne
      %   \nu_{i+1}(t)$ unless $k_{\sigmabar,i}-k_{\sigmabar,i+1}=p-1$
      %   for all~$\sigmabar$.
      % \end{rem}
	\begin{remark}
          The upper bound of
          Theorem~\ref{thm:Xdred is algebraic}~(3) is quite crude when
          $[K:\Q_p] > 1$, although it is reasonably sharp in the case
          $K = \Q_p$.  However, it suffices for our purposes. (In
          fact, in the inductive proof of Theorem~\ref{thm:Xdred is
            algebraic} we use a slightly more complicated upper bound
          to control the dimensions of various extension groups, but
          the stated bound is all that we will need going forward.) % , and
          % indeed we will only use the (even cruder) consequence Theorem~\ref{thm:Xdred is
          %   algebraic}~(4) in the rest of the paper.
          % for
          % which it is sufficient to have an upper bound of
          % $[K:\Qp]d(d-1)/2-r$ on the dimension as in
          % , an upper bound which
          % is certainly satisfied by
          % $[K:\Qp]d(d-1)/2-\lceil
          % r\bigl((a^2+1)r-a\bigr)/2\rceil$.
          % \mar{TG: this is only true
%                 if $r>1$! Hopefully not an issue? Earlier you were saying~$r$\dots ME: Yes, sorry, we need $r+1$ when we work over $\Z_p$, and (equivalently)
% 	$r$ when we work on the special fibre.  I think this same
% confusion (on my part) is underlying your comment above; I'll work through
% the argument and try to sort this out.}
	\end{remark}

\section{Crystalline Lifts}\label{sec: Extensions and Crystalline
  Lifts}        \subsection{Crystalline lifts}% In Lecture~\ref{sec: geometric BM} we will explain how to refine\mar{TG:
%   dimension statement to precise description.} This refinement will use the
% crystalline moduli stacks defined in Lecture~\ref{sec: crystalline and
%   semistable}, in combination with
The following theorem is the
main result of this lecture.

\begin{theorem}\label{thm:  main crystalline lifts}
If $\barr\rho: G_K \to \GL_d(\Fpbar)$ is continuous, then there exists
$\rho^\circ: G_K \to \GL_d(\Zpbar)$ lifting $\barr\rho$ such that the
corresponding $p$-adic Galois representation $\rho: G_K \to
\GL_d(\Qpbar)$ is crystalline with regular Hodge-Tate weights. 
\end{theorem}
\begin{rem}
In fact, the lift~$\rho^\circ$ in Theorem~\ref{thm:  main crystalline
  lifts} can be chosen to be potentially diagonalizable in the sense
of~\cite{BLGGT}, which implies in turn that if $p \nmid 2d$, then
$\barr\rho$ can be globalized to come from an automorphic form
(see~\cite[Thm.\ 1.2.3]{emertongeepicture}).
\end{rem}

While the statement of Theorem~\ref{thm: main crystalline lifts} is
purely local (in the sense that the representation~$\rhobar$ is
fixed), and our proof is for the most part via a local argument, we
will also make crucial use of Theorem~\ref{thm:Xdred is algebraic}.%\mar{TG: correct reference in the end?}

% \begin{theorem}\label{thm: crystalline lifts}
% Suppose we're given $0 \to \barr\rho_d \to \barr\rho_{d+a} \to \barr\sigma \to 0$ where $\barr\sigma$ is irreducible of dimension $a$ and suppose $X$ is a non-empty component of a crystalline deformation ring for $\barr\rho_d$. Let $\sigma$ be a crystalline lift of $\barr\sigma$ chosen so that all extensions of $\sigma$ by points of $X$ are crystalline: in other words, you need to choose the Hodge-Tate weights of $\sigma$ to be suitably spaced with respect to the Hodge-Tate weights of $x \in X$.

% Then there exists some $\rho^0_d \in X(\Zpbar)$ such that there exists an extension
% 	\[ 0 \to \rho_d^\circ \to \rho_{d+a}^\circ \to \sigma \to 0 \]
% lifting the original sequence.
% \end{theorem}

% The dimension argument (i.e. showing that there are no smaller dimensional components in the reduced substack) relies on the existence of crystalline lifts. We use the fact that the crystalline stack is $p$-adic to deduce that if $X$ is a component of a crystalline deformation ring, then the codimension of
% 	\[ \set{x \in X \mid H^2(\rho_x) \geq r} \]
%         (which is Zariski closed) is at least $r + 1$.

We begin by explaining the ``obvious''
 strategy to prove Theorem~\ref{thm: main crystalline
   lifts}. Firstly consider the case
 that~$\rhobar:G_K\to\GL_d(\Fpbar)$ is irreducible. Then it is easy to
 show that~$\rhobar$ is induced from a character of the unramified
 extension~$K'/K$ of degree~$d$, and furthermore it is easy to show
 that this character can be lifted to a crystalline character, and
 because~$K'/K$ is unramified, the induction of a crystalline
 character of~$G_{K'}$ to~$G_K$ gives a crystalline
 representation. Furthermore, one has considerable control over the
 Hodge--Tate weights of such a lift.

 An immediate consequence is that all semisimple~$\rhobar$ have
 crystalline lifts. However, not all mod~$p$ representations are
 semisimple, so the remaining problem is to show that we can lift
 extensions of representations. To get a feeling
 for this, consider the two-dimensional
 case, and so suppose that $\rhobar:G_K\to\GL_2(\Fpbar)$ is of the form \[
   \begin{pmatrix}
     \chibar_1 &*\\0&\chibar_2
   \end{pmatrix}.\] Choose, as we may, crystalline lifts~$\chi_1:G_K\to\Zpbartimes$, $\chi_2:G_K\to\Zpbartimes$
 of~$\chibar_1$, $\chibar_2$ respectively, and suppose that the labelled
 Hodge--Tate weights of~$\chi_1$ are \emph{slightly greater} than those of~$\chi_2$,
in the sense that for each embedding $K\into\Qpbar$ the corresponding
Hodge--Tate weight of~$\chi_1$ is greater than that of~$\chi_2$, and
for at least one embedding, the gap is greater than~$1$. Then any
 extension of~$\chi_2$ by~$\chi_1$ is automatically crystalline (for
 example by  the formulae in~\cite[Prop.\
  1.24]{MR1263527}). Since
 extensions are classified by~$H^1$, it is therefore enough to show
 that given any class in~$H^1(G_K,\chibar_1\chibar_2^{-1})$, we can
 choose~$\chi_1$, $\chi_2$ in such a way that this class can be lifted
 to~$H^1(G_K,\chi_1\chi_2^{-1})$.

 As we will soon see, this can always be arranged, and in ``most''
 cases the natural map $H^1(G_K,\chi_1\chi_2^{-1})\to
 H^1(G_K,\chibar_1\chibar_2^{-1})$ is actually surjective. It is then
 natural to imagine that one could inductively prove the existence of
 lifts in all dimensions by induction on the number of
 Jordan--H\"older factors of~$\rhobar$. Quite a few people have tried
 to do this, but no-one succeeded, as far as we know. The basic problem
 seems to be that it is in general not possible to fix the lifts of
 each Jordan--H\"older factor when arguing inductively; rather, one
 keeps having to go back and adjust the previous choices in light of
 the next step.

 For concreteness, here
 are some examples (due to Tong Liu) which are difficult to do by hand, and give  a
 sense of why it's hard to work inductively. For the rest of the
 lecture we write~$\varepsilon$ for the cyclotomic character, and~$\varepsilonbar$ for the
 mod~$p$ cyclotomic character. Firstly, consider

 \numequation\label{eqn: Tong W}\overline{W}=
   \begin{pmatrix}
     \varepsilonbar  &*  &*  &* \\
&   \varepsilonbar &0  &*\\
&    & \varepsilonbar &*\\
&    &   &  1
   \end{pmatrix}
\end{equation}where~$*$ could be zero or not. It is quite easy to convince
yourself that there are enough ``degrees of freedom'' (by choosing
for example different unramified twists of the characters that you
lift to, and lifting to different extension classes) to manage to make
the lift, but quite tricky to come up with an actual argument. If that
one is too easy, you can try something like
\[
  \begin{pmatrix}
    \varepsilonbar^2  &*    &*   &*\\
&    \varepsilonbar^2 &0     &*\\
 &&   \varepsilonbar^2 &*\\
   &&& \overline{W}
  \end{pmatrix}.
\]

 The key to our approach is the following theorem. % Here ``slightly
 % less than'' is the condition above, which guarantees that the
 % extension is crystalline.
 Note that when we use this theorem inductively, we are not
 guaranteeing that you can fix lifts of each Jordan--H\"older factor
 once and for all before considering the various possible extension
 classes, and this increased flexibility is important in our argument.
 \begin{theorem}
	\label{thm:crystalline lifts}
	Suppose given a representation $\rhobar_d: G_K \to \GL_d(\Fbar_p)$
	that admits a lift
	$\rho_d^{\circ} : G_K \to \GL_d(\Zbar_p)$ which 
%	is a lift of $\rhobar_d$ such that
        % for which the associated
	% $p$-adic representation $\rho_d: G_K \to \GL_d(\Qbar_p)$
	is crystalline with labelled Hodge--Tate weights
        $\underline{\lambda}$. 
% \mar{ME: Add condition on HT wts.\ that ensure $H^2
%           = 0$. TG: I think that since we changed from the trivial
%           representation, we can't just write a condition like
%           this. Instead I've baked in crystalline conditions for both
%           factors, and an inequality on the weights. Maybe we should
%           do something a bit more general, but I'm inclined to stick
%           with at least potentially semistable with weights satisfying
%         this condition?} 
	% \mar{ME: Ultimately need to generalize from extension by $\triv$
	% 	to an extension by an irreducible rep.\ of arb.\ dim'n. 
	% ME: I've now done this (hopefully without blundering).}
	Let $0 \to \rhobar_d \to \rhobar_{d+a} \to \alphabar \to 0$
	be any extension of $G_K$-representations over $\Fbar_p$,
	with $\alphabar: G_K \to \GL_a(\Fbar_p)$ irreducible,
	and let $\alpha^{\circ}: G_K \to \GL_a(\Zbar_p)$ be any
        crystalline lifting
	of $\alphabar$ with labelled Hodge--Tate
        weights~$\underline{\lambda}'$, which we assume to be slightly
        less than~$\underline{\lambda}$.

	Then we may find a lifting of the given extension to an extension
	$$0 \to \theta^{\circ}_d \to \theta_{d+a}^{\circ} \to \alpha^{\circ} \to 0$$
	of $G_K$-representations over $\Zbar_p$, where
        $\theta^{\circ}_d: G_K \to \GL_d(\Zbar_p)$ again has the
        property that the associated $p$-adic representation
        $\theta_d: G_K \to \GL_d(\Qbar_p)$ is crystalline with
        labelled Hodge--Tate weights $\underline{\lambda}$.
        Furthermore, $\theta_{d+a}^{\circ}$ is crystalline, and we may
        choose $\theta_d^{\circ}$ to lie on the same irreducible
        component of
        $\Spec R_{\rhobar_d}^{\underline{\lambda}, \crys}$ that
        $\rho^{\circ}_d$ does.
	% \mar{ME: We would like to actually have some relationship b/w properties
	% 	of $\theta_d^{\circ}$ and $\rho^{\circ}_d$, e.g.\ the former can
	% 	be taken to be niveau $1$ if the latter is.   But I don't know
	% 	how to do this yet.}
      \end{theorem}
      Given Theorem~\ref{thm:crystalline lifts}, it is easy to prove
      Theorem~\ref{thm: main crystalline lifts} by induction
      on~$d$. In fact, without much additional difficulty one can
      prove the following stronger version, which is useful in
      applications.
\begin{thm}
  \label{thm: strong existence of crystalline lifts}Let $K/\Qp$ be a
  finite extension, and let~$\rhobar:G_K\to\GL_d(\Fpbar)$ be a
  continuous representation. Then~$\rhobar$ admits a lift to a
  crystalline representation~$\rho^\circ:G_K\to\GL_d(\Zpbar)$ of some
  regular labelled Hodge--Tate weights~$\underline{\lambda}$. Furthermore:
  \begin{enumerate}
  \item $\rho^\circ$ can be taken to be potentially diagonalizable.
  \item If every Jordan--H\"older factor of~$\rhobar$ is
    one-dimensional, then~$\rho^\circ$ can be taken to be
    ordinary.%\mar{TG: add defn}
  \item $\rho^\circ$ can be taken to be potentially diagonalizable,
    and~$\underline{\lambda}$ can be taken to be a lift of a Serre weight.
  \item $\rho^\circ$ can be taken to be potentially diagonalizable,
    and~$\underline{\lambda}$ can be taken to have arbitrarily spread-out
    Hodge--Tate weights: that is, for any $C>0$, we can
    choose~$\rho^\circ$ such that for each~$\sigma:K\into\Qpbar$, we
    have $\lambda_{\sigma,i}-\lambda_{\sigma,i+1}\ge C$ for each $1\le i\le n-1$.
  \end{enumerate}

\end{thm}% The notion of a potentially diagonalizable representation was defined
% in~\cite{BLGGT}, and is important for applications to automorphy
% lifting theorems; for example, Theorem~\ref{thm: strong existence of crystalline
%   lifts} allows us to remove a hypothesis on the construction of a
% candidate $p$-adic Langlands correspondence in~\cite{Gpatch}. In
% addition to proving that the Hodge--Tate weights can be taken  to be in the
%         interval~$[0,dp-1]$, 
% we are 
% able to prove a conjecture of~\cite{2015arXiv150902527G} on the
% existence of a crystalline lift whose Hodge--Tate weights correspond
% to a Serre weight.
% \begin{rem}One can now deduce Theorem~\ref{thm: intro theorem about X}
% from Theorem~\ref{thm: weak existence of crystalline lifts} and
% Theorem~\ref{thm: effectivity of crystalline deformation rings};  the
% point is that we can now prove that every irreducible component
% of~$\cX_d$ had dimension at least~$[K:\Qp]d(d-1)/2$, and thus exactly ~$[K:\Qp]d(d-1)/2$.
% \end{rem}

%It remains to explain something about the proof of
We now discuss the proof of Theorem~\ref{thm:crystalline lifts}. % , and what this has to do with
% Theorem~\ref{thm: codim bound crystalline deformation ring}.
Return to the two-dimensional case, and
write~$\chibar=\chibar_1\chibar_2^{-1}$, $\chi=\chi_1\chi_2^{-1}$, so
that we want to lift classes in~$H^1(G_K,\chibar)$
to~$H^1(G_K,\chi)$. Fix a sufficiently large finite extension~$E/\Qp$ with ring of
integers~$\cO$, uniformizer~$\varpi$, and residue field~$\F$, so that
in particular~$\chi$ takes values in~$\cO^\times$. % , and
% assume that~$E$ is large enough to contain the images of all
% embeddings $K\into\Qpbar$. \mar{TG: temp}
% Suppose
% that~$\chi:G_K\to\cO^\times$; then 
Taking cohomology
of
\[0\to\cO(\chi)\stackrel{\varpi}{\to}\cO(\chi)\to\F(\chibar)\to 0,\]we
see that we have an exact sequence
\[H^1(G_K,\chi)\otimes_{\cO}\F\to H^1(G_K,\chibar)\to
  H^2(G_K,\chi)[\varpi].\] By Tate local duality, $H^2(G_K,\chi)$ is
dual to $H^0(G_K,(E/\cO)(\chi^{-1})(1))$, with the~$(1)$ denoting a
Tate twist, so we see in particular that
if~$\chibar\ne\varepsilonbar$, then there is nothing to
prove. However, if~$\chibar=\varepsilonbar$, then there \emph{is} some
work to do.

One thing that we could do in this case is to take~$\chi=\varepsilon$, in
which case~$H^2(G_K,\chi)[\varpi]$ vanishes, and we see that we can
lift to semistable (not necessarily crystalline) extensions. One
might wonder if this means that in general we should just relax things
and try to only produce semistable lifts, rather than crystalline
ones, but in practice this doesn't seem to be the case - the inductive
arguments still run into trouble on examples like~\eqref{eqn: Tong W}
above. Indeed, let us try to handle~\eqref{eqn: Tong W} by lifting
inductively, allowing ourselves to use semistable lifts. The
bottom right $3\times 3$ matrix is a sum of two extensions
with~$\chibar=\varepsilonbar$, so we've just seen that we can lift it to
a representation~$U$ of the form \[U=
  \begin{pmatrix}
    \varepsilon & 0&*\\ &\varepsilon&*\\&&1
  \end{pmatrix}.
\]Then~\[\overline{W}=
  \begin{pmatrix}
    \varepsilonbar &*\\&\overline{U}
  \end{pmatrix},\]so we could try to consider lifts of the form \[W= \begin{pmatrix}
    \varepsilon &*\\&U
  \end{pmatrix}. \]However, if the classes giving the two extensions
of~$\varepsilonbar$ by~$\varepsilonbar$ aren't twists by~$\varepsilonbar$ of
the unramified extension of the trivial character by itself, then
these lifts won't be semistable (because the only semistable
extensions of the trivial character by itself are unramified). So we
are forced instead to consider something like \[W= \begin{pmatrix}
    \varepsilon^p &*\\&U
  \end{pmatrix}, \]and then we run into difficulties with the
existence of torsion in~$H^2$.

So let's return to the case of~$\chibar=\varepsilonbar$, and consider the
problem of finding crystalline lifts. Rather than
taking~$\chi=\varepsilon$, we take~$\chi$ to be an unramified
twist of~$\varepsilon^p$, so that all extensions are crystalline. In this
case one finds that ~$H^2(G_K,\chi)[\varpi]$ is 1-dimensional, so
there is a possible obstruction to lifting; and
indeed~$H^1(G_K,\varepsilonbar)$ is 2-dimensional,
while~$H^1(G_K,\chi)\otimes_{\cO}\F$ is 1-dimensional, so no
fixed~$\chi$ can provide all the lifts that we need. It is possible to analyse
the situation directly using Kummer theory, and one finds that given
any class in~$H^1(G_K,\varepsilonbar)$, one can choose~$\chi$ so that a
lift exists.

It is then natural to consider the moduli space of all such~$\chi$,
and the universal extension over this space. The universal~$\chi$ is
the character~$\chi_t:G_K\to\cO[[t]]^\times$ given
by~$\chi_t=\varepsilon^{p}\lambda_{1+t}$, 
% \mar{ME: I changed exponent from $p-1$ to $p$, which I think is consistent
% 	with the normalizations in the discussion up to now.}
where~$\lambda_x$ is the
unramified character sending a (geometric) Frobenius element
to~$x$. Then by Tate local duality and local class field theory, it is easy
to see that~$H^2(G_K,\chi_t)=\cO[[t]]/(p,t)=\cO/p$. One checks
that~$H^1(G_K,\chi_t)$ is a free $\cO[[t]]$-module of rank~$1$, but that
the corresponding ``universal family'' of extensions over~$\Spec
\cO[[t]]$ is not in fact universal; more precisely, the fibre over any given
$x:\cO[[t]]\to\cO$, $t\mapsto a$ sits in a short exact sequence
% \mar{ME: Probably better not to use $\lambda$ for an element of $\Gm$,
% 	because you've already just used it to denote an unramified
% 	character.  Also, I don't think $\chi_{1+\lambda}$ is the 
% correct notation anyway; it should be something like $\varepsilon^p 
% \lambda_{1+\lambda}$, which also shows the problem with using $\lambda$
% for two different things!}

\numequation\label{eqn: base change of H1}0\to H^1(G_K,\chi_t)\otimes_{\cO[[t]],x}\cO\to
  H^1(G_K,\varepsilon^p\lambda_{1+a})\to \cO[[t]]/(p,t)\to 0.\end{equation} So we see would
like to have a way of ``rescaling'' $H^1(G_K,\chi_t)$ by dividing
by~$(p,t)$, so as to have a universal family that sees every
extension; and then we could try to check that this family surjects
onto~$H^1(G_K,\varepsilonbar)$. To do this, we blow up~$\Spec\cO[[t]]$
along~$(p,t)$, giving a %$p$-adic formal
scheme~$X$, which is flat
over~$\Zp$, and for which~$(p,t)$ pulls back to an ideal
sheaf~$\cI$. Over~$X$, we do have a universal
family~$H^1(G_K,\chi_X)$, and one can check that its base change to any
point~$\tx:\Spec\cO\to X$ gives the full~$H^1$.% \mar{TG: can we then
  % explain why this does actually surject onto
  % $H^1(G_K,\varepsilonbar)$?}

%\mar{ME: I added this additional sketch; should be fleshed out perhaps?}
More explicitly, since $H^0$ plays no role, we may assume that the
Herr complex is  computed by a
complex $C^1 \to C^2$ of locally free modules over 
 $\cO[[t]]$.
Since $H^2 = \cO[[t]]/(p,t),$ we can imagine that this complex looks like
\numequation
\label{eqn:complex}
\cO[[t]]^2 \buildrel (p,t) \over \longrightarrow \cO[[t]].
\end{equation}
(Note that this is
consistent with~\eqref{eqn: base change of H1}.)
In particular, the image of the coboundary map is the ideal $I = (p,t)$.
When we blow up $\Spec \cO[[t]]$ at $(p,t)$ to obtain the scheme~$X$,
the complex~\eqref{eqn:complex} pulls back over $X$ to a complex
\numequation
\label{eqn:complex on X}
\tC^1 = \cO_{X}^{\oplus 2} \to \cO_{X} = \tC^2,
\end{equation}
and the image of the coboundary map is now locally free, equal to the
ideal sheaf $\cI$ of the exceptional divisor.   %Thus the 
%formation of $H^1$ of this pulled-back complex commutes with
%arbitrary base change (the surjection $\tC^1 \to \cI$ induced by
%the boundary map is split, since the target is locally free).
The  kerrnel $\tZ^1$ is  then a locally  free  sheaf (being the kernel  of a surjection
between  locally  free sheaves), and gives the desired ``universal family'' of extensions.

To see this, consider
a class in $H^1(G_K,\varepsilonbar)$ arising from a
cocycle~$\cbar$, and choose some $c: \cO[[t]] \to C^1$ which specializes
at $(\varpi,t) = 0$ (the  closed point of $\Spec   \cO[[t]]$) to~$\cbar$.  
%\mar{ME: I'd like to change $(p,t)$ here to $(\unif,t)$,
%so we actually specialize to the closed point rather than some thickening of it.
%Is there any reason not to do that? TG: no reason that I'm aware of,
%that looks sensible to me. I guess we were using $(p,t)$ as it's the
%support of~$H^2$? Or perhaps just because we weren't thinking about coefficients.} 
This pulls back to a map
$\tc: \cO_{X} \to \tC^1,$ which composed with the coboundary gives a map
$\delta \tc: \cO_{X} \to \cI$.  Since $\cI$ is a non-trivial
invertible sheaf, this map (thought of as a section of the line bundle
associated to $\cI$) has a zero at some closed point $x \in X$.
We choose an affine open neighbourhood $U$ of $x$ in~$X$,
and consider the complex $\tC^1_U  \to \tC^2_U$ obtained by restricting~\eqref{eqn:complex on X}
to~$U$ (or, equivalently, by pulling back~\eqref{eqn:complex} to~$U$).
Now (letting $\cI_U$ denote the restriction of $\cI$ to $U$) since $\cI_U$ is locally
free, and since $U$ is affine, the surjection $\tC^1_U \to \cI_U$ splits, and we choose a splitting
$\tC^1_U = \tZ^1_U \oplus \cI_U$.  

The restriction $\tc_U$ of $\tc$ to $U$ is then 
a section of $\tZ^1_U \oplus \cI_U$ whose projection to the second factor vanishes at~$x$.
Hence, if we let $\tz_U$ denote the projection of $\tc$ to the first factor, then
$\tz_U$ and $\tc_U$ coincide at~$x$; 
thus $\tz_U$ is a cocycle over $U$  
which lifts our original cocycle~$\cbar$.
(Projecting $\tc$ to the first factor has achieved the desiered ``division
by~$(p,t)$''.)
If we specialize this cocycle over any $\cO$-valued point of $U$  passing through~$x$
(and since $U$ is  flat over~$\Z_p$,  such a point always exists after extending~$\cO$
if necessary),  then we obtain the desired lift of~$\cbar$.

% which means it gives
%a class in $H^1$ of the base-change of $\tC^{\bullet}$ to this
%point, which then (by the base-change observation of the
%preceding paragraph) lifts to a class in $H^1$ of $\tC^{\bullet}$.
%(Basically we see that $\tC^1 = \tZ^1 \oplus \cI,$
%and we replace $\tc$ by its projection to the first component,
%so that it becomes a cocycle;
%this doesn't change its value at the point where it vanishes,
%and thus means that this cocycle lifts the cocycle $\cbar$.
%This projection achieves the ``division by $(p,t)$''.)
% \mar{ME: I guess there's an additional step, where we pull-back over
% 	a section $\Spec \Zbar_p \to X$, which I'm ignoring.}

The proof of Theorem~\ref{thm:crystalline lifts} is via a generalisation
of this construction,  with the following theorem % Theorem~\ref{thm: codim bound crystalline
  % deformation ring}
being  the general analogue that we need of the fact
used above that~$H^2(G_K,\chi_t)=\cO[[t]]/(p,t)$.
(To see the relationship with the previous argument, which used 
non-triviality of the line bundle associated to $\cI$, note that 
if $\delta \tc$ induced a surjection $\cO_{\cX} \to \cI$,
then we would see that $c$ induces a surjection
$\cO[[t]] \to I$, hence $I$ would be principal, and $H^2 = \cO[[t]]/I$ would be
supported in codimension one, whereas in fact it is supported only in
codimension two.)
      \begin{thm}\label{thm: codim bound crystalline deformation
          ring}% \mar{TG: at the moment we don't seem to actually
          % discuss the proof of this?}
  For any regular tuple of labelled Hodge--Tate weights
  $\underline{\lambda}$, and any fixed irreducible representation $\alphabar: G_K \to \GL_a(\Fbar_p)$,  the locus of points
  $x\in \Spec R^{\underline{\lambda},\crys}/\varpi$ for which
	$$\dim_{\kappa(x)}
	\kappa(x)\otimes_{R^{\square}_{\rhobar}} \Ext^2_{G_K}(\alphabar,\rho^{\univ})
        \geq r$$ has codimension at least~$r$.
      \end{thm}
      \begin{proof}
        See~\cite[Thm.\ 6.1.1]{emertongeepicture}. The basic idea is
        that since the stack $\cX_d^{\underline{\lambda},\crys}$ is a
        $p$-adic formal algebraic stack, its special fibre is an
        algebraic stack, and the ring
        $R^{\underline{\lambda},\crys}/\varpi$ is an effective versal
        ring at some finite type point. It therefore suffices to prove
        a corresponding codimension bound on the support of
        $\Ext^2_{G_K}(\alphabar,-)$ in the special fibre of~$\cX_d$,
        and this is provided by  Theorem~\ref{thm:Xdred
  is algebraic}~(3).
      \end{proof}

% \begin{example}
% Take $\barr\rho = \omega$ and $X = \Spec R_{\barr\rho}^{\sq, \cris, HTwt=-p} = \Spec \Z_p[[\alpha-1]]$, over which lives the family $\ur_\alpha \epsilon^p$. Then $H^2$ is supported on the locus $(\alpha-1,p)$, where it has dimension 1.
% \end{example}

\section{Geometric Breuil--M\'ezard}\label{sec: geometric BM}As in the
previous lectures, we will concentrate on the potentially crystalline
case when formulating the Breuil--M\'ezard conjecture; most of what we say goes
over unchanged to the potentially semistable version, and we refer the
reader to~\cite[\S 8]{emertongeepicture} for the details.

\subsection{The irreducible components of $\cX_{d,\red}$}\label{subsec:
  irred components}We can now complete the analysis of the irreducible
components of~$\cX_{d,\red}$. Recall that by Theorem~\ref{thm:Xdred is algebraic}, $\cX_{d,\red}$ is an algebraic stack of finite presentation
over~$\F$, and has dimension~$[K:\Qp]d(d-1)/2$. Furthermore, for each
Serre weight~$\underline{k}$, there is a corresponding irreducible
component~$\cX_{d,\red,\Fpbar}^{\underline{k}}$
of~$(\cX_{d,\red})_{\Fpbar}$, and the components for
different weights~$\underline{k}$ are distinct. Using our results on
crystalline lifts, we can now show that these are the only irreducible
components. The following is~\cite[Thm.\ 6.5.1]{emertongeepicture}.

\begin{thm}
  \label{thm:reduced dimension} $\cX_{d,\red}$ is equidimensional of
  dimension $[K:\Q_p] d(d-1)/2$, and the irreducible components
  of~$(\cX_{d,\red})_\Fpbar$ are precisely the various closed
  substacks~$\cX_{d,\red,\Fpbar}^{\underline{k}}$; in particular,
  $(\cX_{d,\red})_\Fpbar$ is maximally nonsplit of niveau~$1$.
  Furthermore each~$\cX_{d,\red,\Fpbar}^{\underline{k}}$ can be
  defined over~$\F$, i.e.\ is the base change of an irreducible
  component~$\cX_{d,\red}^{\underline{k}}$ of~$\cX_{d,\red}$.
	% There is a finite set of tuples $\underline{\lambda}$ of regular labeled Hodge--Tate
	% weights such that $\cX_{d,\red}$ is the union of corresponding
	% underlying reduced substacks $\cX_d^{\crys,\underline{\lambda}}.$
	% In particular,
	% each irreducible component of $\cX_{d,\red}$ is of dimension
	% $[K:\Q_p] d(d-1)/2$.
\end{thm}
\begin{proof}%[Proof of Theorem~\ref{thm:reduced dimension} in the rank $d+1$ case]
  To see that the~$\cX_{d,\red,\Fpbar}^{\underline{k}}$ may all be
  defined over~$\F$, we need to show that the action
  of~$\Gal(\Fpbar/\F)$ on the irreducible components
  of~$(\cX_{d,\red})_\Fpbar$ is trivial. This follows immediately by
  considering its action on the maximally nonsplit representations of
  niveau~$1$ (since the action of~$\Gal(\Fpbar/\F)$ preserves the
  property of being maximally nonsplit of niveau~$1$ and
  weight~$\underline{k}$).  It is therefore enough to prove that each
  irreducible component of $(\cX_{d,\red})_\Fpbar$ is of dimension of
  at least $[K:\Q_p] d(d-1)/2$; this follows by choosing a closed
  point not contained in any other irreducible component, and noting
  that (by Theorem~\ref{thm:  main crystalline lifts}) it is
  contained in the special fibre of some~$\cX_d^{\crys,\lambdau}$
  with~$\lambdau$ regular.
\end{proof}
We can also prove~\cite[Prop.\ 6.5.2]{emertongeepicture}:
\begin{prop}\label{prop: X is not padic formal algebraic}
   $\cX_d$ is not a $p$-adic formal algebraic stack.
\end{prop}
\begin{proof}
  Assume that~$\cX_d$ is a $p$-adic formal algebraic stack, so that
  its special fibre~$\overline{\cX}_d:=\cX_d\times_\cO\F$ is an
  algebraic stack, which is furthermore of finite type
  over~$\F$. Since the underlying reduced substack
  of~$\overline{\cX}_d$ is~$\cX_{d,\red}$, which is equidimensional of
  dimension~$[K:\Qp]d(d-1)/2$, we see that~$\overline{\cX}_d$ also has
  dimension~$[K:\Qp]d(d-1)/2$.

  Computing with versal rings, we see that for every 
 $\rhobar:G_K\to\GL_d(\F)$ the unrestricted framed deformation ring~$R_{\rhobar}^\square/\varpi$ must have
  dimension~$d^2+[K:\Qp]d(d-1)/2$. However, it is known that there are
  representations~$\rhobar$ for which~$R_{\rhobar}^\square/\varpi$ is
  formally smooth of dimension $d^2+[K:\Qp]d^2$ (see for
  example~\cite[Lem.\ 3.3.1]{Allen2019}). Thus we must
  have~$d^2=d(d-1)/2$, a contradiction.
  %\mar{TG: should double check
 %   that I have these dimensions correct; I guess the only possible
 %   issue would be~$d=1$, which will be explicitly checked in the
 %   following section anyway. ME: Looks fine to me}
\end{proof}

\subsection{The qualitative geometric Breuil--M\'ezard
  conjecture}\label{subsec: qualitative BM}If~$\lambdau$ is a regular
Hodge type, and~$\tau$ is any inertial type, then the
stack~$\cX^{\crys,\lambdau,\tau}_d$ 
is a finite type $p$-adic formal algebraic stack over~$\cO$, which is
$\cO$-flat and equidimensional of dimension~$1+[K:\Qp]d(d-1)/2$. It
follows that its special fibre $\cXbar^{\crys,\lambdau,\tau}_d$  is
an  algebraic stack over~$\F$ which
is equidimensional of dimension~$[K:\Qp]d(d-1)/2$. Since
~$\cX^{\crys,\lambdau,\tau}_d$ is a 
closed substack of~$\cX_d$, its  special fibre $\cXbar^{\crys,\lambdau,\tau}_d$
is a closed substack of the
special fibre~$\cXbar_d$, and its irreducible components (with the
induced reduced substack structure) are therefore closed substacks of
the algebraic stack~$\cXbar_{d,\red}$.

Since~$\cXbar_{d,\red}$ is equidimensional of
dimension~$[K:\Qp]d(d-1)/2$, it follows that the irreducible
components of $\cXbar^{\crys,\lambda,\tau}_d$ are irreducible
components of~$\cXbar_{d,\red}$, and are therefore of the
form~$\cXbar_{d,\red}^{\underline{k}}$ for some Serre
weight~$\underline{k}$.

For each~$\underline{k}$, we
write~$\mu_{\underline{k}}(\cXbar^{\crys,\lambdau,\tau}_d)$ for the
multiplicity
of~$\cXbar_{d,\red}^{\underline{k}}$ as a component
of~$\cXbar^{\crys,\lambdau,\tau}_d$. % , in the sense
% of~\cite[\S1]{EGcomponents}.
We write~$Z_{\crys,\lambdau,\tau}=Z(\cXbar^{\crys,\lambdau,\tau}_d)$
for the corresponding cycle, i.e.\ for the formal sum
\numequation\label{eqn: cris HS multiplicity
  stack}Z_{\crys,\lambdau,\tau}=\sum_{\underline{k}}\mu_{\underline{k}}(\cXbar^{\crys,\lambdau,\tau}_d)\cdot\cXbar_d^{\underline{k}}, \end{equation}
which we regard as an element of the finitely generated free abelian
group~$\Z[\cX_{d,\red}]$ whose generators are the irreducible
components~$\cXbar_d^{\underline{k}}$.

Fix some representation~$\rhobar:G_K\to\GL_d(\F)$, corresponding to a
point $x:\Spec\F\to\cX_d$.   For each regular Hodge type~$\lambdau$ and inertial
type~$\tau$, we have an effective versal morphism
$\Spec
R_{\rhobar}^{\crys,\underline{\lambda},\tau}/\varpi\to\cXbar^{\crys,\lambdau,\tau}$.
% \mar{TG: probably we should make the superscripts be in the
  % same order for the stacks and the versal rings!}
For each~$\underline{k}$ we set 
\[\cC_{\underline{k}}(\rhobar):=\Spf
  R_{\rhobar}^\square\times_{\cX_d}\cX_d^{\underline{k}},\]which we
regard as a
cycle of dimension~$d^2+[K:\Qp]d(d-1)/2$ % codimension $[K:\Qp]d(d+1)/2$\mar{TG: I'm adopting the
  % phrasing from the introduction, but it seems problematic to me,
  % because I don't think we actually know the dimension
  % of~$R_{\rhobar}$ at this point? So probably it's better to say that
  % it is a cycle of dimension~$[K:\Qp]d(d-1)/2$?}
in~$\Spec R_{\rhobar}/\varpi$ (note that since~$\cX_d^{\underline{k}}$
is algebraic, it has effective versal rings, so we really get a
subscheme of $\Spec R_{\rhobar}/\varpi$, rather than of
$\Spf R_{\rhobar}/\varpi$). The following theorem gives a qualitative
version of the refined Breuil--M\'ezard
conjecture~\cite[Conj.\ 4.2.1]{emertongeerefinedBM}. While its statement is purely local, we do not know
how to prove it without making use of the stack~$\cX_d$.
\begin{thm}
  \label{thm: qualitative BM}Let $\rhobar:G_K\to\GL_d(\F)$ be a
  continuous representation. Then there are finitely many cycles of
  dimension~$d^2+[K:\Qp]d(d-1)/2$ in~$\Spec R_{\rhobar}^\square/\varpi$ such that
  for any regular Hodge type~$\lambdau$ and any inertial type~$\tau$,
  each of the special fibres~$\Spec
  R_{\rhobar}^{\crys,\underline{\lambda},\tau}/\varpi$  is set-theoretically
  supported on some union of these cycles.
\end{thm}
\begin{proof}
	% \mar{ME: This argument is probably implicitly using $G$-ring facts,
	% 	to know that multiplicities don't change when passing
	% 	to versal rings. TG: how about just having this stacks
        %       project citation?}
  We have
  $\Spec R_{\rhobar}^{\crys,\underline{\lambda},\tau}/\varpi =\Spf
  R^{\square}_{\rhobar}\times_{\cX_d}\cXbar^{\crys,\lambdau,\tau}$. It
  follows from~\eqref{eqn: cris HS multiplicity stack}, together with the definition
  of~$\cC_{\underline{k}}(\rhobar)$, that we may write the underlying
  cycle as \numequation\label{eqn: cris HS multiplicity def
    ring}Z(\Spec
  R_{\rhobar}^{\crys,\underline{\lambda},\tau}/\varpi)=\sum_{\underline{k}}\mu_{\underline{k}}(\cXbar^{\crys,\lambdau,\tau}_d)\cdot\cC_{\underline{k}}(\rhobar).\end{equation}
% \numequation\label{eqn: ss HS multiplicity def ring}Z(\Spec
% R_{\rhobar}^{\semis,\underline{\lambda},\tau}/\varpi)=\sum_{\underline{k}}\mu_{\underline{k}}(\cXbar^{\semis,\lambdau,\tau}_d)\cdot\cC_{\underline{k}}(\rhobar). \end{equation}(Note
% that by~\cite[\href{https://stacks.math.columbia.edu/tag/0DRD}{Tag
%   0DRD}]{stacks-project}, the multiplicities to not change
% when passing to versal rings.)
The
theorem follows immediately (taking our finite set of cycles to be
the~$\cC_{\underline{k}}(\rhobar)$). 
\end{proof}
We can regard this theorem as isolating the ``refined'' part
of~\cite[Conj.\ 4.2.1]{emertongeerefinedBM}; that is, we have taken
the original numerical Breuil--M\'ezard conjecture, formulated a geometric
refinement of it, and then removed the numerical part of the
conjecture. The numerical part of the conjecture % (in the optic of this
% paper)
consists of relating the multiplicities
$\mu_{\underline{k}}(\cXbar^{\crys,\lambdau,\tau}_d)$ to the 
representation theory of~$\GL_n(k)$, as we now recall.

The following theorem is essentially  due to
Schneider--Zink \cite{MR1728541}.
\begin{thm}Let ~$\tau:I_K\to\GL_d(\Qpbar)$ be an inertial type. Then
  there is a finite-dimensional smooth irreducible
  $\Qpbar$-representation $\sigma^{\crys}(\tau)$ of~$\GL_d(\cO_K)$ with the property
  that if ~$\pi$ is an irreducible smooth $\Qpbar$-representation
  of~$\GL_d(K)$, then the $\Qpbar$-vector space
  $\Hom_{\GL_d(\cO_K)}(\sigma^{\crys}(\tau),\pi)$ has dimension at
  most~$1$, and is nonzero precisely
  if~$\rec_p(\pi)|_{I_F}\cong \tau$, and $N=0$ on~$\rec_p(\pi)$.
\end{thm}

For each regular Hodge
type~$\underline{\lambda}$ we let~$W(\lambdau)$ be the corresponding
representation of~$\GL_d(\cO_K)$, defined as follows: For
each~$\sigma:K\into\Qpbar$, we
write~$\xi_{\sigma,i}=\lambda_{\sigma,i}-(d-i)$, so that
$\xi_{\sigma,1}\ge\dots\ge\xi_{\sigma,d}$. We view each
$\xi_\sigma:=(\xi_{\sigma,1},\dots,\xi_{\sigma,d})$ as a dominant
weight of the algebraic group~$\GL_d$ (with respect to the upper
triangular Borel subgroup), and we write~$M_{\xi_\sigma}$ for the
algebraic $\cO_K$-representation of~$\GL_d(\cO_K)$ of highest
weight~$\xi_\sigma$. Then we define
$L_{\lambdau}:=\bigotimes_{\sigma}M_{\xi_\sigma}\otimes_{\cO_K,\sigma}\cO$.

For each~$\tau$ we let~$\sigma^{\crys,\circ}(\tau)$ denote a choice of $\GL_d(\cO_K)$-stable
$\cO$-lattice in~$\sigma^{\crys}(\tau)$, write
$\sigma^{\crys}(\lambda,\tau):=L_{\lambdau}\otimes_\cO\sigma^{\crys,\circ}(\tau)$,
and write
$\sigmabar^{\crys}(\lambda,\tau)$ 
for the semisimplification of the $\F$-representation of~$\GL_d(k)$
given by
$\sigma^{\crys}(\lambda,\tau)\otimes_\cO\F$.
 For
each Serre weight~$\underline{k}$, we write~$F_{\underline{k}}$ for
the corresponding irreducible $\F$-representation of~$\GL_d(k)$. Then there are unique
integers $n_{\underline{k}}^\crys(\lambda,\tau)$ such
    that \[\sigmabar^{\crys}(\lambda,\tau)\cong\oplus_{\underline{k}}F_{\underline{k}}^{\oplus
        n_{\underline{k}}^\crys(\lambda,\tau)}.\] 
Our ``universal'' geometric Breuil--M\'ezard conjecture is as
follows (see Lecture~\ref{subsec: Serre
  weights} below for some motivation for this conjecture).% \mar{TG: rephrase this material as in my Montreal slides to
  % emphasise support of conjectural sheaves?}
\begin{conj}
  \label{conj: geometric BM}There are cycles~$Z_{\underline{k}}$ with
  the property that for each regular Hodge type~$\lambdau$ and each
  inertial type~$\tau$, we have
  $Z_{\crys,\lambdau,\tau}=\sum_{\underline{k}}n_{\underline{k}}^\crys(\lambda,\tau)\cdot
  Z_{\underline{k}}$.% ,  $Z_{\semis,\lambdau,\tau}=\sum_{\underline{k}}n_{\underline{k}}^\semis(\lambda,\tau)\cdot
  % Z_{\underline{k}}$.
\end{conj} 
    
\subsection{The relationship between the numerical, refined and
  geometric Breuil--M\'ezard conjectures}\label{subsec: relating
  different BM conjectures}In brief (see \cite[\S
8.3]{emertongeepicture} for the details), the relationship is as follows:
Conjecture~\ref{conj: geometric BM} implies (by pulling back to versal
rings) the geometric  conjecture
of~\cite{emertongeerefinedBM}, which in turn implies the numerical
conjecture.

Conversely, the numerical conjecture implies Conjecture~\ref{conj:
  geometric BM}; in fact, it is enough to know the numerical
conjecture for a single sufficiently generic~$\rhobar$ on each
irreducible component of~$\cX_{d,\red}$. To see this, recall that the
numerical conjecture for~$\rhobar$ is that there are integers
$\mu_{\underline{k}}(\rhobar)$ such that for all~$\underline{k}$, we
have \numequation\label{eqn: numerical version of BM cris}e(\Spec
R_{\rhobar}^{\crys,\underline{\lambda},\tau}/\varpi)=\sum_{\underline{k}}n_{\underline{k}}^\crys(\lambda,\tau)\mu_{\underline{k}}(\rhobar).\end{equation}

For each $\underline{k}$ we choose a point
$x_{\underline{k}}:\Spec\F\to\cXbar_{d,\red}$ which is contained
in~$\cXbar^{\underline{k}}$ and not in any~$\cXbar^{\underline{k}'}$
for $\underline{k}'\ne\underline{k}$. We furthermore demand that
$x_{\underline{k}}$ is a smooth point of~$\cXbar_{d,\red}$. (Since
$\cXbar_{d,\red}$ is reduced and of finite type over~$\F$, there is a dense set
of points of~$\cXbar^{\underline{k}}$ satisfying these conditions.)
% \mar{TG: this would all be cleanest if we could furthermore choose it
%   to have formally smooth versal ring -- can we? TG: yes, we're on
%   something reduced and varieties have smooth
%   points.}
% Let~$e_{\underline{k}}$ denote the multiplicity
% of~$\cXbar^{\underline{k}}$ at~$x_{\underline{k}}$.\mar{TG: needs to be
%   defined/explained, but I'm getting tired.} 
Write~$\rhobar_{\underline{k}}:G_K\to\GL_d(\F)$ for the
representation corresponding to~$X_{\underline{k}}$, and assume that
the numerical conjecture holds for
each~$\rhobar_{\underline{k}}$. Then if we set
 \numequation\label{eqn: cycle from multiplicities}Z_{\underline{k}}:=\sum_{\underline{k}'}
  \mu_{\underline{k}}(\rhobar_{\underline{k}'})\cdot\cXbar^{\underline{k}'},\end{equation}
%and assume that the numerical conjecture~\eqref{eqn: numerical version of BM cris} holds,
it is easy to check that Conjecture~\ref{conj: geometric BM} holds.

The numerical Breuil--M\'ezard conjecture (and consequently
Conjecture~\ref{conj: geometric BM}) holds if~$K=\Qp$
and~$d=2$. Most cases are proved
in Kisin's paper~\cite{KisinFM} % (which gave a proof in many cases by a
% mixture of local and global techniques),
and Pa{\v{s}}k{\=u}nas'
paper~\cite{paskunasBM}, % which reproved these results by purely local
% means (the $p$-adic local Langlands correspondence), relaxed the
% hypotheses on~$\rhobar$, and also proved the refined version of the
% correspondence. The
and the remaining cases not handled by these papers are proved in the
papers~\cite{HuTan,sandermultiplicities,2018arXiv180307451T,2019arXiv190806174T}. The
conjecture also holds if~$d=2$, $p>2$,
$\lambdau=(\underline{0},\underline{1})$, and~$K$ and~$\tau$ are
arbitrary, by the main result of~\cite{geekisin}.   For some recent progress for~$d>2$ in suitably generic
  situations, we refer the reader to~\cite{le2020local}.

In the case that~$d=2$, $p>2$, and~$K$ is arbitrary, we can make the
cycles~$Z_{\underline{k}}$ completely explicit: we say that a Serre
weight~$\underline{k}$ for~$\GL_2$ is ``Steinberg'' if for
each~$\sigmabar$ we have
$k_{\sigmabar,1}-k_{\sigmabar,2}=p-1$. If~$\underline{k}$ is Steinberg
then we define~$\underline{\tilde{k}}$ by
$\tilde{k}_{\sigmabar,1}=\tilde{k}_{\sigmabar,2}=k_{\sigmabar,2}$. Then
if~$\underline{k}$ is not Steinberg, we have
$Z_{\underline{k}}=\cX_2^{\underline{k}}$, while if $\underline{k}$ is
Steinberg,
$Z_{\underline{k}}=\cX_2^{\underline{k}}+\cX_2^{\underline{\tilde{k}}}$.

  This explicit description follows from the results
  of~\cite{CEGSKisinwithdd}; for the details, see~\cite[Thm.\
  8.6.2]{emertongeepicture}. (Roughly speaking, the point is that it's
  easy to compute the tamely potentially Barsotti--Tate deformation rings for
  generic extensions of generic characters, and they're either zero or formally
  smooth.)

\subsection{The weight part of Serre's conjecture}\label{subsec: Serre
  weights}Finally, we very briefly explain some motivation for
Conjecture~\ref{conj: geometric BM}, and its relationship to the
weight part of Serre's conjecture. This connection was first explained
in the context of the numerical Breuil--M\'ezard conjecture
in~\cite{geekisin}. For more details, see for
example~\cite{2015arXiv150902527G} (particularly Sections~3
and~4). We will further expand on this discussion in the context of a
hypothetical $p$-adic local Langlands correspondence involving sheaves
on~$\cX_d$ in Lecture~\ref{subsec: conjectural p adic LL}.
% \mar{TG: add
  % some discussion of why BM should be true.}

We expect that the cycles~$Z_{\underline{k}}$ will be effective, in
the sense that they are combinations of the~$\cX_d^{\underline{k}}$
with non-negative coefficients. Indeed, we expect that after pulling
back to the special fibre of the universal deformation ring at
some~$\rhobar$, the cycle~$Z_{\underline{k}}$ will be precisely the
support of the Taylor--Wiles--Kisin patched modules of mod~$p$ automorphic
forms of weight~$\underline{k}$, and the support of a sheaf is an
effective cycle by definition. (Again, see for example~\cite[\S 3,
4]{2015arXiv150902527G} for more details; we are slightly simplifying
the situation by supposing that the Taylor--Wiles--Kisin method does not
introduce any patching variables, but these are essentially irrelevant
for our discussion in any case.)

% This expectation is borne out in all
% known examples, and in any case would be
% a consequence of standard conjectures about the Taylor--Wiles
% method.

The \emph{weight part of Serre's conjecture} is a prediction that the
possible Serre weights of mod~$p$ automorphic forms giving rise to a
fixed global Galois representation only depends on the restrictions of
this representation to decomposition groups of places above~$p$. If
this is the case, there is a way to associate a corresponding set of
Serre weights to each $\rhobar:G_K\to\GL_d(\Fpbar)$. Such a
description was originally given by Serre in the case $K=\Qp$
and~$d=2$; while Serre's recipe is completely explicit, in general it
seems to be unreasonable to hope for such a description.

However, if we admit the standard expectation (which is closely
related to the Fontaine--Mazur conjecture) that all irreducible
components of the generic fibres of crystalline deformation rings are
witnessed by automorphic forms, it follows formally from the
Taylor--Wiles--Kisin method that the cycle~$Z_{\underline{k}}$ is indeed (as
conjectured above) the
support of the corresponding patched module of weight~$\underline{k}$
modular forms, and from this one easily deduces that
representation~$\rhobar$ admits~$\underline{k}$ as a Serre weight if
and only if $Z_{\underline{k}}$ is supported at~$\rhobar$.

% Assume Conjecture~\ref{conj: geometric BM}, and assume that the
% cycles~$Z_{\underline{k}}$ are effective. As explained in Section~\ref{subsec: relating
%   different BM conjectures}, it follows that the numerical
% Breuil--M\'ezard holds for every~$\rhobar$, with
% $\mu_{\underline{k}}(\rhobar)$ being the Hilbert--Samuel multiplicity
% of the cycle~$Z_{\underline{k}}(\rhobar)$,
% \mar{ME: Has $Z_{\underline{k}}(\rhobar)$ been defined?}
% which (since~$Z_{\underline{k}}(\rhobar)$ is effective) is positive if and
% only if~$Z_{\underline{k}}(\rhobar)$ is nonzero, i.e.\ if and only
% if~$Z_{\underline{k}}$ is supported at~$\rhobar$. Thus we can rephrase
% the Breuil--M\'ezard version of the weight part of Serre's conjecture
% (i.e.\ the version suggested by~\cite{geekisin})
% as saying that $W(\rhobar)$ is the set of~$\underline{k}$ such that
% $Z_{\underline{k}}$ is supported at~$\rhobar$.

Equivalently, we can rephrase the weight part of Serre's conjecture in
the following way: to each irreducible component of~$\cX_{d,\red}$, we
assign the set of weights~$\underline{k}$ with the property
that~$Z_{\underline{k}}$ is supported on this component. Then for
each~$\rhobar$, the corresponding set of Serre weights is simply the
union of the sets of weights for the irreducible components
of~$\cX_{d,\red}$ which contain~$\rhobar$.

We expect that in particular the irreducible
component~$\cX_{d,\red}^{\underline{k}}$ is assigned the Serre
weight~$\underline{k}$, but that the list of Serre weights associated
to this irreducible component can be longer; indeed, it follows from
the description given above of the cycles~$Z_{\underline{k}}$ in the
case~$d=2$ that if $k_{\sigmabar,1}-k_{\sigmabar,2}=0$ for
all~$\sigmabar$, then there are two Serre weights associated
to~$\cX_{d,\red}^{\underline{k}}$, namely~$\underline{k}$ and the
corresponding Steinberg weight.

\section{Bernstein Centers, Moduli Spaces, and the Categorical
  $p$-adic Langlands program}%\mar{TG: drop this section?!}
In this final lecture we discuss some of the general theory of moduli spaces
associated to algebraic stacks, and explain how it applies to the stacks
we have constructed in the previous lectures.  We will see that there
is a surprising connection between these ideas and the $p$-adic local Langlands 
correspondence.  

\subsection{Moduli spaces}
%``Moduli spaces'' refers to the spaces associated to moduli stacks.
An algebraic stack $\cX$ has an underlying ``Zariski'' topological
space, %with points,\mar{TG: delete ``with points''?}
formed by taking a smooth cover by a scheme, taking its topological space, and then taking a quotient topological space.
%For example, if $G$ is a smooth group scheme over some base-scheme $S$,
%acting on an  $S$-scheme~$Y$, and if $\Spec 
%ground field~$k$)
%acting on the affine scheme $\Spec A$ (say with  $A$ being a finite
%type  $k$-algebra), then 

\begin{example}
%For example,
If $\cX = [\bA^1/\Gm]$   (over some field~$k$,
with $t \in \Gm$ acting  on $x  \in \bA^1$
via multiplication, i.e.  $t\cdot x = tx$),
then the associated  topological space consists of two  points, one of which
is open and specializes to the other, which is closed.   These two
points correspond to the  two orbits of $\Gm$  acting on  $\bA^1$ --- the  open
orbit  and the closed  orbit.   The  closure  relation  between the points
corresponds  to the closure relation between these orbits.
\end{example}

The topological space underlying $\cX$ is {\em just} a topological space,
though; it doesn't come equipped with a structure sheaf making it into
a scheme.  So we can ask:
is there a morphism
%But you can look for a map
$f: \cX \to X$ to $X$ a scheme, or, more generally, to an algebraic space,
which is the ``best possible approximation to the stack $\cX$''?
%and has no automorphisms.
Somewhat more precisely,
%Loosely speaking,
one might call $X$ a \textit{moduli space associated to $\cX$} if $f$ is initial in the space of maps to algebraic spaces. Or one could say that
$f$ is an \textit{associated moduli space} morphism.

Unfortunately,
this doesn't really help you access $X$ or say anything concrete about it,
or even determine if such an~$X$ exists.
%Various authors, in particular Jared Alper~\cite{MR3237451,MR3272912},
%have investigated this question in various contexts.
In \cite[Prop.~7.1.1]{MR3272912}, Jared Alper
gives properties that ensure that $f$ is initial for morphisms to
{\em locally separated} algebraic spaces.
%(for instance, a line with the origin replaced by a $\bP^1$ is not locally separated).
The properties are:
%\mar{TG: is this really a list of defining properties? If so, where is this proved? I don't see it in the original papers}
\begin{enumerate}
	\item If $k$ is algebraically closed, the induced map
$\cX(k) \to X(k)$ identifies the target with the quotient of the source
by the equivalence relation generated by $x \sim y$ if $\barr{\set{x}} \cap \barr{\set{y}} \neq \es$.
	\item $f$ is a universal submersion.\footnote{A {\em submersion}
is a morphism which is surjective and which induces a quotient map on the
underlying topological spaces.}
	\item $f_*\cO_{\cX} = \cO_X$.
\end{enumerate}

As Alper  notes, the first property says that $X$ has the right points,
the second that it has the right topology, and the third that it has the
right functions.

\begin{example}
%E.g.
In the case of
$\cX = [\bA^1/\bG_m]$ over the algebraically  closed field~$k$,
there are two $k$-points: the origin, and the open point (corresponding to
the open orbit  of $\Gm$ on~$\bA^1$). 
Since the  open point  specializes  to  the closed point,
the associated  moduli space is just the single  point~$\Spec~k$.
\end{example}

This still leaves open the question of when  such an $X$ exists.
The theorem of Keel--Mori~\cite{MR1432041} (plus various technical
improvements) implies that if $\cX$ is Deligne--Mumford (see Definition~\ref{defn: algebraic stack})
%(and maybe quasi-DM?)
then there exists $X$ which is a ``coarse moduli space'' (the map $\cX \to X$ induces a bijection on $k$-points for algebraically closed fields~$k$).
Unfortunately, this result doesn't apply in our context: firstly, our moduli stacks
of \'etale $(\varphi,\Gamma)$-modules are formal algebraic stacks, rather than
being actually algebraic; but, more significantly, they are not  Deligne--Mumford:
their points have infinite automorphism groups (every Galois representation
$\rhobar$ admits  at  least the scalar matrices as automorphisms).

In the context of more general (i.e.\  not-necessarily Deligne--Mumford)
algebraic stacks,
Alper~\cite{MR3237451,MR3272912}
has developed a theory of ``good'' and ``adequate'' moduli spaces.
His theory is related to earlier ideas in Geometric Invariant Theory  (GIT) --- 
the  difference between ``good''  and ``adequate'' (the latter is the more
general notion) is related to  the difference of behaviour of the representation
theory of reductive linear algebraic groups in positive characteristic versus in
characteristic zero.
One feature of this theory is that if $f: \cX \to X$ is good or adequate, then it is universally closed. Also, if $\cX$  admits a good or adequate
moduli space, then all stabilizers at closed $k$-points are reductive. Further,
the relation $x \sim y$ if $\barr{\set{x}} \cap \barr{\set{y}} \neq \es$ is
in fact  an equivalence relation. % (under some assumption\dots good/adequate?).

\begin{example}
If $G$ is reductive over $k$ and $A$ is a finite type $k$-algebra with a $G$-action, then $[\Spec A/G] \to \Spec A^G$ is adequate (or good in characteristic $0$).
People often refer to  $\Spec A^G$  as the GIT quotient of $\Spec A$  by~$G$.
\end{example}
 
\begin{example}
Take $\bP^1$, with $\bG_m$ acting via scaling.  This is similar to  the example
of $\bG_m$ acting on~$\bA^1$ considered above, but now there are {\em two}
closed orbits (each of  the  points  $0$  and $\infty$), and  the open orbit
specializes to both of them.   Thus 
the relation $x \sim y$ if $\barr{\set{x}} \cap \barr{\set{y}} \neq \es$ is
{\em not} an equivalence relation in this example, and so $[\mathbb P^1/\bG_m]$
does  not  admit an adequate moduli space.
\end{example}

\subsection{Closed points and stacks of Galois
  representations}\label{subsec: closed points GK reps}
As we have just seen, the notion of an underlying moduli space
for~$\cX$ involves understanding the topology on its underlying set of
$\Fpbar$-points, which we know biject with isomorphism classes of
representations $\rhobar:G_K\to\GL_d(\Fpbar)$.

The following is part of~\cite[Thm.\ 6.6.3]{emertongeepicture} (which
also completely describes the closure relations between points in
terms of ``virtual partial
semi-simplification'').

\begin{theorem}
\label{thm:closed points}
% \mar{ME: I think there's a command that fixes the spacing in this context, 
% but I forget what it is!}
\leavevmode
\begin{enumerate}
\item
The finite type points
{\em (}equivalently, $\Fbar_p$-valued points{\em )}
of $(\cX_{\red})_{\Fbar_p}$ are in natural bijection with the isomorphism
classes of continuous representations $\rhobar: G_K \to \GL_d(\Fbar_p)$.
\item A finite type point
of $|(\cX_{\red})_{\Fbar_p}|$
is closed if and only if the associated Galois representation $\rhobar$ is semi-simple.
\item If $x \in |(\cX_{\red})_{\Fbar_p}|$ is a finite type point, corresponding
to the Galois representation~$\rhobar$, then the closure $\overline{\{x\}}$
contains a unique closed point, whose corresponding Galois
representation is the semi-simplification $\rhobar^{\ss}$ of~$\rhobar$.
\end{enumerate}
\end{theorem}

The proof of Theorem~\ref{thm:closed points} involves showing that
closure relations in~$\cX$ are actually realized by representations
of~$G_K$ (rather than just by $(\varphi,G_K)$-modules). A closely
related result is ~\cite[Thm.\ 6.7.2]{emertongeepicture}, which
makes precise the notion of the largest substack of~$\cX$ which genuinely
parameterizes $G_K$-representations.  To explain this, we recall first
that in~\cite{MR3831282} Wang-Erickson constructed
a formal algebraic stack $\cX^{\Gal}$ characterised by the following property: if~$A$ is a
$\Zp$-algebra in which $p$ is nilpotent, then $\cX^{\Gal}(A)$ is the
groupoid of continuous morphisms \[\rho:G_{K}\to\GL_d(A)\](where~$A$
has the discrete topology, and~$G_{K}$ its natural profinite
topology).

 Just as was the case with the stacks of Weil--Deligne representations discussed
 in Lecture~\ref{sec:intro},
 the geometry of $\cX^{\Gal}$ is quite
 different from that of~$\cX$; in particular, the members 
 of any connected family of~$G_K$-representations over $\Fbar_p$ have constant
 semisimplification, 
%Wang-Erickson shows that
and so
%~$(\cX^{\Gal})_{\Fbar_p}$ is a disjoint union of infinitely many substacks, and indeed
we can write \[(\cX^{\Gal})_{W(\Fbar_p)}=\coprod_{D}\cX_{D}^{\Gal},\]where~$D$
runs over the isomorphism classes of $d$-dimensional semisimple
$\Fpbar$-representations of~$G_K$, and~$\cX_{D}^{\Gal}(A)$ is the
groupoid of those~$\rho$ the semisimplification of whose reductions
modulo~$p$ is~$D$. 

We then have the following result relating Wang-Erickson's stack to ours.
%The following is ~\cite[Thm.\ 6.7.2]{emertongeepicture}; it is a
%straightforward consequence of our almost Galois descent results
%described in Lecture~\ref{subsec: almost Galois descent}.
\begin{thm}
  \label{thm: map from Galois stack to ours}
  %There is a natural monomorphism
  %$\cX_d^{\Gal}\to\cX_d$, which  for  any finite
  %type~$\cO/\varpi^a$-algebra $A$, is given by the functor
  %$\cX_d^{\Gal}(A)\to\cX_d(A)$ taking $T_A\to \mathbb{D}_A(T_A)$.
  There is a natural monomorphism $\cX^{\Gal}\to\cX$, which induces a
  bijection on $\Fpbar$-points, and is
furthermore versal at these points. %$\Fpbar$-points.
For any~$D$ as above, the
induced monomorphism  \[\cX_{D}^{\Gal}\into\cX_{W(\Fbar_p)}\]   induces a closed immersion
  on underlying topological spaces {\em (}or,  equivalently, on underlying
  reduced substacks{\em )}.
\end{thm}
\begin{proof}
Other than the claim regarding closed immersions,
this is a restatement of~\cite[Thm.\ 6.7.2]{emertongeepicture}, which
itself is a
straightforward consequence of our almost Galois descent results
described in Lecture~\ref{subsec: almost Galois descent}.

To prove the claim about closed immersions, 
we note that it follows from the description of specialization of $\Fbar_p$-points
given in Theorem~\ref{thm:closed points} that
the image of the monomorphism $\cX^{\Gal}_D\to\cX_{W(\Fbar_p)}$ is closed under specialization
of finite type points.
It then suffices
(since both $(\cX^{\Gal}_D)_{\red}$ and $(\cX_{\red})_{\Fbar_p}$ are finitely presented
algebraic stacks over $\Fbar_p$) to
show that this image is constructible.  
This follows from Chevalley's constructibility theorem, once we write
$(\cX^{\Gal}_D)_{\red}$ as the union of a finite number 
of families of extensions, using the ideas of
Lecture~\ref{sec: Herr complex and Ext groups}.
\end{proof}

%\mar{TG: this might be a good place to add CWE discussion?}

In Galois deformation theory, the passage from Galois representations to
 pseudorepresentations is related to the ideas that we're discussing
 (and indeed Wang-Erickson's arguments use the theory of pseudorepresentations).
For example, one subtlety in the theory of pseudorepresentations
is that traces don't know about extension classes, so the passage
to pseudorepresentations  factors through semisimplification; indeed,
two representations $\rho_1,\rho_2: G_K \to \GL_d(\Fbar_p)$ have the
same underlying pseudorepresentation precisely if $\rho_1^{\semis}
=\rho_2^{\semis}$.   The  $d$-dimensional pseudorepresentations
of $G_K$ over $\Fbar_p$ are then in  bijection with the semisimple
$\rho:G_K \to \GL_d(\Fbar_p)$.
These phenomena mirror corresponding phenomena related to the points
of $\cX_{d,\red}$.  
Namely, as we have just discussed,
the closed points of $\cX_{d,\red}(\Fpbar)$ correspond to semisimple representations, and specialization to a closed point
 corresponds to semisimplification. That is, if $\rho$  is any element
of~$\cX_{d,\red}(\Fbar_p)$, then the closure $\barr{\set{\rho}}$ has a unique closed point, which is $\rho^{\semis}$.

Thus if we could find a moduli space map $\cX \to  X$  (whatever its precise
meaning in the context  of formal algebraic stacks), we might expect $X$
to be a moduli space of pseudorepresentations.   This is one of our main
motivations for studying associated moduli spaces in our context.

The fact that any point of
$\cX_{d,\red}(\Fbar_p)$ specializes to a  unique  closed point
implies that the relation
$x \sim y$ considered above {\em is}  an equivalence relation.
Since closed points of $\cX_{d,\red}$ correspond to  semisimple  representations,
we also see that the
stabilizers of closed points are reductive.
These are two of the properties that  are necessary to admit an adequate
moduli space.  Nevertheless, we will see that  $\cX_{d,\red}$ does not
admit an adequate moduli space when $d  > 1$, and so we seem to be outside the scope
of any generally developed theory.

\begin{example}
Consider $[(\bA^2 \setminus \set{0})/\bG_m]$, with  the  $\Gm$-action given via
$t\cdot (x,y) = (x, ty)$.  Then the open orbits in each of the vertical lines
satisfying $y \neq 0$
specialize to the intersection of the line with the horizontal axis, away from the origin. On the other hand, the vertical line $y = 0$ is a single closed orbit.
There's an obvious map $[(\bA^2\setminus \set{0})/\bG_m]\to \bA^1$,
given by $(x,y) \mapsto x$, 
%which is projection to the horizontal axis,
which {\em is} the associated moduli space map
(in that it satisfies Alper's properties~(1), (2), (3) above; in fact,
it is even initial for morphisms to arbitrary algebraic spaces).
%\mar{ME: ME should probably double-check  this  last claim.}
%\mar{TG: in what sense? initial among maps to schemes?}
This map is \textit{not} adequate, for instance because it's not
closed. %and in particular not a  universal submersion.
\end{example}

\begin{rem}
  The preceding example illustrates that in general, the quotient of a
  {\em quasi-affine} scheme by a reductive group can be much nastier
  than the quotient of an affine scheme by a reductive group.
\end{rem}
As we saw in Lecture~\ref{sec:intro}, stacks similar to that in the previous
example appear as moduli stacks of two-dimensional Fontaine--Laffaille modules,
and so also appear as irreducible  components in $\cX_{2,\red}$.
Thus $\cX_{2,\red}$ (and, more generally, $\cX_{d,\red}$ for $d > 1$) will
not admit  an adequate moduli space.

Nevertheless, we anticipate the following result.
%\mar{ME: cite [DEG]}

\begin{expectedtheorem}[\cite{DottoEG}]\label{thm: expected chain of P1}
Assume that $p\ge 5$. Let $K = \Qp$, and fix a character  $\psi: G_{\Q_p} \to \Zbar_p^{\times}$.
%\mar{ME: Do we have coefficients in our set-up? TG: I think this is
%  fine, we've done this above, I'd be inclined to just use $\Zpbar$
%  here though. ME: Done.}
Then $\cX_2^{\det = \psi}$ admits an associated formal algebraic moduli space $X$
{\em (}in an appropriately understood sense{\em )},
with $X_{\red}$ being a  certain chain  of $\bP^1$s.
The  points of $X(\Fbar_p)$  correspond to $2$-dimensional
pseudorepresentations  of $G_{\Q_p}$ over  $\Fbar_p$ with  determinant~$\psi$
{\em (}equivalently, to semi-simple $\rho: G_{\Q_p}  \to \GL_2(\Fbar_p)$
with $\det \rho = \psi${\em )}.
%for which the induced map $(\cX_2^{\det = \psi})_{\red} \to X_{\red}$
%which is a finite chain of $\bP^1s$.
The complete local ring of  $X$  at  one  of its closed points is naturally
identified with
%Its complete local rings are
the corresponding  pseudodeformation ring.
\end{expectedtheorem}

%We expect that $\cX_2$ admits an associated formal moduli space which would be a thickening of this.

\begin{example}
The preceding expected theorem fits nicely with the deformation
theory  of  $2$-dimensional crystalline  representations of~$G_{\Q_p}$.
%How does this picture relate to crystalline deformation rings?
The $\Zbar_p$-points of one of the crystalline stacks we have constructed
correspond to Galois-invariant  lattices in crystalline representations;
but when we  pass to pseudodeformations,  the different possible 
lattices all share a common image.

In particular if we fix the Hodge--Tate weights to be~$0,k-1$
with~$k\ge 2$, the images of these points in the moduli space $X$ will be a family of
crystalline representations $V_{k,a_p}$, parameterised by the
crystalline Frobenius~$a_p$ % \mar{TG: this notation hasn't
  % been defined}
with $|a_p| \leq 1$, i.e.\ parameterized by the rigid analytic closed
unit disk. (That there is just a single parameter~$a_p$ is a simple
calculation with weakly admissible modules.)  The image of the
crystalline moduli stack itself will then be a certain formal model of
the closed unit disk, whose underlying reduced scheme embeds into the
chain of $\bP^1$s given by~$X_{\red}$.
%
%. So the rigid generic fiber of the moduli space is just this disk, and for the stack you get a formal model, lands in the $\bP^1$s? The easiest model is $\wh{\bA^1}_{\Z_p}$, but if you blow up in the special fiber, you get $\bP^1$s.
Now the simplest  formal model of the closed unit disk is
$\wh{\bA^1}_{\Z_p}$,
whose special fibre $\bA^1$ certainly embeds into~$\bP^1$.
More general formal models  are obtained by blowing up points in
$\wh{\bA^1}_{\Z_p}$, and the special fibre then contains
additional~$\bP^1$s. In the case that~$k\le 2p+1$, this can be seen
explicitly by looking at the formulae for the reductions of
crystalline representations in~\cite[Thm.\ 5.2.1]{MR2906353}.

In fact, the results of~\cite[Thm.\ 5.2.1]{MR2906353} were an
important clue that our stacks should exist. Indeed, Mark Kisin
suggested in around 2004 that some kind of non-formal moduli (i.e.\
with $\rhobar$ varying) spaces of local crystalline Galois
representations should exist, motivated by the conjectures
of~\cite{BreuilGL2II} and the results
of~\cite{BergerBreuilunpublished}.
\end{example}

%\mar{ME: Mention the ``poids moyen''  computations, Kisin's observation \dots} 
\subsection{Bernstein Centers}

When we investigate the relationship between automorphic forms
and Galois representations,\footnote{This is a huge topic, of course,
and we introduce it here only to provide motivation,  and in the
most abbreviated way possible!}
the basic link between the two concepts is that the Hecke eigenvalues of an
automorphic eigenform at primes not dividing the level should match 
with traces of Frobenius  at primes that  are unramified in the
associated  Galois representation.

At primes that divide the level, there is a more subtle  way to extract
eigenvalues from automorphic forms/automorphic representations.  Namely,
if $K$ is (as always) a finite extension of $\Q_p$,
then the abelian category of smooth representations of $\GL_d(K)$
on $\C$-vector spaces admits a commutative ring of endomorphisms,
called the {\em Bernstein centre} for~$\GL_d(K)$.   On any irreducible
representation, it will act via scalars (because of Schur's Lemma).  
On unramified representations we recover the usual Hecke eigenvalues.

If $f$ is an  automorphic Hecke eigenform, for $\GL_d$ over some
number field~$F$, generating an automorphic
representation $\pi$, and if $K$ arises as the completion of $F$ at
some prime $v$ above~$p$,  then $\pi$ has a local factor $\pi_v$ at~$v$,
which is a representation of~$\GL_d(K)$.  This gives rise to a system
of eigenvalues for the Bernstein centre.   If $v$ does not divide
the level of~$f$, this is just the usual collection of Hecke eigenvalues of~$f$.

On the other hand, if $\rho:G_F \to \GL_d(\Qbar_{\ell})$
is the global $\ell$-adic Galois representation (provably in some cases,
conjecturally in others) associated to~$f$, with $\ell$ chosen so
that $\ell \neq p$, then we can consider the
restriction of $\rho$ to a decomposition group at~$v$,
giving rise to a representation $\rho_v: G_K \to \GL_d(\Qbar_{\ell})$.
We know in some cases, and anticipate in general,
that $\rho_v$ and $\pi_v$ are related via the local Langlands correspondence
(once we fix an isomorphism $\imath: \C \iso \Qbar_{\ell}$).

On the other hand,
the pseudorepresentation associated to $\rho_v$ encodes the traces of $\rho_v$
on all the elements of $G_K$, and so we have two collections of numbers
associated to the local-at-$v$ aspects of our situation:   the eigenvalues
of the Bernstein centre acting on~$\pi_v$, and the pseudocharacter of~$\rho_v$.
It turns out (as a kind of numerical shadow of the local Langlands
correspondence) that these numbers also determine one another.

In fact, there is even an integral version of this statement.  Namely,
we can consider smooth representations of $\GL_d(K)$ on $\Z_{\ell}$-modules,
and form the corresponding $\Z_{\ell}$-Bernstein centre. 
And we can consider the moduli stacks
$\cV_Q = [\Spec A_Q/\GL_d]$
parameterizing local-at-$v$
Weil--Deligne representations
over $\Z_{\ell}$-algebras,
described in Lecture~\ref{sec:intro}.
This stack is the quotient of an affine scheme by
a reductive group, and so 
the associated moduli space is just
$\Spec A_Q^{\GL_d}$.

%It makes sense to connect the Hecke side to pseudocharacters, since you get numbers from Hecke eigenvalues. So for $\ell \neq p$ we had $\cV_Q = [\Spec A_Q/\GL_d]$ and the associated moduli space is just $\Spec A_Q^{\GL_d}$.
%

The following theorem of Helm--Moss encapsulates the  manner in which numerical data extracted from 
Galois representations matches with eigenvalues of the Bernstein centre.
\begin{theorem}[\cite{MR3867634}]
\label{thm:HM}
$\varprojlim_Q A_Q^{\GL_d}$ is the $\Z_\ell$-Bernstein center for $\GL_d(K)$.
\end{theorem}

\begin{example}
If $g$ is an element of the Weil--Deligne group $\WD_K$, then ``trace of $\rho_v(g)$'' 
gives an element of each  
$\Spec A_Q^{\GL_d}$, and thus an element $h_{\sigma}$
of the Bernstein centre.  The value of
this  function at a particular~$\rho_v$ --- which is just the trace of
the particular matrix $\rho_v(g)$ --- will then correspond to the
eigenvalue of $h_{\sigma}$ on the representation $\pi_v$ associated 
to $\rho_v$ via local Langlands.  
\end{example}

%This is supposed to encompass the step in automorphy lifting theorems where you match characters on the Galois and automorphic sides.

We can then ask if there is an $\ell = p$ analogue of this  result.
Even in the case of $\GL_2(\Q_p)$, 
Expected Theorem~\ref{thm: expected chain of P1}
shows the associated formal moduli space to the stack $(\cX_2)^{\det = \psi}$
is {\em not} formally affine, so there is no obvious ring appearing
on the Galois/$(\varphi,\Gamma)$-module side to compare with a Bernstein centre.
However, we have the following result.

\begin{expectedtheorem}[\cite{DottoEG}]Assume that $p\ge 5$.
If $\cA$ is the abelian category of smooth
$\GL_2(\Qp)$-representations on $\Z_p$-modules which are locally
$p$-power torsion, with central character equal to $\psi \varepsilon$,
then $\cA$ localizes to a stack of categories over
the formal scheme~$X$ of
Expected Theorem~{\em \ref{thm: expected chain of P1}}.
\end{expectedtheorem}

Given this,
we can form a {\em sheaf} of Bernstein centres of $\cA$ over~$X$.
Of course, we also have the structure sheaf $\cO_X$.
Our expectation then is that these two sheaves can be identified.
%$\cO_{X}$ can be identified with the sheaf of Bernstein centers of $\cA$.
Just as in the $\ell \neq p$ case, this identification will be
mediated via the  local Langlands correspondence --- but now we will
have to use the {\em $p$-adic local Langlands correspondence} for~$\GL_2(\Q_p)$.

%So what's the identification?
Indeed,
if
$D \to \cX_2^{\det \chi}$ is the universal $(\varphi,\Gamma)$-module,
then $(D \boxtimes \bP^1)/(D^\natural \boxtimes \bP^1)$
(as defined by Colmez \cite{MR2642409})
should be a quasi-coherent sheaf of $\GL_2(\Qp)$-representations over $\cX_2$.
% (mod finite dimensional representations).
Then the identification between $\cO_X$ and  the sheaf of Bernstein
centres of $\cA$ will be determined
by saying that the two sheaves of rings act in the same way on this family.

This is all strongly related to, and uses, Pa{\v{s}}k{\=u}nas'
work~\cite{MR3150248}.

\bigskip

\subsection{A conjectural $p$-adic local Langlands
  correspondence}\label{subsec: conjectural p adic LL}%\mar{TG: this will probably contain nonsense}
We return to the case of general~$K$ and $d$,
and end our lectures with a vague and speculative discussion of a possible $p$-adic
local Langlands correspondence, analogous to the conjectures made in
the $\ell\ne p$ setting
in~\cite{benzvi2020coherent,hellmann2020derived,zhu2020coherent}. We
also explain briefly how such a conjecture could explain the geometric
Breuil--M\'ezard conjecture of the previous lecture.

Each of the papers
\cite{benzvi2020coherent,hellmann2020derived,zhu2020coherent} 
proposes a certain conjectural enhancement of Theorem~\ref{thm:HM}.
We state a rough  form of this  conjecture here.

\begin{roughconj}
\label{conj:localization}
The category of $\GL_d(K)$-representations
on $\Z_{\ell}$-modules admits a fully faithful embedding
into the category of quasi-coherent sheaves
%\mar{TG: quasi or ind-coh might be better}
on the Ind-algebraic stack $\varinjlim_Q \cV_Q,$
compatibly with the identification of
Theorem~{\em \ref{thm:HM}}.
\end{roughconj}

We have omitted all technical details  and caveats from this statement
(in particular, it should be stated at a derived level).  The key
intuition, though, is that we are conjecturing that
we can localize $\GL_d(K)$-representations
in some fashion as sheaves of $\cO$-modules
over  the moduli stack of rank~$d$ Weil--Deligne
representations.   And although the action of the Bernstein centre
is not sufficient to  separate all irreducible representations
(e.g.\ any two such representations that admit a non-trivial extension,
such as the trivial and Steinberg representations  of $\GL_2(K)$,
will have the same system of eigenvalues under the Bernstein centre),
this more refined
localization process should be able to distinguish non-isomorphic irreducibles.

%We return to the case of general~$K$ and$d$.

There is another important facet of the (conjectural) fully faithful embedding
of Conjecture~\ref{conj:localization} which
we want to explain, if  only in somewhat general and vague  terms.   
The essential point is that
there is {\em another} method of obtaining coherent sheaves
out of  $\GL_d(K)$-representations,
not on the stacks  $\cV_Q$ themselves, but on their versal rings at closed
points, which is to say, on formal deformation rings of  Galois representations.
This is the method of {\em Taylor--Wiles--Kisin} patching (which already
made an appearance in the previous lecture). 

Briefly,  
if we choose some global context related to $\GL_d(K)$ (e.g.\ a modular curve or Shimura
curve for~$\GL_2$, or a certain unitary Shimura variety more generally), 
and an irreducible\footnote{In  fact $r$ should satisfy a stronger condition than
irreducibility, but we will suppress all such technical details.}
 automorphic global Galois representation $r: G_F  \to\GL_d(\Fbar_p)$
(here $F$ is a totally  real or CM field of which $K$ is a certain completion,
and which is  related to the Shimura variety at hand),
and an $\ell$-adic  representation $\sigma$ of $\GL_d(\cO_K)$,
then we are able to ``patch'' appropriate Hecke localizations of the cohomology
of whichever Shimura variety is in play so as to obtain a coherent sheaf,
traditionally denoted  by~$M_{\infty}(\sigma)$,
which lives
over the formal neighbourhood\footnote{We are here ignoring the ``patching  variables''
that might have been  introduced to  effect  the patching process; these  {\em should}
be  ignorable,  although this is not known in general.}
of $r_{|G_K}$.   

The relationship between patching and Conjecture~\ref{conj:localization} 
can be stated very roughly  as follows:\footnote{For more  precise statements,
the reader can look at~\cite[Conj.~4.6.3]{zhu2020coherent}, which
treats the  function field analogue of the situation we are describing here ---
although the connection to patching is not made precise there.
The case of number fields should be discussed carefully (just at
the level of conjecture, to be  sure!) in  forthcoming work of Xinwen Zhu 
and M.E.}
In sufficiently good situations,\footnote{E.g.\ if $p$ is unramified in~$F$,
and if we choose the level and coefficients of our cohomology carefully enough.}
$M_{\infty}(\sigma)$
%this coherent sheaf
should be equal to the formal completion at $r_{|G_K}$
of the coherent sheaf associated to $\cInd_{\GL_d(\cO_K)}^{\GL_d(K)}\sigma$
by Conjecture~\ref{conj:localization}.
In general, the patched module $M_{\infty}(\sigma)$ will
be  related to this latter coherent sheaf,  but we will have to take into
account level structure at completions of $F$ other than~$K$,
the most subtle being the level structure  at the primes above~$\ell$;
and this latter
case amounts to studying patching for $p$-adic $\GL_n(\cO_K)$-representations.
%If we interchange  the  roles  of $\ell$ and~$p$, then  we can reexpress
%this most subtle case as being the case when $\ell = p$.

%Indeed,
The patching described above, in the $\ell \neq p$ case (see e.g.\ \cite{MR3769675,manning2020patching,manning2020iharas}),
%\mar{ME: Cite Shotton and Manning as examples?}
is less commonly studied than the case of $\ell = p$, and 
we now return to the latter %$\ell = p$
case.
In this context, 
the appropriate stacks over which  to state an analogue 
of Conjecture~\ref{conj:localization} are
%and to considering
our moduli stacks $\cX_d$ parameterizing \'etale $(\varphi,\Gamma)$-modules,
in place of the moduli stacks of Weil--Deligne representations.

%By analogy with  Conjecture~\ref{conj:localization},
Our hope, then, is that there is
a canonically defined quasi-coherent sheaf~$\Pi$ % ~$\cM$
of $\GL_d(K)$-representations
on~$\cX_d$ which in particular enjoys the properties that it satisfies
local-global compatibility for completed cohomology of locally
symmetric spaces (in the sense of being compatible with Taylor--Wiles--Kisin
patching, as carried  out in e.g.~\cite{Gpatch}),
and explains the Breuil--M\'ezard conjecture and the
weight part of Serre's conjecture in a sense that we will explain
shortly.

The analogy with  Conjecture~\ref{conj:localization} comes by
considering the  functor
%We also hope that the functor
%$\pi\mapsto\RHom_{\GL_d(K)}(\pi,\cM)$
$\pi\mapsto\RHom_{\GL_d(K)}(\pi,\Pi)$
from the derived category of smooth representations of~$\GL_d(K)$ on
$\Zp$-algebras to the derived category of quasi-coherent sheaves
on~$\cX_d$; this functor should be fully faithful. This being the case, one could in
particular localize the former derived category over~$\cX_d$,
generalizing the case~$d=2$, $K=\Qp$ explained above. (However, in
this more general context, there does not seem to be any obvious
analogue of the moduli space for~$\cX_d$ that we considered above, and
it seems likely that one really has to consider this localization %localizing
as taking place over the stack.)

In the case that $d=2$ and~$K=\Qp$, we explicitly construct a
candidate quasi-coherent sheaf ~$\Pi$ in~\cite{DottoEG}, by
generalising Colmez's construction from \cite{MR2642409} of the
representation he denotes
$(D\boxtimes\Pone)/(D^\natural\boxtimes\Pone)$. Note that after pulling back to
versal rings, we know that this sheaf realises the $p$-adic local Langlands
correspondence for~$\GL_2(\Qp)$ by results of Colmez and
Pa{\v{s}}k{\=u}nas~\cite{MR3150248} (see also~\cite{MR3732208} for
a perspective related to patching).
In particular, a construction of $\Pi$ in general would %in particular
generalize the $p$-adic local Langlands correspondence from
the case of~$\GL_2(\Q_p)$ to the general case of~$\GL_d(K)$.

%The expected connection with the Breuil--M\'ezard and Serre weight conjectures is as follows.
If~$\sigma$ is a representation
of~$\GL_n(\cO_K)$ on a finitely generated $\Zp$-module, 
then we can  apply the functor
$\RHom_{\GL_d(K)}(\text{--},\Pi)$
to $\pi:= \cInd_{\GL_d(\cO_K)}^{\GL_d(K)}\sigma^{\vee}$.  (The Pontrjagin dual
is included just so as to make the functor covariant.)  We anticipate that this
$\RHom$ should be supported in degree~$0$, and that the resulting
	sheaf $\cM(\sigma):=\Hom_{\GL_d(\Zp)}(\sigma^{\vee},\Pi)$
will  be coherent. (Just as in the $\ell \neq p$ case discussed
above, the formal completions of $\cM(\sigma)$ at $\Fbar_p$-points
of $\cX_{d,\red}$ should give the usual patched modules~$M_{\infty}(\sigma)$.)
%\mar{TG: say why this should be coherent? Say anything about duals?}
These (conjectural) coherent sheaves allow us to connect  the present discussion to
the Breuil--M\'ezard and Serre weight conjectures. 

Indeed, we expect that if~$\underline{k}$ is a Serre weight (regarded as
usual as a representation of~$\GL_d(\cO_K)$ via inflation
from~$\GL_d(k)$), then the cycle~$Z_{\underline{k}}$ considered in
Lecture~\ref{sec: geometric BM} is simply the support of
$\cM(F_{\underline{k}})$. Furthermore, we expect that if~$\underline{\lambda}$ is
a regular Hodge type, $\tau$ is an inertial type,
and~$\sigma^{\circ}(\lambda,\tau)$ % \mar{TG: actually we only defined
  % the mod $p$ version of this}
is the
representation considered in Lecture~\ref{sec: geometric BM}, then the
support of~$\cM(\sigma^{\circ}(\lambda,\tau))$ is
exactly~$\cX_d^{\crys,\underline{\lambda},\tau}$. This being the case,
Conjecture~\ref{conj: geometric BM} is an immediate consequence of the
additivity of supports, while the expectation that the
cycles~$Z_{\underline{k}}$ encode the weight part of Serre's
conjecture should be a consequence of a local-global compatibility for
completed cohomology of Shimura varieties (or more generally locally
symmetric spaces), generalising the results of~\cite{emerton2010local,
MR3732208}
for the modular curve.

%   $K=\Qp$. Let $\sigma$ be an irreducible $\Fpbar$-representation
%   of~$\GL_d(\Fp)$.

This may all seem a speculation too far, but we note that while there
is in general no candidate for~$\Pi$,
the construction of~\cite{Gpatch} does produce such 
a candidate %does exist
after
pulling back to the versal deformation rings;
%namely the Taylor--Wiles--Kisin patched modules of~\cite{Gpatch};
and the conjectural explanation
above for the Breuil--M\'ezard conjecture describes exactly how the results
of~\cite{KisinFM,geekisin} are proved (in the setting of the
numerical conjecture, over a deformation ring).

% For   $\underline{\lambda}=(\lambda_1>\dots> \lambda_d)$, set
% $\cM(\lambdau):=\Hom_{\GL_d(\Zp)}(\cM,\pi_{\lambdau}^\vee)^\vee$, where
% $\pi_{\lambdau}=$ irreducible algebraic $\GL_d(\Zp)$-representation of
% highest weight
% $(\lambda_1-(d-1),\dots,\lambda_{d-1}-1,\lambda_d)$. 

% Expectations imply: $\cZ(\cM(\lambdau))=\cZ(\lambdau)=\cZ((\cX_d^{\crys,\lambdau})_{\Fp})$. 

% Then we have the \emph{geometric Breuil--M\'ezard conjecture} $\cZ(\lambdau)=\sum_\sigma
% n_{\sigma}(\lambdau)\cZ(\sigma)$, where $n_\sigma(\lambdau)=$
% multiplicity of~$\sigma$ in~$\pi_{\lambdau}\otimes_{\Zp}\Fp$.

% No construction is known
%   of~$\cM$ other than for~$\GL_1$ or~$\GL_2(\Qp)$. 

% However there is a candidate after pulling back to the versal ring at
% a fixed $\rhobar:G_K\to\GL_d(\Fpbar)$: the \emph{patched module}
% $M_\infty$ of
% Caraiani--Emerton--G.--Geraghty--Pa\v{s}k\=unas--Shin.

% Construction via globalisation and Taylor--Wiles patching of
% cohomology of unitary Shimura varieties.

% The pullback of the geometric Breuil--M\'ezard conjecture
% $\cZ(\lambdau)=\sum_\sigma n_{\sigma}(\lambdau)\cZ(\sigma)$ to the
% versal ring is equivalent to automorphy lifting theorems (Kisin).

% pull back to patched modules

\emergencystretch=3em
\bibliographystyle{amsalpha}
\bibliography{universalBM}
\end{document}